\documentclass[11pt]{article}

\usepackage{amsmath}
\usepackage{amsthm}
\usepackage{amsfonts}
\usepackage{verbatim}
\usepackage{amssymb}
\usepackage{fancyhdr}
\usepackage{graphicx}
\usepackage{float}
\usepackage{pifont}
\usepackage{extarrows}
\usepackage{mathabx}
\usepackage{changepage}
\usepackage{adjustbox}
\usepackage{mdframed}
 \usepackage{graphicx,rotating,booktabs}
\usepackage[verbose]{placeins}
\usepackage{enumitem}
\usepackage{titlesec}
\usepackage{ltablex}
\titleformat*{\section}{\center\Large\bfseries}
\titlespacing{\section}{0pt}{0pt}{\parskip}
\usepackage{caption}

\usepackage[english]{babel}
\usepackage[top=30mm,right=30mm,bottom=30mm,left=30mm]{geometry}

\usepackage{mathrsfs}
\usepackage{longtable}
\usepackage{tabularx}
    \usepackage{ltablex}

\usepackage{chngcntr}
\allowdisplaybreaks

\newtheorem*{bigtheorem}{Theorem}
\newtheorem{theorem}{Theorem}[section]
\newtheorem{lemma}[theorem]{Lemma}
\newtheorem{proposition}[theorem]{Proposition}

\begin{document}

\title{\bf On the finite simple images of free products of finite groups}
\author {Carlisle S. H. King\\
\it{Imperial College} \\
\it{London}\\
\it{United Kingdom}\\
\it{SW7 2AZ}\\
carlisle.king10@ic.ac.uk}
\date{}
\maketitle

\section*{\center{Abstract}}
Given nontrivial finite groups $A$ and $B$, not both of order 2, we prove that every finite simple group of sufficiently large rank is an image of the free product $A \ast B$. To show this, we prove that every finite simple group of sufficiently large rank is generated by a pair of subgroups isomorphic to $A$ and $B$. This proves a conjecture of Tamburini and Wilson.

\section*{\center{1. Introduction}}
\stepcounter{section}

\par Given a finitely presented group $\Gamma$, a natural question to ask is whether $\Gamma$ has a nontrivial finite image. If so, what are the finite images of $\Gamma$? There is a large literature on the study of finite images of finitely presented groups such as Fuchsian groups, hyperbolic groups, and other families arising geometrically. 
\par A natural refinement of this question is to determine the finite simple images of $\Gamma$. For certain groups, their finite simple images have been studied extensively. The finite simple images of Fuchsian groups $\Gamma$ are studied in many papers, culminating with Everitt who showed in \cite{Eve} that if $\Gamma$ is oriented then $\Gamma$ has all but finitely many alternating groups as images, and with Liebeck and Shalev who showed in \cite{LiSh4} that if $\Gamma$ is has genus $g \ge 2$ then $\Gamma$ has all finite simple groups as images up to finitely many exceptions. The same authors conjecture that for any Fuchsian group $\Gamma$ there exists an integer $N(\Gamma)$ such that all finite simple groups of rank at least $N(\Gamma)$ are images of $\Gamma$. We note that for Fuchsian groups $\Gamma$ of genus 0 or 1 the condition on the rank is necessary, as there can be infinitely many finite simple groups of small rank which are not images of $\Gamma$ (see \cite{Co}). Triangle groups are examples of Fuchsian groups of genus 0 that have received much attention, and there are strong results on the finite simple images of these groups (see for example \cite{LaLuMa}).
\par The finite simple images of the modular group $PSL_2(\mathbb{Z}) \cong C_2 \ast C_3$ are precisely the $(2,3)$-generated finite simple groups, on which there is a large literature (see for example \cite{LiSh1}). More generally, for primes $r$ and $s$ not both equal to 2, the finite simple images of $C_r \ast C_s$ are studied in \cite{LiSh2}: it is shown that there exists a number $N=N(r,s)$ such that all finite simple classical groups of rank at least $N$ are images of $C_r \ast C_s$. 

\par In this paper we study finite simple images of free products $\Gamma=A \ast B$ of arbitrary nontrivial finite groups $A$ and $B$. In \cite{TaWi2}, Tamburini and Wilson conjectured that for any nontrivial finite groups $A, B$, not both of order 2, there exists an integer $N(A,B)$ such that all finite simple groups of rank at least $N(A,B)$ are images of $\Gamma$. The same authors prove in \cite[Theorem~1.2]{TaWi1} that their conjecture holds for alternating groups. If $A$ and $B$ are nontrivial finite groups with $| A | | B | \ge 12$, then by \cite[Theorem~1]{TaWi2} the projective special linear groups $PSL_n(q)$ with $n \ge N(A,B)= 4 | A | | B | + 12$ are images of $\Gamma$. If $A$ and $B$ are not both 2-groups, then by \cite[Theorem~2.3]{LiSh3} the conjecture holds for finite simple classical groups. In this paper, we complete the proof of this conjecture.
\par We note that a condition on the rank of the classical groups is necessary to ensure the existence of subgroups isomorphic to $A$ and $B$.
\begin{bigtheorem} Let $A$ and $B$ be nontrivial finite groups, not both of order 2. There exists a natural number $N=N(A,B)$ such that every alternating group of degree greater than $N$ and every finite simple classical group of rank greater than $N$ is an image of $A \ast B$.
\end{bigtheorem}
\par To do this, we will show that every finite simple group of sufficiently large rank is generated by a pair of subgroups isomorphic to $A$ and $B$.
\par By \cite[Theorem~2]{TaWi1} the theorem holds for alternating groups, and by \cite[Theorem~1.2]{LiSh3} the theorem holds for classical groups if $A$ and $B$ are not both 2-groups. It therefore suffices to prove the result for classical groups with $A$ and $B$ both 2-groups. This result follows from our two main theorems (see Theorems 3.1 and 3.2 below). The proof has a probabilistic flavour: we construct specific subgroups $A$ and $B$ of a classical group $G$ of sufficiently large rank, and show that some pair of conjugates of these subgroups generates $G$ using probabilistic methods. 
\par We now discuss the structure of the paper. By the above paragraph, we can assume $A$ and $B$ are both 2-groups. Assume $|B| > 2$, so $B$ contains a subgroup isomorphic to either $C_4$ or $C_2 \times C_2$. Section 2 introduces the notation and some preliminary results used throughout this paper. It also introduces the ``almost-free embedding'' which we use to identify $A$ and $B$ with subgroups of $G$. In Section 3 we present Theorems \ref{bigthm1} and \ref{bigthm2}, which correspond to the cases where $B$ contains a subgroup isomorphic to $C_4$ and $C_2 \times C_2$ respectively, and from these theorems we deduce the main theorem. These theorems involve counting subgroups isomorphic to $C_4$ and $C_2 \times C_2$ in finite simple classical groups $G$ and their maximal subgroups. The remaining sections are dedicated to showing that $G$ can be generated by almost-free copies of $A$ and $B$.

\paragraph{\it Acknowledgements.} This paper is part of work towards a PhD degree under the supervision of Martin Liebeck, and the author would like to thank him for bringing this problem to his attention and for his guidance throughout. The author is also grateful for the financial support from EPSRC.

\section*{\center{2. Notation and preliminary results}}
\stepcounter{section}

\par In this section we introduce the notation used throughout the paper. We then discuss some preliminary results on algebraic groups and automorphisms and maximal subgroups of classical groups.

\subsection{Notation}

\par Throughout, $G$ will be a finite almost simple classical group with natural module of dimension $n$ over $\mathbb{F}_{q^\delta}$ of characteristic $p$, where $\delta=2$ if $G$ is unitary and $\delta=1$ otherwise, and we let $\operatorname{soc}(G)=S$ (we will often restrict to the case where $G=S$). Let $H$ be the simple adjoint algebraic group over $l=\overline{\mathbb{F}_p}$ such that $(H^F)'=S$ for a Steinberg endomorphism $F$.
\par  In the case where $S$ is orthogonal and $n$ is unspecified, it is understood that $P\Omega_n(q)$ may denote any one of the three families of finite simple orthogonal groups, and similarly for $O_n(q), SO_n(q), \dots$ etc.
\par For a group $M$ we write $M \operatorname{max} G$ if $M$ is a maximal subgroup of $G$. For $x \in G$, we write $x^G = \{ x^g: g \in G \}$ for the conjugacy class of $x$, and for a subgroup $L \le G$ let $L^G =   \{ L^g: g \in G \}$. If $H_1$ is any finite classical group with $Z=Z(H_1)$ and $H_2 \le H_1$, then denote by $\overline{H_1}=H_1/Z$ the image of $H_1$ modulo scalars, and similarly $\overline{H_2}=H_2/(H_2 \cap Z)$.
\par Let $M_{n_1, n_2}(q)$ denote the additive group of $n_1 \times n_2$ matrices with entries in $\mathbb{F}_q$. For $B=(b_{ij}) \in M_{n_1,n_2}(q^2)$, we write $B^q=(b_{ij}^q)$. Denote by $J_i \in GL_i(q)$ the unipotent Jordan block of size $i$.
\par For finite groups $H_1, H_2$, let $I_{s}(H_1)$ denote the set of elements of $H_1$ of order $s$, and $I_{s,H_2}(H_1)$ the subset of $I_s(H_1)$ consisting of all elements of order $s$ lying in some subgroup of $H_1$ isomorphic to $H_2$. We let $i_s(H_1)=|I_s(H_1)|$ and $i_{s,H_2}(H_1) = |I_{s,H_2}(H_1)|$. Let $I_{2 \times 2}(H_1)$ denote the set of subgroups of $H_1$ isomorphic to $C_2 \times C_2$, and similarly define $I_{2 \times 2, H_2}(H_1), i_{2 \times 2}(H_1)$ and $i_{2 \times 2, H_2}(H_1)$.
\par For two functions $f,g$ of variables $x_i, 1 \le i \le m$, we write $f \sim g$ if there exist positive constants $c, c'$ such that
\begin{align*}
c < \frac{f}{g} < c' \hspace{0.5cm} \text{for all} \ x_i, 1 \le i \le m.
\end{align*}

\subsection{Algebraic groups and automorphisms of classical groups}

\par We build on the notation setup in \S~2.1. Recall the definitions of $G, S, H$ and $F$. For a maximal torus $T \le H$, let $\Phi$ denote the associated root system with a base $\Pi \subset \Phi$. Also, denote by $U_\alpha, \alpha \in \Phi$ the root subgroups of $H$, with associated algebraic group morphisms $u_\alpha:l \rightarrow U_\alpha$.
\par We consider elements of the abstract automorphism group $\operatorname{Aut}(H)$. For simplicity, we assume $n \ge 9$ throughout this discussion. As in \cite[Definition~1.15.5]{GoLySo}, we define subgroups $\Phi_H$ and $\Gamma_H$ of $\operatorname{Aut}(H)$ as follows: $\Phi_H$ contains elements $\phi_a, a \in \mathbb{Z}$ inducing a field automorphism on the root subgroup parameters, defined by 
\begin{equation*}
u_\alpha(c)^{\phi_a} = u_\alpha(c^{p^a}) \hspace{0.5cm} \text{for all} \ \alpha \in \Pi, c \in l; 
\end{equation*}
\noindent $\Gamma_H$ contains elements $\gamma_\rho$ for each symmetry $\rho$ of the Dynkin diagram of $H$, inducing an algebraic group automorphism such that \begin{equation*}
u_\alpha(c)^{\gamma_\rho}=u_{\rho(\alpha)}(c) \hspace{0.5cm} \text{for all} \ \alpha \in \Pi, c \in l.
\end{equation*}
\noindent By conjugating by an element of $\operatorname{Inn}(H)$, we can write $F=\gamma_\rho \phi_a$ for some $\gamma_\rho \in \Gamma_H, \phi_a \in \Phi_H$.
\par We consider the structure of $\operatorname{Aut}(S)$. A result of Steinberg states that every element of $\operatorname{Aut}(S)$ can be written as $idfg$, where $i \in \operatorname{Inn}(S)$ and $d, f$ and $g$ are ``diagonal", ``field" and ``graph" automorphisms respectively (see \cite[Theorem~2.5.1]{GoLySo} for more details). As in \cite[Definition~2.5.10]{GoLySo}, define the subgroups $\Phi$ and $\Gamma$ of $\operatorname{Aut}(S)$ as follows: define $\Phi$ as the restriction of $\Phi_H$ to $S$; if $S$ is untwisted, define $\Gamma$ as the restriction of $\Gamma_H$ to $S$; and if $S$ is twisted, let $\Gamma=1$. Write $\operatorname{Inndiag}(S)$ for the subgroup of $\operatorname{Aut}(S)$ generated by inner and diagonal automorphisms. Then by \cite[Lemma 2.5.8, Theorem 2.5.12]{GoLySo}, $H^F=\operatorname{Inndiag}(S), [\Phi, \Gamma]=1$, and $\operatorname{Aut}(S)$ is the split extension of $\operatorname{Inndiag}(S)$ and $\Phi\Gamma$. If $q=p^e$, then the structures of $\Phi$ and $\Gamma$ are as follows:
\begin{align*}
\Phi &= \left\{ 
  \begin{array}{l l}
C_e \hspace{0.3cm} &\text{if $S$ is untwisted,} \\[0cm]
C_{2e} &\text{if $S$ is twisted,} \\[0cm]
\end{array}
\right. \\
\Gamma &= \left\{ 
  \begin{array}{l l}
C_2  &\text{if $S=PSL_n(q), P\Omega^+_n(q)$,} \\[0cm]
1 &\text{otherwise.} 
\end{array}
\right. 
\end{align*}
\par We now introduce our conventions concerning ``field" , ``graph-field'' and ``graph" automorphisms, as in \cite[Definition~2.5.15]{GoLySo}:
\begin{itemize}
\item If $S$ is untwisted, then 

\vspace{-0.2cm}
\begin{itemize}
\item A field automorphism of $S$ is an $\operatorname{Aut}(S)$-conjugate of an element of $\Phi$; 
\vspace{-0.1cm} 
\item A graph-field automorphism of $S$ is an $\operatorname{Aut}(S)$-conjugate of an element of \\ $\Phi\Gamma \ \backslash \left( \Phi \cup \Gamma \right)$; 
\vspace{-0.1cm}
\item A graph automorphism of $S$ is an $\operatorname{Aut}(S)$-conjugate of an element of \\ $\Gamma \operatorname{Inndiag}(S) \backslash \operatorname{Inndiag}(S)$;
\end{itemize}
\item If $S$ is twisted, then 
\vspace{-0.2cm}
\begin{itemize}
\item A field automorphism of $S$ is an $\operatorname{Aut}(S)$-conjugate of an element of $\Phi$ of odd order; \vspace{-0.5cm}
\item A graph automorphism of $S$ is an element of $\operatorname{Aut}(S)$ whose image modulo $\operatorname{Inndiag}$ has even order; \vspace{-0.1cm}
\item There are no graph-field automorphisms of $S$.
\end{itemize}
\end{itemize}
\par For the remainder of this subsection, let $G=P\Delta(S)=\overline{\Delta(S)}$ be the projective similarity group of the same type as $S$, the standard notation for which is $G=PGL^\epsilon_n(q), PGSp_n(q)$ or $PGO^\epsilon_n(q)$ as $S=PSL^\epsilon_n(q), PSp_n(q), P\Omega_n^\epsilon(q)$.
\par Occasionally we will be interested in involutions $x \in \operatorname{Aut}(S)$. In particular, we will be concerned with $C_G(x)$. We briefly discuss the structure of $C_G(x)$ for $q$ odd. 
\par First suppose $x$ is a field or graph-field automorphism of $S$. If $y$ is another involutory automorphism of $S$ of the same type as $x$, then, by \cite[Proposition~4.9.1]{GoLySo}, $x$ and $y$ are $\operatorname{Inndiag}(S)$-conjugate, and $C_G(x)$ contains a subgroup $C_0$ listed in Table 1 below. 
\par For the remaining involutions $x \in \operatorname{Inndiag}(S)$ or $x$ a graph automorphism, $\operatorname{Inndiag}(S)$-classes of such involutions are listed in \cite[Theorem~4.5.1]{GoLySo}, and $C_G(x)$ contains a subgroup $C_0$ isomorphic to a group in Table 1. There are fewer than $cn$ $\operatorname{Inndiag}(S)$-classes of such involutions for some absolute constant $c$.
\par Note that in most cases, for $C_0$ listed in Table 1 we have $C_0/Z(C_0)$ isomorphic to a simple group $S_1$ or a direct product of simple groups $S_1 \times S_2$, with the following exceptions:
\begin{itemize}
\item $G=PGL_n(q), C_0 \cong \overline{SL_2(3) \times SL_{n-2}(3)}$; \vspace{-0.1cm}
\item $G=PGSp_n(q), C_0 \cong \overline{Sp_2(3) \times Sp_{n-2}(3)}$;  \vspace{-0.1cm}
\item $G=PGU_n(q), C_0 \cong \overline{SU_2(3) \times SU_{n-2}(3)}$;  \vspace{-0.1cm}
\item $G=PGO_n(q), C_0 \cong \overline{\Omega_3(3) \times \Omega_{n-3}(3)}$ or $\overline{\Omega^+_4(3) \times \Omega_{n-4}(3)}$. 
\end{itemize}

\renewcommand{\arraystretch}{1.2}
\setlength\LTleft{0.2cm}
\begin{longtable}{c c c c} 
\caption {Subgroups $C_0 \le C_G(x)$ for $x \in I_2(\operatorname{Aut}(S))$, $q$ odd} \\
\label{auttable}
$G$  & Type of $x$ & $C_0$ & Conditions\\[2pt] \hline \vspace{0.2cm}
$PGL_n(q)$ & inner-diagonal & $\overline{SL_m(q) \times  SL_{n-m}(q)}, 1 \le m \le \frac{n}{2}$  \\ 
& & $\overline{SL_{\frac{n}{2}}(q^2)}$ \\
& field & $PSL_n(q^\frac{1}{2})$  \\
& graph-field & $PSU_n(q^\frac{1}{2})$ \\
& graph & $PSp_n(q)$ \\
& & $P\Omega_n(q)$ \\ \hline
$PGU_n(q)$ & inner-diagonal & $\overline{SU_m(q) \times  SU_{n-m}(q)}, 1 \le m < \frac{n}{2}$  \\
& & $\overline{SL_{\frac{n}{2}}(q^2)}$ \\
& graph & $PSp_n(q)$ \\
& & $P\Omega^\epsilon_n(q)$ \\ \hline
$PGSp_n(q)$ & inner-diagonal & $\overline{Sp_m(q) \times  Sp_{n-m}(q)}, 1 \le m \le \frac{n}{2}$  \\
& & $PSp_{\frac{n}{2}}(q^2)$ \\
& & $\overline{SL_{\frac{n}{2}}(q)}$ \\
& & $\overline{SU_{\frac{n}{2}}(q)}$ \\
& field & $PSp_n(q^\frac{1}{2})$ \\ \hline
$PGO_n(q), n$ odd & inner-diagonal & $\overline{\Omega_m(q) \times  \Omega_{n-m}(q)}, 1 \le m \le \frac{n-1}{2}$  \\
& field & $P\Omega_n(q^\frac{1}{2})$ \\ \hline
$PGO^\epsilon_n(q), n$ even & inner-diagonal & $\overline{\Omega_m(q) \times  \Omega_{n-m}(q)}, 1 \le m \le \frac{n}{2}$  \\ 
& & $\overline{\Omega_{\frac{n}{2}}(q^2)}$ \\
& & $\overline{SL_{\frac{n}{2}}(q)}$ \\
& & $\overline{SU_{\frac{n}{2}}(q)}$ \\
& field & $P\Omega^+_n(q^\frac{1}{2})$ & $\epsilon=+$ \\
& graph-field  & $P\Omega^-_n(q^\frac{1}{2})$ & $\epsilon=+$ \\
& graph & $P\Omega_{2i-1}(q) \times P\Omega_{n-2i+1}(q), 1 \le i \le \frac{n}{2}$ \\
& & $P\Omega_{\frac{n}{2}}(q)^2$ & $\frac{n}{2}$ odd \\
& & $P\Omega_{\frac{n}{2}}(q^2)$ & $\frac{n}{2}$ odd \\ \hline

\end{longtable}
\setlength\LTleft{0cm}

\subsection{Maximal subgroups of finite simple classical groups}
In this subsection, we restrict to the case where $G=S$, a finite simple classical group with natural module $V$ of dimension $n$ over $\mathbb{F}_{q^\delta}$, where $\delta=2$ if $G$ is unitary and $\delta=1$ otherwise. A theorem of Aschbacher \cite{As} states that if $M$ is a maximal subgroup of $G$, then $M$ lies in a natural collection $\mathscr{C}_1, \dots, \mathscr{C}_8$, or $M \in \mathscr{S}$. Subgroups in $\mathscr{C}_i$ are described in detail in \cite{KlLi}, where the structure and number of conjugacy classes are given. Subgroups in $\mathscr{S}$ are almost simple groups which act absolutely irreducibly on the natural module $V$ of $G$. 
\par In Table \ref{maxsubgroups} below we give a description of subgroups $M$ in each Aschbacher class $\mathscr{C}_i, 1 \le i \le 8$. The precise structure of $M$ in each class can be found in \cite[\S~4]{KlLi}. We let $P_m$ denote a maximal parabolic subgroup of $G$ stabilizing a totally singular $m$-space for $1 \le m \le \frac{n}{2}$.

\renewcommand{\arraystretch}{1}
\setlength\LTleft{0cm}
\begin{longtable}{p{0.07\linewidth} p{0.9\linewidth}} 
\caption {Subgroups $M \operatorname{max} G$ in Aschbacher class $\mathscr{C}_i, 1 \le i \le 8$.} \label{maxsubgroups}\\
\hspace{0.03cm} Class  & \hspace{0.4\linewidth} Structure of $M$ \\[2pt] \hline \rule{0pt}{\normalbaselineskip}
\hspace{0.12cm} $\mathscr{C}_1$  & i) Non-parabolic: $M$ lies in the image modulo scalars of $GU_m(q) \times GU_{n-m}(q), Sp_m(q) \times Sp_{n-m}(q), O_m(q) \times O_{n-m}(q)$ or $Sp_{n-2}(q) \ (q \ \text{even})$ as $G=PSU_n(q), PSp_n(q), P\Omega_n(q), P\Omega_n(q)$ respectively. \\ \rule{0pt}{\normalbaselineskip}
& ii) Parabolic: $M=P_m$, $1 \le m \le \frac{n}{2}$. \\[2pt] \hline \rule{0pt}{\normalbaselineskip}
\hspace{0.12cm} $\mathscr{C}_2$ & i) $M$ lies in the image modulo scalars of $Cl_{m}(q) \wr S_t$, where $Cl_m(q)=GL^\epsilon_m(q), Sp_m(q),O_m(q), n=mt, t>1$ and $G=PSL^\epsilon_n(q), PSp_n(q), P\Omega_n(q)$ respectively; \\ \rule{0pt}{\normalbaselineskip}
& ii) $M$ lies in the image modulo scalars of $GL_\frac{n}{2}(q^2).2, GL_\frac{n}{2}(q).2, GL_\frac{n}{2}(q).2$ and $G=PSU_n(q),$ $PSp_n(q), P\Omega_n(q)$ respectively. \\[2pt] \hline \rule{0pt}{\normalbaselineskip}
\hspace{0.12cm} $\mathscr{C}_3$ & i) $M$ lies in the image modulo scalars of $Cl_{m}(q^t).t$ for a prime $t$ where $Cl_m(q^t)=GL^\epsilon_m(q^t),Sp_m(q^t), O_m(q^t), n=mt$ as $G=PSL^\epsilon_n(q), PSp_n(q), P\Omega_n(q)$; \\ \rule{0pt}{\normalbaselineskip}
& ii) $M$ lies in the image modulo scalars of $GU_\frac{n}{2}(q).2$ and $G$ is symplectic or orthogonal. \\[2pt] \hline \rule{0pt}{\normalbaselineskip}
\hspace{0.12cm} $\mathscr{C}_4$  & $M \le PGL_d^\epsilon(q) \times PGL_e^\epsilon(q), PGSp_d(q) \times PGO_e(q), PGO_d(q) \times PGO_e(q)$ or $PGSp_d(q) \times PGSp_e(q)$ as $G=PSL^\epsilon_n(q), PSp_n(q), P\Omega_n(q)$ or $P\Omega_n(q)$ respectively, where $n=de$ and $d, e<n$. \\ \hline   \rule{0pt}{\normalbaselineskip} 
\hspace{0.12cm} $\mathscr{C}_5$ & i) $M$ lies in $Cl_n(q^\frac{1}{t})$ for a prime $t$, where $Cl_n(q^\frac{1}{t})=PGL^\epsilon_n(q^\frac{1}{t}), PGSp_n(q^\frac{1}{t}), PGO_n(q^\frac{1}{t})$ as $G=PSL^\epsilon_n(q), PSp_n(q), P\Omega_n(q)$ respectively; \\ \rule{0pt}{\normalbaselineskip}
& ii) $G$ is unitary and $M$ lies in $PGSp_\frac{n}{2}(q)$ or $PGO_\frac{n}{2}(q)$. \\[2pt] \hline \rule{0pt}{\normalbaselineskip}
\hspace{0.12cm} $\mathscr{C}_6$ & i) $G=PSL^\epsilon_n(q)$ and $M$ lies in the image modulo scalars of $t^{1+2m}.Sp_{2m}(t)$ for a prime $t$ such that $n=t^m$, where either $t$ is odd and $t \mid q-\epsilon$ or $t=2$ and $4 \mid q-\epsilon$; \\ \rule{0pt}{\normalbaselineskip}
& ii) $G$ is symplectic or orthogonal and $M$ lies in the image modulo scalars of $2^{1+2m}.O_{2m}(2)$, where $n=2^m$ and $q$ is odd. \\[2pt] \hline \rule{0pt}{\normalbaselineskip}
\hspace{0.12cm} $\mathscr{C}_7$ & $M$ lies in $Cl_m(q) \wr S_t$ where $n=m^t$ and $Cl_m(q)$ has socle $PSL^\epsilon_m(q), PSp_m(q)$ ($qt$ odd), $PSp_m(q)$ ($qt$ even) or $P\Omega_m(q)$ as $G=PSL^\epsilon_n(q), PSp_n(q), P\Omega_n(q), P\Omega_n(q)$ respectively. \\[2pt] \hline \rule{0pt}{\normalbaselineskip}
\hspace{0.12cm} $\mathscr{C}_8$ & i) $G=PSL_n(q)$ and $M$ is the normalizer of $PSp_n(q), P\Omega_n(q)$ or $PSU_n(q^\frac{1}{2})$; \\ \rule{0pt}{\normalbaselineskip}
& ii) $G=Sp_n(q)$ with $q$ even and $M=SO_n(q)$. \\ \hline
\end{longtable}
\setlength\LTleft{0cm}

\subsection{The almost-free embedding}

\par Let $G$ be an almost simple classical group as in \S~2.1, and let $A$ be a nontrivial finite group. We define the almost-free embedding of $A$ into $G$, similar to \cite{LiSh3}. If $|A|=a, L$ is a field and $V$ is an $LA$-module, then we say $V$ is almost-free if $V=U \oplus I$, where $U \ne 0$ is a free $LA$-module and $I$ is the trivial $LA$-module of dimension $s$, where $2 \le s<2a+2$. 
\par Assume $n \ge 2a+2$, and let $n=ka+s$ with $k$ even and $2 \le s<2a+2$, and denote by $V$ the natural module of $G$. We can embed $A$ into $G$ such that $V$ is an almost-free $\mathbb{F}_qA$-module as follows: observe that if $R$ is the regular module for $A$, then $R^2=R \oplus R$ admits non-degenerate $A$-invariant symplectic, quadratic and unitary forms. Therefore we can embed $A$ into $G$ such that $V \downarrow A=R^2 \perp R^2 \perp \dots \perp R^2 \perp I$, where there are $\frac{k}{2} \ R^2$ factors and $I$ is the trivial $A$-module of dimension $s$, allowing us to identify $A$ with a subgroup of $G$.

\section*{\center{3. Main results and reduction of the proof}}
\stepcounter{section}
\par We now begin stating our two theorems mentioned in \S~1, from which we will deduce the main theorem. 
\par Let $A$ and $B$ be nontrivial finite 2-groups, not both of order 2, and assume $|B|>2$. The first result below is proved in Sections 5 and 6. It is used in the proof of the main theorem in the case where $B$ is not elementary abelian. For an integer $s$ and finite groups $H_1, H_2$, recall the definitions of $I_{s,H_2}(H_1)$ and $I_{2 \times 2,H_2}(H_1)$ from \S~2.1.

\begin{theorem} \label{bigthm1}
Let $A$ and $B$ be nontrivial 2-groups such that $B$ contains an element of order 4. There is an integer $N=N(A,B)$ such that for any finite simple classical group $G$ of rank at least $N$, there exist elements $x \in I_{2,A}(G), y \in I_{4,B}(G)$ such that
\begin{equation*}
\sum_{M \operatorname{max} G}  \frac{|x^G \cap M|}{|x^G|} \frac{|y^G \cap M|}{|y^G|}<1.
\end{equation*}
\end{theorem}

\par The next result is proved in Sections 7 and 8. It is used in the proof of the main theorem in the case where $B$ is an elementary abelian 2-group. 

\begin{theorem} \label{bigthm2}
Let $A$ and $B$ be nontrivial 2-groups such that $B$ contains a subgroup isomorphic to $C_2 \times C_2$. There is an integer $N=N(A,B)$ such that for any finite simple classical group $G$ of rank at least $N$, there exists an element $x \in I_{2,A}(G)$ and a subgroup $K \in I_{2 \times 2,B}(G)$ such that
\begin{equation*}
\sum_{M \operatorname{max} G} \frac{|x^G \cap M|}{|x^G|} \frac{| \{ K^g: g \in G, K^g \le M \} |}{|K^G|} <1.
\end{equation*}
\end{theorem}

\par We now deduce the main theorem from Theorems \ref{bigthm1} and \ref{bigthm2}. 

\par Let $A$ and $B$ be nontrivial finite groups, not both of order 2, and let $G$ be a finite simple classical group of rank $n$. By \cite[Theorem~2.3]{LiSh3} we may assume $A$ and $B$ are 2-groups. Assume that $|B|>2$, so either $B$ contains an element of order 4 or $B$ is elementary abelian. First suppose that $B$ contains an element of order 4. Let $N=N(A,B)$ be as in Theorem \ref{bigthm1}, and assume $n \ge N$. For $x \in I_{2,A}(G), y \in I_{4,B}(G)$, the number of pairs of conjugates $(x^g, y^h)$ with $g, h \in G$ such that $x^g, y^h \in M$ for some maximal subgroup $M$ of $G$ is at most 
 \vspace{-0.1cm}
\begin{align*}
\sum_{M \operatorname{max} G}  |x^G \cap M| |y^G \cap M|.
\end{align*}

\noindent By Theorem \ref{bigthm1} this is strictly less than $|x^G||y^G|$ for some $x \in I_{2,A}(G), y \in I_{4,B}(G)$, and hence there exists a pair of conjugates $(x^g,y^h)$ that generates $G$.
\par If $B$ is elementary abelian, an entirely similar argument using Theorem \ref{bigthm2} gives the main theorem.

\section*{\center{4. Lower bounds for sizes of almost-free conjugacy classes}}
\stepcounter{section}

\par Recall the definitions of $G$ and $H$ from \S~2.1, and let $A$ be a nontrivial finite group embedded almost-freely into $G$ (as in \S~2.4). In this section we find lower bounds for the sizes of $G$-conjugacy classes of elements $x \in A$ of prime order or order 4 (if such elements exist), and subgroups $K \le A$ isomorphic to the Klein 4-group (if such subgroups exists).
\par We first need a result bounding the sizes of conjugacy classes of elements and Klein four-groups in finite classical groups. This is a small extension of \cite[Corollary~1.8]{LaLiSe}.

\begin{lemma} \label{centorder}
Let $H$ be a simple adjoint algebraic group defined over a field of characteristic $p>0$, and let $F$ be a Steinberg endomorphism of $H$ such that $H^F$ is a group of Lie type over $\mathbb{F}_{q}$. Assume $H^F \ne {}^2F_4(q), {}^2G_2(q), {}^2B_2(q)$. Let $x \in H^F$ and $y_1, y_2 \in I_2(H^F)$ such that $[y_1,y_2]=1$. Also let $C=C_H(x), D=C_H(y_1) \cap C_H(y_2)$ with $b=\operatorname{rank}(Z(C^\circ/R_u(C^\circ))), d=\operatorname{rank}(Z(D^\circ/R_u(D^\circ)))$. Then 
\begin{align*}
\frac{1}{2}\frac{(q-1)^b}{q^b |C:C^\circ|}q^{\dim(H) - \dim(C^\circ)} < |x^{H^F}| < \frac{(q+1)^b}{q^b |C:C^\circ|}q^{\dim(H) - \dim(C^\circ)}, \\
\frac{1}{12} \frac{(q-1)^d}{q^d |D:D^\circ|}q^{\dim(H) - \dim(D^\circ)} < | \langle y_1, y_2 \rangle^{H^F} | < \frac{(q+1)^d}{q^d |D:D^\circ|}q^{\dim(H) - \dim(D^\circ)}.
\end{align*}
\end{lemma}

\begin{proof}
\par We follow the proof of \cite[Corollary~1.8]{LaLiSe}. Let $U=R_u(C^\circ), E=C^\circ/U$. Then $|U^F|=q^{\dim(U)}$ by \cite[Lemma~1.7]{LaLiSe}, and by \cite[Proposition~1.6]{LaLiSe},
\begin{align*}
\frac{(q-1)^b}{q^b}q^{\dim(E)} \le |E^F| \le \frac{(q+1)^b}{q^b}q^{\dim(E)}.
\end{align*} 
\noindent From \cite[Lemma~1.2]{LaLiSe} we have $\frac{1}{2} q^{\dim(H)}<|H^F| <q^{\dim(H)}$, and the bounds for $|x^{H^F}|$ follow.

\par The proof for $ | \langle y_1, y_2 \rangle^{G} | $ is similar, noting that $|N_{H^F}(\langle y_1,y_2 \rangle):C_{H^F}(\langle y_1,y_2 \rangle)| \le |S_3|$.
\end{proof}

\par We now find lower bounds for the sizes of conjugacy classes of elements $x \in A \le G$ of order $r$, where $r$ is prime or $r=4$ (if such an element exists), and we restrict to the case where $G=S$ is simple. Write $\mathbb{P}$ for the set of all prime numbers. Recall that, for groups $H_2 \le H_1$, we write $\overline{H_1}=H_1/Z(H_1)$ and $\overline{H_2}=H_2/(Z(H_1) \cap H_2)$.

\begin{proposition} \label{conjclassprop1}
Let $A$ be a nontrivial finite group, and choose $r \in \mathbb{P} \cup \{ 4 \}$ such that $A$ contains an element $x$ of order $r$. Let $G$ be a finite simple classical group with natural module of dimension $n$, and assume $n \ge 2|A| +2$. If $A$ is embedded almost-freely into $G$, then
\begin{equation*}
|x^G| > |G|^{\frac{r-1}{r}+\frac{c}{n}}
\end{equation*}
\noindent for a constant $c=c(A)$ depending only on $A$. In particular, $i_{r,A}(G) > |G|^{\frac{r-1}{r}+\frac{c}{n}}$.

\end{proposition}

\begin{proof}
\par Let $H$ be the simple adjoint algebraic group over $l=\overline{\mathbb{F}_q}$ of characteristic $p$ such that $(H^F)'=G$ for a Frobenius morphism $F$ of $H$. Also, let $|A|=a, x \in I_r(A)$ and write $n=ka+s$ where $k$ is even and $2 \le s < 2a+2$. We embed $A$ into $G$ almost-freely as in \S~2.4.

\par First suppose $p \nmid r$, so $x$ is semisimple with $r$ distinct eigenvalues. With this embedding, $C_H(x)^\circ$ is isomorphic to a group in the second or third columns of Table \ref{centxAtable}.

\renewcommand{\arraystretch}{2.2}
\begin{longtable}[c]{c c c}
\caption {Centralizers of elements $x \in A \le H$ of order $r, p \nmid r$} \\ \label{centxAtable}
$H$  & $C_H(x)^\circ, r$ odd & $C_H(x)^\circ, r$ even \\[2pt] \hline \vspace{0.2cm}
 $PSL_n$ & $\overline{\bigg(GL_{\frac{ka}{r}}^{r-1} \times GL_{{\frac{ka}{r}}+s} \bigg) \cap SL_n}$ & $\overline{\bigg(GL_\frac{ka}{r}^{r-1} \times  GL_{\frac{ka}{r}+s}  \bigg) \cap SL_n}$  \\[3pt] \hline 
$PSp_n$ & $\overline{GL_{\frac{ka}{r}}^{\frac{r-1}{2}} \times Sp_{{\frac{ka}{r}}+s}}$ & $\overline{GL_\frac{ka}{r}^{\frac{r-2}{2}} \times Sp_{\frac{ka}{r}} \times  Sp_{\frac{ka}{r}+s}}$\\[3pt] \hline 
$PSO_n$ & $\overline{GL_{{\frac{ka}{r}}}^{\frac{r-1}{2}} \times SO_{{\frac{ka}{r}}+s}} $ & $\overline{GL_\frac{ka}{r}^{\frac{r-2}{2}} \times SO_{\frac{ka}{r}} \times  SO_{\frac{ka}{r}+s}}$\\[3pt] \hline
\end{longtable}
\setlength\LTleft{0cm}
\par Observe that in each case we have
\begin{equation} \label{dimeqn1}
\dim C_H(x)^\circ < \frac{1}{r}\dim(G) + cn
\end{equation}
\noindent for some constant $c=c(A)$ only depending on $A$.

\par By Lemma \ref{centorder} and (\ref{dimeqn1}),
\begin{align*}
|x^{H^F}| &> \frac{1}{2}\frac{(q-1)^b}{q^b|C_H(x):C_H(x)^\circ|}q^{(1-\frac{1}{r})\dim(G)-cn}
\end{align*}
\noindent where $b=\operatorname{rank}(Z(C_H(x)^\circ))$. Both the index $|C_H(x):C_H(x)^\circ|$ and $b$ are bounded above by $r$. Hence, for some constant $c'=c'(A)$ depending only on $A$, we have
\begin{align*}
|x^{H^F}| &> q^{(1-\frac{1}{r})\dim(G)+c'n}
\end{align*}

\noindent and the result follows, completing the proof in the case $p \nmid r$.

\par We now consider the case where $p \mid r$, so $x$ has Jordan form $J_r^{\frac{ka}{r}} \oplus J_1^s$. Therefore, by \cite[Theorem~3.1]{LiSe} if $H=PSL_n$ or $p \ne 2$ and by \cite[Theorem~4.2]{LiSe} if $H=PSp_n, PSO_n$ and $p=2$, we have upper bounds for $\operatorname{dim}C_H(x)^\circ$ listed in Table \ref{cent1table}. In each case, we observe that (\ref{dimeqn1}) holds.

\begin{table}[h]
\begin{center}
\def\arraystretch{1.5}
\caption {Upper bounds for $\dim(C_H(x)^\circ), \ p \mid r$} \label{cent1table}
\begin{tabular}{ c@{\hskip 0.4in}  c@{\hskip 0.4in} c }
$H$  & $\dim(C_H(x)^\circ), r$ odd & $\dim(C_H(x)^\circ), r$ even \\[2pt] \hline
$PSL_n$ & $\frac{n^2}{r} + s^2(1-\frac{1}{r}) $ & $\frac{n^2}{r} + s^2(1-\frac{1}{r})$  \\[5pt] \hline 
$PSp_n$ & $ \frac{1}{2r}(n^2+n)+\frac{s^2}{2}(1-\frac{1}{r})+ \frac{s}{2}(1-\frac{1}{r}) $ & $\frac{1}{2r}(n^2+4n)+\frac{s^2}{16}(r+2)+\frac{3s}{8}(r-2)$\\[5pt] \hline 
$PSO_n$ & $ \frac{1}{2r}(n^2-n) + \frac{s^2}{2}(1-\frac{1}{r}) + \frac{s}{2}(\frac{1}{r}-1)$ & $ \frac{n^2}{2r}+\frac{s^2}{8}(4-r)+\frac{s}{8}(2-r)$\\[5pt] \hline
\end{tabular}
\end{center}
\end{table}

\par The result now follows by Lemma \ref{centorder} and (\ref{dimeqn1}) as in the previous case, noting that $b=\operatorname{rank}(Z(C_H(x)^\circ/R_u(C_H(x)^\circ)))$ and $|C_H(x):C_H(x)^\circ|$ are bounded above by $2^3$ (see \cite[Theorem~3.1, Theorem 4.2]{LiSe}).
\end{proof}

\par We now combine \cite[Proposition~4.1]{LiSh1} and \cite[Lemma~1]{LiSh2} to give an upper bound for $i_r(G)$ for a prime $r$. 

\begin{proposition}[\cite{LiSh1, LiSh2}] \label{irprop}
Let $G$ be a finite almost simple classical group with natural module $V$ of dimension $n$, and assume $G \le PGL(V)$. Then for any prime $r$, 
\begin{equation*} 
i_r(G) < |G|^{\frac{r-1}{r}+\frac{c}{n}}
\end{equation*}
\noindent for some constant $c=c(r)$ depending only on $r$. 
\end{proposition}

\par We now consider Klein four-groups, and prove a result analogous to Proposition \ref{conjclassprop1}.

\begin{proposition} \label{conjclassprop2}
Let $A$ be a finite group containing a Klein four-group $K$. Let $G$ be a finite simple classical group with natural module of dimension $n$, and assume $n \ge 2|A|+2$. If $A$ is embedded almost-freely into $G$, then
\begin{equation*}
|K^G|> |G|^{\frac{3}{4}+\frac{c}{n}}
\end{equation*}
\noindent for a constant $c=c(A)$ depending only on $A$.
\end{proposition}

\begin{proof}
The proof is similar to that of Proposition \ref{conjclassprop1}. Let $K=\langle y_1, y_2 \rangle$. If $|A|=a$ then let $n=ka+s$, where $k$ is even and $2 \le s < 2a+2$.
\par First suppose $p \ne 2$, so $y_1$ and $y_2$ are semisimple. Let $D=C_H(y_1) \cap C_H(y_2)$. Since $A$ is embedded almost-freely, $D^\circ$ is isomorphic to
\begin{align*}
\overline{\left( GL_{\frac{ka}{4}}^3 \times GL_{{\frac{ka}{4}}+s} \right) \cap SL_n} \hspace{1cm}  &\text{if} \ H=PSL_n, \\
\overline{Sp_{\frac{ka}{4}}^3 \times Sp_{{\frac{ka}{4}}+s}} \hspace{1cm}  &\text{if} \ H=PSp_n, \\
\overline{SO_{{\frac{ka}{4}}}^3 \times SO_{{\frac{ka}{4}}+s}} \hspace{1cm}  &\text{if} \ H=PSO_n.
\end{align*}

\noindent In each case we have
\begin{equation} \label{dimeqn3}
\dim D^\circ < \frac{1}{4}\dim(G) + cn
\end{equation}
\noindent for some constant $c=c(A)$ depending only on $A$.

\par By Lemma \ref{centorder} and (\ref{dimeqn3}),
\begin{align*}
|\langle y_1, y_2 \rangle^{H^F}| &> \frac{1}{12}\frac{(q-1)^d}{q^d|D:D^\circ|}q^{\frac{3}{4}\dim(G)-cn}
\end{align*}
\noindent where $d=\operatorname{rank}(Z(D^\circ))$. Both the index $|D:D^\circ|$ and $d$ are bounded above by $4$. Hence, for some constant $c'=c'(A)$ depending only on $A$, we have
\begin{align*}
|\langle y_1, y_2 \rangle^{H^F}| &> q^{\frac{3}{4} \dim(G)+c'n}
\end{align*}
\noindent and the result follows, which completes the proof for the case $p \ne 2$.

\par We now consider the case where $p=2$, so $y_1$ and $y_2$ are unipotent elements. For $n_1 \le \frac{n_2}{2}$, let $j_{n_1, n_2}$ denote the matrix
 \vspace{-0.2cm}
\begin{align*}
\begin{pmatrix}
I_{n_1} \\
& I_{n_2-2n_1} \\
 I_{n_1}&   & I_{n_1} 
\end{pmatrix} \in GL_{n_2}.
\end{align*}
\noindent Since $A$ is almost-free, we can take $y_1, y_2$ as the image modulo scalars of
\begin{align*}
j_{\frac{ka}{2},n},
\begin{pmatrix}
j_{\frac{ka}{4},\frac{ka}{2}} \\
&  I_s \\
& & j_{\frac{ka}{4},\frac{ka}{2}}
\end{pmatrix}
\end{align*}
\noindent respectively, where for $H$ symplectic or orthogonal this is with respect to a basis with Gram matrix
 \vspace{-0.2cm}
\begin{align*}
\begin{pmatrix}
& & & & J \\
& & & J \\ & & J' \\ & J \\ J
\end{pmatrix}
\end{align*}
 \vspace{-0.1cm}
\noindent where 
 \vspace{-0.1cm}
\begin{align*}
J= \begin{pmatrix}
&  I_{\frac{ka}{8}} \\
 I_{\frac{ka}{8}}
\end{pmatrix}, \ J'= \begin{pmatrix}  & I_{\frac{s}{2}}  \\
I_{\frac{s}{2}}  \end{pmatrix},
\end{align*}
\noindent and if $H$ is orthogonal this is with respect to a basis $\{e_1, \dots, e_n \}$ and quadratic form $f$ defined as
 \vspace{-0.2cm}
\begin{align} \label{quadform}
f(\sum_{i=1}^n \lambda_i e_i)=\sum_{j=1}^4 \sum_{i=1}^\frac{ka}{8} \lambda_{\frac{(j-1)ka}{8}+i} \lambda_{n-\frac{jka}{8}+i}+\sum_{i=\frac{ka}{2}+1}^{\frac{s}{2}} \lambda_{i}\lambda_{\frac{s}{2}+i}.
\end{align}
 \vspace{-0.2cm}
\par Let $D=C_H(y_1) \cap C_H(y_2)$. Then $D$ consists of images of matrices
\begin{align*}
\begin{pmatrix}
X \\
P & Y \\
Q & R & X \\
\end{pmatrix}
\end{align*}
\noindent where $X \in GL_{\frac{ka}{2}}, Y \in GL_s$, $X$ and $Q$ commute with $j_{\frac{ka}{4},\frac{ka}{2}}$, $P=P j_{\frac{ka}{4},\frac{ka}{2}}$ and $R=j_{\frac{ka}{4},\frac{ka}{2}} R$. Observe that $D=UL$, where $U$ is the subgroup with $X=Y=1$ and $L$ is the subgroup with $P=Q=R=0$. Write $X,Q, P$ and $R$ as block matrices
\begin{align*}
X=\begin{pmatrix} X_1 \\ X_2 & X_1 \end{pmatrix}, \ Q=\begin{pmatrix} Q_1 \\ Q_2 & Q_1 \end{pmatrix}, \ P=\begin{pmatrix} P_1 & 0 \end{pmatrix}, \ R=\begin{pmatrix} 0 \\  R_1 \end{pmatrix}
\end{align*}
\noindent with blocks of appropriate dimensions.
\par If $H=PSL_n$, then $X, Y$ lie in $C_{GL_{\frac{ka}{2}}}(j_{\frac{ka}{4},\frac{ka}{2}}), GL_s$ respectively. By \cite[Theorem~3.1]{LiSe},
 \vspace{-0.1cm}
\begin{align*}
\dim(C_{GL_\frac{ka}{2}}{(j_{\frac{ka}{4},\frac{ka}{2}})})= \frac{1}{8}k^2a^2,
\end{align*}
\noindent and so $\dim(L)=\frac{1}{8}k^2a^2+s^2-1$. We have $\dim(U)=\frac{1}{8}k^2a^2+\frac{1}{2}kas$.
\noindent Hence
 \vspace{-0.1cm}
\begin{align*}
\dim(D^\circ) = \frac{n^2}{4} + \frac{3}{4}s^2 -1.
\end{align*}

\par If $H$ is symplectic, then the following conditions hold: 
 \vspace{-0.1cm}
\begin{align}
 X &\in C_{Sp_{\frac{ka}{2}}}(j_{\frac{ka}{4},\frac{ka}{2}}), \hspace{4cm} \label{Kcondition1} \\
Y &\in Sp_s,  \label{Kcondition2} \\
R &= Y^TJ'PX^{-1}J, \label{Kcondition3} \\
X_1^T J Q_1 &= Q_1^T J X_1,  \label{Kcondition4} \\
X_1^T J Q_2 + Q_2^T J X_1+X_2^T J Q_1 + Q_1^T J X_2  &=P_1^T J' P_1. \label{Kcondition5}
\end{align}
\noindent By \cite[Theorem~4.3]{LiSe}, 
 \vspace{-0.2cm}
\begin{align}
\dim(C_{Sp_\frac{ka}{2}}{(j_{\frac{ka}{4},\frac{ka}{2}})}) = \frac{1}{16}k^2a^2+\frac{1}{4}ka, \label{Spcent}
\end{align}
\noindent and so $\dim(L)= \frac{1}{16}k^2a^2+\frac{1}{2}ka+\frac{1}{2}s(s+1)$ by (\ref{Kcondition1}) and (\ref{Kcondition2}). Using conditions (\ref{Kcondition3})-(\ref{Kcondition5}), matrix calculations show that $\dim(U)=\frac{1}{16}k^2a^2+\frac{1}{4}kas+\frac{1}{4}ka$. Hence
 \vspace{-0.2cm}
\begin{align*}
\dim(D^\circ) =\frac{n^2}{8}+\frac{n}{2} + \frac{3s^2}{8}. 
\end{align*}

 \vspace{-0.1cm}
\par If $H$ is orthogonal, then in addition to conditions (\ref{Kcondition1}), (\ref{Kcondition3})-(\ref{Kcondition5}) two further conditions hold (coming from the quadratic form in (\ref{quadform})): 
\begin{align}
Y &\in SO_s,  \hspace{2cm} \label{Kcondition6} \\
(X_1^T J Q_2)_{ii}&=(X_2^TJ Q_1)_{ii}, \ 1 \le i \le \frac{ka}{4}. \label{Kcondition7}
\end{align}

\noindent By (\ref{Kcondition1}), (\ref{Spcent}) and (\ref{Kcondition6}) we have $\dim(L) = \frac{1}{16}k^2a^2+\frac{1}{4}ka+\frac{1}{2}s(s-1)$. From (\ref{Kcondition3})-(\ref{Kcondition5}) and (\ref{Kcondition7}) we calculate $\dim(U)=\frac{1}{16}k^2a^2+\frac{1}{4}kas-\frac{1}{4}ka$. Hence
 \vspace{-0.1cm}
\begin{align*}
\dim(D^\circ) =\frac{n^2}{8} + \frac{3s^2}{8} - \frac{s}{2}.
\end{align*}
 \vspace{-0.5cm}
\par Therefore, in each case (\ref{dimeqn3}) holds for some constant $c=c(A)$ only depending on $A$. 
\par Both the index $|D:D^\circ|$ and $\operatorname{rank}(Z(D^\circ/R_u(D^\circ)))$ are bounded above by $2^{4}$ (see \cite[Theorem~3.1, Theorem 4.2]{LiSe}), and so the result follows as in the previous case by Lemma \ref{centorder} and (\ref{dimeqn3}).
\end{proof}

\section*{\center{5. Elements of order 4: classical groups}}
\stepcounter{section}

\par Recall the definitions of $S,G,H$ and $F$ from \S~2.1. In this section we count elements of order 4 in $G$ where $G$ is contained in $H^F$ or a split extension of $H^F$ by the group generated by an involutory automorphism of $S$. The result will be useful in the proof of Proposition \ref{i4Mprop1} in \S~6, and we only require the result for these specific split extensions.

\begin{proposition} \label{i4prop}
There exist absolute constants $c_1, c_2$ such that if $S$, $G$, $H$ and $F$ are as above with $G \le H^F\langle \alpha \rangle$ for some $\alpha \in I_2(\operatorname{Aut}(S)) \cup \{ 1 \}$, then 
\begin{equation*}
|G|^{\frac{3}{4}+\frac{c_1}{n}} < i_4(G) < |G|^{\frac{3}{4}+\frac{c_2}{n}}.
\end{equation*}
\end{proposition}

\par Before we begin the proof of Proposition \ref{i4prop}, we briefly discuss the setup for Shintani descent. This technique uses facts about algebraic groups and allows us to count elements of a given order in $G \backslash S$.

 \vspace{-0.2cm}
\subsection{Shintani descent}
\par We follow \cite[\S~2.6]{BuGu}. Let $X$ be a connected linear algebraic group over the algebraic closure of $\mathbb{F}_p$. Let $F_X:X \rightarrow X$ be a Frobenius map, a bijective algebraic group endomorphism with finite fixed point subgroup $X^{F_X}$. Let $Y_0=X^{F_X}$ and $Y=X^{F_X^e}$ for some integer $e$. Denote by $E$ be the cyclic group of automorphisms of $Y$ generated by $F_X$, so $|E|=e$, and let $Y.E$ be the semidirect product of $Y$ and $E$.
\par We consider $Y.E$-conjugacy classes of elements of the form $ g F_X \in Y F_X$. By the Lang-Steinberg theorem, there exists $x \in X$ such that $g=(x^{-1})( {}^{F_X} x)$. Then $x(g F_X)^e x^{-1}$ is fixed by $F_X$, and hence an element of $Y_0$. This allows us to define a well-defined map $f$ from the $Y.E$-conjugacy classes in the coset $Y F_X$ to $Y_0$-conjugacy classes by 
\begin{equation*}
f: (g F_X)^{Y.E} \mapsto (x(g F_X)^e x^{-1})^{Y_0},
\end{equation*}
\noindent and we call this map the {\it Shintani map} of $Y$ corresponding to $F_X$. The following result is taken from \cite[Lemma~2.13]{BuGu}.
\begin{lemma} \label{shintani}
With the notation above, the following hold:
\begin{enumerate}[label=(\roman*)]
\item We have
\begin{equation*}
C_Y(gF_X)=x^{-1}C_{Y_0}(x(gF_X)^ex^{-1})x=C_{x^{-1}Y_0x}((gF_X)^e).
\end{equation*}
\noindent In particular, $|C_Y(gF_X)|=|C_{Y_0}(x(gF_X)^ex^{-1})|$ for all $g \in Y$.
\item The Shintani map $f$ is a bijection.
\end{enumerate}
\end{lemma}

\subsection{Proof of Proposition \ref{i4prop}}

\par Observe that to prove the upper bound in Proposition \ref{i4prop} it suffices to prove the result for $G=H^F \langle \alpha \rangle$, and similarly for the lower bound it suffices to prove the result for $G=S$. In fact, the lower bound holds by Proposition \ref{conjclassprop1}, and so from now on we will only consider the case $G=H^F \langle \alpha \rangle$.

\begin{lemma} \label{i4prop1} 
Proposition \ref{i4prop} holds in the case $\alpha=1$.
\end{lemma}

\begin{proof}

\par Firstly suppose $p \ne 2$, and let $x \in I_4(H^F)$. Then $x$ has at most 4 distinct eigenvalues. If $n_i \in \mathbb{Z}_{\ge 0}, 1 \le i \le 4$ are the dimensions of the eigenspaces of $x$, then $C_H(x)^\circ$ is isomorphic to a group in column 2 of Table \ref{prop5.3table}.

\begin{table}[h]
\begin{center}
\def\arraystretch{1.3}
\caption {Centralizers of elements $x \in H$ of order 4, $p \ne 2$}
\label{prop5.3table}
\begin{tabular}{ c  c c }
$H$  & $C_H(x)^\circ$ & Conditions \\ \hline
$PSL_n$ & $\overline{\prod_{n_i} GL_{n_i} \cap SL_n}$ &  \\ \hline 
$PSp_n$ & $\overline{GL_{n_1} \times Sp_{n_3} \times Sp_{n_4}}$ & $n_1=n_2$\\ \hline
$PSO_n$ & $\overline{GL_{n_1} \times SO_{n_3} \times SO_{n_4}} $ & $n_1=n_2$\\ \hline
$PSp_n, PSO_n$ & $\overline{GL_{n_1} \times GL_{n_3}}$ & $n_1=n_2, n_3=n_4$\\ \hline
\end{tabular}
\end{center}
\end{table}

\par By the Cauchy-Schwarz inequality we have 
\begin{align} \label{5.1eqn1}
\dim(C_H(x)^\circ) > \frac{1}{4}\dim(G)+cn
\end{align}
\noindent for some absolute constant $c$. Therefore, by Lemma \ref{centorder} we have
\begin{align} \label{5.1eqn2}
|x^{H^F}| < \frac{(q+1)^b}{q^b |C_H(x):C_H(x)^\circ|}q^{\frac{3}{4}\dim(G) -cn}
\end{align}
\noindent where $b=\operatorname{rank} Z(C_H(x)^\circ) < 4$. 

\par There are fewer than $c'(qn)^3$ $H$-conjugacy classes of elements of order 4 for some absolute constant $c'$. For any $x \in H^F$, the $F$-stable elements of the $H$-conjugacy class $x^H$ split into at most $|C_H(x):C_H(x)^\circ|$ $H^F$-classes. Therefore, using (\ref{5.1eqn2}),
\begin{align*}
i_4(H^F) &< q^{\frac{3}{4}\dim(G)+c'n}
\end{align*}
\noindent for some absolute constant $c'$, and the result follows for $p \ne 2$.

\par We now move onto the case where $p=2$. Let $x \in I_4(H^F)$ have Jordan form $\bigoplus_{i=1}^4 J_i^{n_i}$.  Then, by \cite[Theorem~3.1]{LiSe} if $H$ is linear and \cite[Theorem~4.2]{LiSe} if $H$ is symplectic or orthogonal, we have $\operatorname{dim} C_{H}(x)^\circ$ listed in Table 6.

\begin{table}[h]
\begin{center}
\def\arraystretch{1.55}
\caption {Dimensions of $C_H(x)^\circ, \ p = 2$}
\label{dimcentorder4}
\begin{tabular}{ c  c c }
$H$  & $\dim(C_H(x)^\circ)$ & Notes \\ \hline
$PSL_n$ & $\sum in_i^2+2\sum_{i<j} in_in_j -1$ &  \\ \hline 
$PSp_n$ & $\frac{1}{2} \sum in_i^2+\sum_{i<j} in_in_j +\frac{1}{2}\sum_{i \ \text{odd}} n_i + \sum_{i \ \text{even}} \delta_in_i $ & $\delta_i \in \{ 0,1 \}$\\ \hline
$PSO_n$ & $\frac{1}{2}\sum in_i^2+\sum_{i<j} in_in_j -\frac{1}{2}\sum_{i \ \text{odd}}n_i- \sum_{i \ \text{even}} \delta_in_i $ & $\delta_i \in \{ 0,1 \}$\\ \hline
\end{tabular}
\end{center}
\end{table}

\par One can easily verify that 
\begin{align*}
\sum in_i^2 + 2\sum_{i<j}in_in_j \ge \frac{1}{4}(\sum in_i)^2=\frac{1}{4}n^2,
\end{align*}
\noindent and so (\ref{5.1eqn1}) holds. Therefore, by Lemma \ref{centorder},
\begin{align*} 
|x^{H^F}| < \frac{(q+1)^b}{q^b |C_H(x):C_H(x)^\circ|}q^{\frac{3}{4}\dim(G) -cn}
\end{align*}
\noindent where $b=\operatorname{rank} Z(C_H(x)^\circ/R_u(C_H(x)^\circ)) \le 4$.
\par By \cite[Theorem~4.2]{LiSe}, for an element of order 4 with a given Jordan form, the number of $H$-classes of elements with that form is at most 4. Therefore there are fewer than $c'n^3$ $H$-classes of elements of order 4 for some absolute constant $c'$. By \cite[Theorem~7.1]{LiSe} if $H$ is linear and by \cite[Theorem~7.3]{LiSe} if $H$ is symplectic or orthogonal, for any $x \in I_4(H^F)$ the $F$-stable elements of the $H$-conjugacy class $x^H$ split into at most $4$ $H^F$-classes. The result now follows as in the previous case.
\end{proof}

\begin{lemma}
Proposition \ref{i4prop} holds if $p \ne 2$ and $\alpha \ne 1$.
\end{lemma}

\begin{proof}
\par By Lemma \ref{i4prop1} it suffices to consider elements in the coset $H^F \alpha$.
\par Let $g\alpha \in I_4(H^F \alpha)$, so $g g^\alpha \in I_2(H^F)$. Observe that $g\alpha$ acts on $C_{H^F}(gg^\alpha)$, and that $C_{H^F}(gg^\alpha)$ contains a subgroup $C=C_0$ listed in Table \ref{auttable}. In most cases $C/Z(C)$ is isomorphic to a simple group or a direct product of simple groups, and $g\alpha$ acts on $C/Z(C)$ ($C$ is a characteristic subgroup of $C_{H^F}(gg^\alpha)$). The exceptions are listed above Table \ref{auttable}, and in these cases let $C$ be last term in the derived series of $C_0$, on which $g\alpha$ acts. Therefore $g \alpha$ acts on $C/Z(C) \cong S_1$ or $S_1 \times S_2$, where $S_1, S_2$ are simple.
\par If $C/Z(C) \cong S_1$, then $C_{S_1}(g\alpha)$ contains a subgroup $C_1$ isomorphic to a group listed in Table \ref{auttable}. We compute that 
\begin{equation*}
|C_1| > |S_1|^{\frac{1}{2}+\frac{c}{n}} > |H^F|^{\frac{1}{4}+\frac{c'}{n}}
\end{equation*}
\noindent for some absolute constants $c, c'$. Therefore
\begin{align*}
|C_C(g\alpha)| > |H^F|^{\frac{1}{4}+\frac{c'}{n}},
\end{align*}
\noindent and so
\begin{align} \label{galphaconjclass}
|g\alpha^{H^F \langle \alpha \rangle}|<|H^F|^{\frac{3}{4}+\frac{c''}{n}}
\end{align}
\noindent for some absolute constant $c''$.
\par If $C/Z(C) \cong S_1 \times S_2$ and $S_1 \ncong S_2$, then $g\alpha$ acts on each factor. As in the previous case, we compute
\begin{align*}
|C_{S_i}(g\alpha)|>|S_i|^{\frac{1}{2}+\frac{c}{n}},
\end{align*}
\noindent and so
\begin{align*} 
|C_C(g\alpha)| > |S_1 \times S_2|^{\frac{1}{2}+\frac{c}{n}} > |H^F|^{\frac{1}{4}+\frac{c'}{n}}.
\end{align*}
\noindent Hence (\ref{galphaconjclass}) holds as above.

\par If $C/Z(C) \cong S_1 \times S_2$ and $S_1 \cong S_2$, then $g\alpha$ may switch the factors. In the case where $g \alpha$ does not switch the factors, we have $|g\alpha^{H^F \langle \alpha \rangle}|$ bounded as in the previous case. Otherwise
\begin{align*}
C_{C/Z(C)}(g\alpha) \cong S_1, 
\end{align*}
\noindent and so
\begin{align*}
|C_C(g \alpha)| > |S_1| > |H^F|^{\frac{1}{4}+\frac{c}{n}}.
\end{align*}
\noindent Hence (\ref{galphaconjclass}) holds for all $g \alpha \in I_4(H^F \alpha)$.

\par We now bound the number of conjugacy classes of elements in $I_4(H^F \alpha)$. Let $g\alpha, h \alpha \in I_4(H^F\alpha)$ and $gg^\alpha=t \in I_2(H^F)$, and suppose $t^x=hh^\alpha$ for some $x \in H^F$. Then $g\alpha$ and $(h\alpha)^{x^{-1}}$ act on $C_{H^F}(t)$. Observe that if $g\alpha$ and $(h\alpha)^{x^{-1}}$ are conjugate in $H^F$, then they are conjugate in $C_{H^F}(t)$, since both square to $t$. Therefore, to count the number of classes in $I_4(H^F\alpha)$, it suffices to count the number of $C_{H^F}(t)$-classes of involutions in $\operatorname{Out}(C_{H^F}(t))$ for each class of involutions $t \in H^F$. In \cite{dieu}, it is proved that any automorphism of a simple group $T$ lifts to automorphism of the quasisimple group $Y$ such that $\overline{Y}=T$; moreover the number of automorphisms of $Y$ restricting to a given automorphism of $T$ is less than $q^{cn}$ for some absolute constant $c$. Therefore, the natural homomorphism $\operatorname{Aut}(C_{H^F}(t)) \rightarrow \operatorname{Aut}(C_{H^F}(t)/Z(C_{H^F}(t)))$ is surjective, with the order of the kernel bounded above by $q^{c'n}$ for some absolute constant $c'$. The number of $C_{H^F}(t)$-classes in $I_2(\operatorname{Out}(C_{H^F}(t)))$ is therefore less than the number of $C_{H^F}(t)/Z(C_{H^F}(t))$-classes in $I_2(\operatorname{Out}(C_{H^F}(t)/Z(C_{H^F}(t))))$ multiplied by $q^{c'n}$. For $t \in I_2(H^F)$ define the subgroup $C \le C_{H^F}(t)$ as above. It therefore suffices to bound the number of $C/Z(C)$-classes in $I_2(\operatorname{Out}(C/Z(C)))$, since $\operatorname{Aut}(C/Z(C))=\operatorname{Aut}(C_{H^F}(t)/Z(C_{H^F}(t)))$.
\par We have $C/Z(C) \cong S_1$ or $S_1 \times S_2$, where $S_1,S_2$ are simple groups. By \cite[Theorem~4.5.1]{GoLySo}, the number of $\operatorname{Inndiag}(S_i)$-classes of involutions in $\operatorname{Out}(S_i)$ ($i=1,2$) is less $c''n$ for some absolute constant $c''$. Since $|\operatorname{Inndiag}(S_i):S_i| \le q+1$ ($i=1,2$), there are fewer than $c'''qn$ $S_i$-classes of involutions in $\operatorname{Out}(S_i)$. Involutory automorphisms of $C/Z(C)$ either act on each factor or switch the factors (if two isomorphic factors are present). Therefore, the number of $C/Z(C)$-classes of involutions in $\operatorname{Out}(C/Z(C))$ is less than $c'''qn(c'''qn+1)$. By \cite[Theorem~4.5.1]{GoLySo} again, the number of classes of involutions in $H^F$ is less than $c''''n$. Hence the number of $H^F$-classes in $I_4(H^F\alpha)$ is less than $q^{dn}$ for some absolute constant $d$. The result now follows.
\end{proof}

\begin{lemma}
Proposition \ref{i4prop} holds if $p=2$ and $\alpha \ne 1$.
\end{lemma}

\begin{proof}
\par By Lemma \ref{i4prop1} it suffices to prove the upper bound for $H^F\alpha$. With the terminology of \S~2.2, either $\alpha$ is a field or graph-field automorphism, or $H=PSL_n, PSO_n$ and $\alpha$ is a graph automorphism. 
\par In the case where $\alpha$ is a field or graph-field automorphism, we consider the Shintani map of $H^F$ corresponding to a Frobenius map $F_0:H \rightarrow H$ such that $F_0^2=F$ and, as an element of $\operatorname{Aut}(S)$, $F_0$ is $\operatorname{Inndiag}(S)$-conjugate to $\alpha$. By Lemma \ref{shintani}, $H^F \langle \alpha \rangle$-classes of elements of order 4 in $H^F \alpha$ are in one to one correspondence with $H^{F_0}$-classes of involutions.  Moreover, for $x \in H^F \alpha$, the centralizer $C_{H^F}(x)$ has the same order as the centralizer of an involution $y \in H^{F_0}$. By \cite[Theorem~7.1, Theorem~7.3]{LiSe}, there are fewer than $cn$ $H^{F_0}$-classes of involutions for some absolute constant $c$. By Proposition \ref{irprop}, 
\begin{equation*}
|C_{H^{F_0}}(y)| > |H^{F_0}|^{\frac{1}{2}+\frac{c'}{n}} \hspace{0.5cm} \forall y \in I_2(H^{F_0})
\end{equation*}
\noindent for some absolute constant $c'$, and the result follows.

\par Now suppose $\alpha$ is a graph automorphism. First consider $H=PSL_n$. By considering the quotient map $SL_n\langle \alpha \rangle \rightarrow PSL_n\langle \alpha \rangle$, it suffices to show the bound holds for $i_4(SL^\epsilon_n(q) \alpha)$. Let $g\alpha \in I_4(SL^\epsilon_n(q)\alpha)$. We consider $g \alpha$ as a unipotent element of $Sp_{2n}$. Let $W$ be the natural module of $Sp_{2n}$. By \cite[Theorem~1.2]{LaLiSe2}, $g\alpha$ has an orthogonal decomposition of the form
\begin{align*}
W \downarrow {g\alpha} = W(2)^b + W(4)^c + V(2)^d
\end{align*}
\noindent where $0 \le d \le 2$ and $g\alpha$ has Jordan normal form $J_{2m} \oplus J_{2m}$ on $W(2m)$ ($m=1,2$) and $J_2$ on $V(2)$, so $n=2b+4c+d$. Moreover, each decomposition represents a unique $SL_n$-class of unipotent elements in $I_4(SL_n \alpha)$. By \cite[Theorem~1.3]{LaLiSe2} we have
\begin{align*}
C_{SL^\epsilon_n(q)}(g\alpha) = DR
\end{align*}
\noindent where
\begin{align*}
|D|&=q^{\operatorname{dim}R_u(C_{SL_n}(g\alpha))} =\left\{ 
  \begin{array}{l l}
q^{2c^2+\frac{1}{2}d^2+4bc+2bd+2cd-2b-c-\frac{1}{2}d} \hspace{0.5cm} &\text{if} \ d \ge 1, \\[0.1cm]
q^{2c^2+4bc+c}  &\text{if} \ d=0, \\
\end{array}
\right. \\
R&=\left\{ 
  \begin{array}{l l}
Sp_{2b}(q) \times Sp_{2c}(q) \hspace{0.5cm} &\text{if} \ d \ge 1, \\[0.1cm]
Sp_{2b}(q) \times O_{2c}(q) &\text{if} \ d=0.
\end{array}
\right.
\end{align*}
\noindent It then follows that
\begin{align*}
|C_{SL^\epsilon_n(q)}(g\alpha))| >c'q^{\frac{1}{4}(2b+4c+d)^2+\frac{1}{4}(2b+d)(2b+d-2)} \ge c'q^{\frac{1}{4}(2b+4c+d)^2-\frac{1}{4}}> q^{\frac{n^2}{4}+c''}. 
\end{align*}
\noindent By \cite[Theorem~1.3]{LaLiSe2}, each $SL_n$-class of elements in $I_4(SL^\epsilon_n(q) \alpha)$ splits into at most 2 $SL^\epsilon_n(q)$-classes, and so there are fewer than $c''n$ $SL^\epsilon_n(q)$-classes of elements in $I_4(SL^\epsilon_n(q) \alpha)$. The result now follows for $H=PSL_n$.

\par Now suppose $H=PSO_n$, so $H^F \langle \alpha \rangle \cong SO^\epsilon_n(q)$. Then $H^F \langle \alpha \rangle \le Sp_n(q)$, and so the result follows from Lemma \ref{i4prop1} since $|Sp_n(q)|^{\frac{3}{4}+\frac{c}{n}} \sim |SO_n(q)|^{\frac{3}{4}+\frac{c'}{n}}$. This completes the proof.
\end{proof}

\section*{\center{6. Elements of order 4: maximal subgroups}}
\stepcounter{section}

\par Let $G$ be a finite simple classical group, with natural module of dimension $n$ over $\mathbb{F}_{q^\delta}$, where $\delta=2$ if $G$ is unitary and $\delta=1$ otherwise.

\par In this section, we prove Theorem \ref{bigthm1} over three subsections. Let $A$ and $B$ be nontrivial finite 2-groups with $|A|=a, |B|=b, x \in I_2(A)$ and $y \in I_4(B)$. Also assume $n \ge \operatorname{max}\{2a+2, 2b+2\}$, and embed $A$ and $B$ almost-freely into $G$ as in \S~2.4 with $n=k_aa+s_a=k_bb+s_b$, where $k_i$ is even and $2 \le s_i < 2i+2$ for $i=a,b$.
\par In \S~6.1, we prove the following result.

\begin{proposition} \label{i4Mprop1}
There exists an absolute constant $c$ such that for any non-parabolic maximal subgroup $M$ of $G$, 
\begin{align*}
i_4(M)<|M|^\frac{3}{4}q^{cn}.
\end{align*}
\end{proposition}

\par Recall the notation of \S~2.3. In \S~6.2, we prove 
\begin{proposition} \label{i4Mprop2}
 For a maximal parabolic subgroup $M=P_m \ (1 \le m \le \frac{n}{2})$,
\begin{equation*} 
\sum_{M^g \in M^G} \frac{|x^G \cap M^g|}{|x^G|}\frac{|y^G \cap M^g|}{|y^G|} < q^{-f+cm},
\end{equation*}
\noindent where $c=c(A,B)$ is a constant depending only on $A$ and $B$, and
\begin{align*}
 f=\left\{ 
  \begin{array}{l l}
\frac{1}{4}m(n-m) \hspace{0.3cm} &\text{if $G=PSL_n(q)$,} \\[0.2cm]
\frac{1}{8}m(2n-3m) &\text{if $G=PSp_n(q), P\Omega_n(q)$,} \\[0.2cm]
\frac{1}{4}m(2n-3m) &\text{if $G=PSU_n(q)$}.
\end{array}
\right. 
\end{align*}
\noindent In particular, $-f +cm < -\frac{\delta n}{4}+c'$ for some constant $c'=c'(A,B)$.
\end{proposition}

\par Using Propositions \ref{i4Mprop1} and \ref{i4Mprop2}, we deduce Theorem \ref{bigthm1} in \S~6.3.

\subsection{Non-parabolic maximal subgroups}
 
\par Recall the descriptions of Aschbacher classes $\mathscr{C}_i, \ 1 \le i \le 8$ and $\mathscr{S}$ from \S~2.3. In this section, we prove Proposition \ref{i4Mprop1} for each Aschbacher class.

\begin{lemma}
Proposition \ref{i4Mprop1} holds for $M \in \mathscr{C}_1$ with $M$ non-parabolic.
\end{lemma}

\begin{proof}
\par Let $M \in \mathscr{C}_1$ with $M$ non-parabolic as in Table \ref{maxsubgroups}. Suppose we are in the unitary case, so $M$ lies is the image modulo scalars of $GU_m(q) \times GU_{n-m}(q)$. Then by Propositions \ref{irprop} and \ref{i4prop},
\begin{align*}
i_4(M) &< c(i_4(PGU_{m}(q))+i_2(PGU_{m}(q))+1)(i_4(PGU_{n-m}(q))+i_2(PGU_{n-m}(q))+1) \\
&< |PGU_m(q)|^{\frac{3}{4}+\frac{c'}{m}}|PGU_{n-m}(q)|^{\frac{3}{4}+\frac{c''}{n-m}} \\
&< |M|^{\frac{3}{4}}q^{c'''n}.
\end{align*}
\noindent The remaining cases are proved similarly. \end{proof}

\begin{lemma} \label{i4class2}
Proposition \ref{i4Mprop1} holds for $M \in \mathscr{C}_2$.
\end{lemma}

\begin{proof}
\par Suppose $M \in \mathscr{C}_2$ is as in Table \ref{maxsubgroups}. First consider the case where $M$ is of the form $Cl_m(q) \wr S_t$ (see Table \ref{maxsubgroups}). Let $((g_1, \dots, g_t), \alpha) \in M$ be an element of order 4, where $g_i$ lies in $Cl_{m}(q)$ and $\alpha \in S_t$. Suppose $\alpha$ has $l_k$ $k$-cycles for $1 \le k \le t$. For $i \in \{1, \dots, t \}$, if $\alpha(i)=i$, then $g_i^4=\lambda$ for some scalar $\lambda$; if $\alpha(i)=j, \alpha(j)=i$, then $(g_i g_j)^2=\lambda$; if $i$ is contained in a 4-cycle of $\alpha$, then $g_ig_{\alpha(i)}g_{\alpha^2(i)}g_{\alpha^3(i)}=\lambda$. 

\par Let $j_4(S_t)$ denote the number of elements of order dividing 4 in $S_t$. From \cite{Wilf}, we have
\begin{align*}
j_4(S_t) \sim t! \frac{\tau^t}{\sqrt{8 \pi t}}\operatorname{exp}(\frac{1}{\tau}+\frac{1}{2\tau^2}+\frac{1}{4\tau^4}),
\end{align*}
\noindent where $\tau=\frac{1}{32}t^{-\frac{5}{4}}(1+8t^\frac{1}{4}+8t^\frac{1}{2}+32t)$. Since $t^{-\frac{1}{4}}<\tau<2t^{-\frac{1}{4}}$, we have
\begin{align*}
j_4(S_t)<c2^t t! t^{-\frac{t}{4}-\frac{1}{2}} \operatorname{exp}(t^\frac{1}{4}+2t^\frac{1}{2}+4t).
\end{align*}
\noindent for an absolute constant $c$. Using Stirling's formula, it is then easy to show 
\begin{equation} \label{6.1eqn1}
j_4(S_{t}) < |S_{t}|^\frac{3}{4}e^{c't}
\end{equation}
\noindent for some absolute constant $c'$.

\par For a scalar $\lambda$ let $i_{j, \lambda}(Cl_m(q))$ be the number of elements $g \in Cl_m(q)$ such that $g^j=\lambda$. Also, by Propositions \ref{conjclassprop1} and \ref{i4prop}, for absolute constants $c_i$ we have
\begin{align} 
i_{2, \lambda}(Cl_m(q)) < c_1 i_2(PCl_m(q)) < |PCl_m(q)|^{\frac{1}{2}+\frac{c_2}{m}},\label{6.1eqn2} \\
i_{4, \lambda}(Cl_m(q)) < c_3 i_4(PCl_m(q)) < |PCl_m(q)|^{\frac{3}{4}+\frac{c_4}{m}}. \label{6.1eqn3} 
\end{align}
\noindent Therefore, using (\ref{6.1eqn1})-(\ref{6.1eqn3}),
\begin{align*}
i_4(M) &<  j_4(S_t)  \sum_{\lambda} \sum_{l_1+2l_2+4l_4=t} i_{4, \lambda}(Cl_m(q))^{l_1} i_{2, \lambda}(Cl_m(q))^{l_2} |Cl_m(q)|^{3l_4+l_2} \\
&< |S_t|^\frac{3}{4}|Cl_m(q)|^{\frac{3}{4}(l_1+2l_2+4l_4)}q^{c''n} \\
&< |M|^\frac{3}{4}q^{c'''n}.
\end{align*}

\par Now suppose $M$ is of type $GL_\frac{n}{2}(q^\delta).2$. Here the result follows from Proposition \ref{i4prop}, as
\begin{align*}
i_4(M) <  c i_4(PGL_\frac{n}{2}(q^\delta).2) < |M|^{\frac{3}{4}+\frac{c'}{n}}
\end{align*}
\noindent for absolute constants $c, c'$. This completes the proof for the case $M \in \mathscr{C}_2$.
\end{proof}

\begin{lemma}
Proposition \ref{i4Mprop1} holds for $M \in \mathscr{C}_i, 3 \le i \le 8$, and for $M \in \mathscr{S}$.
\end{lemma}

\begin{proof}
\par If $M \in \mathscr{C}_3, \mathscr{C}_5$ or $\mathscr{C}_8$ as in Table \ref{maxsubgroups}, the result follows immediately from Proposition \ref{i4prop}.
\par For $M \in \mathscr{C}_4$ with $M \le Cl^1_d(q) \times Cl^2_e(q)$,
\begin{align*}
i_4(M) <(i_4(Cl^1_d(q))+ i_2(Cl^1_d(q))+1)(i_4(Cl^2_e(q))+ i_2(Cl^2_e(q))+1),
\end{align*} 
\noindent and using Propositions \ref{irprop} and \ref{i4prop} gives
\begin{align*}
i_4(M) &< |Cl^1_d(q)|^{\frac{3}{4}+\frac{c'}{d}} |Cl^2_e(q)|^{\frac{3}{4}+\frac{c''}{e}} < |M|^{\frac{3}{4}}q^{c'''n}.
\end{align*}
\par Suppose $M \in \mathscr{C}_6$. In each case it can easily be shown that $|M|<q^{cn}$ for some absolute constant $c$, and the result follows.
\par If $M \in \mathscr{C}_7$, then arguing as in Lemma \ref{i4class2} we can show $i_4(M)<|M|^\frac{3}{4}q^{cn}$ for some absolute constant $c$.
\par Finally suppose $M \in \mathscr{S}$. Then $M$ is an almost simple group acting absolutely irreducibly on $V$. If $\operatorname{soc}(M) \ne A_{n+1}, A_{n+2}$, then $|M|<q^{3n}$ by \cite{Li}. If $\operatorname{soc}(M) = A_{n+1}, A_{n+2}$, then $i_4(M)<|M|^\frac{3}{4}q^{cn}$ as in (\ref{6.1eqn1}) from the proof of Lemma \ref{i4class2}. The result follows. \end{proof}

\subsection{Parabolic maximal subgroups}

\par In this subsection we move towards proving Theorem \ref{bigthm1} by proving Proposition \ref{i4Mprop2}. Recall the definitions of $G, x \in A$ and $y \in B$. 

\par For the remainder of this section, we assume $M=P_m$ is a maximal parabolic subgroup of $G$ stabilizing a totally singular $m$-space for some $1 \le m \le \frac{n}{2}$. In the proof of Proposition \ref{i4Mprop2}, instead of considering $i_4(M)$ as in Proposition \ref{i4Mprop1}, we consider fixed points of $x$ and $y$ acting on $M^G=\{ M^g: g \in G\}$ in Lemmas \ref{fprx} and \ref{fpry} respectively. We then prove Proposition \ref{i4Mprop2} immediately afterwards.
\par For a $t$-dimensional vector space $W$ over $\mathbb{F}_{q^\delta}$ equipped with a zero, symplectic, orthogonal or unitary form (where $\delta=2$ if the form is unitary and $\delta=1$ otherwise), let $p_m(W)$ denote the number of totally singular $m$-subspaces of $W$. One can easily calculate 
\begin{align} 
\label{parabolicindex}  p_m(W) &\sim \left\{ 
  \begin{array}{l l}
q^{m(t-m)} \hspace{0.3cm} &\text{if $W$ is linear,}   \\ 
q^{m(t-\frac{3m}{2}+\frac{1}{2})} &\text{if $W$ is symplectic,} \\ 
q^{m(t-\frac{3m}{2}-\frac{1}{2})} &\text{if $W$ is orthogonal,}  \\
q^{m(2t-3m)} \ &\text{if $W$ is unitary.}  \\
\end{array}
\right.   
\end{align}
\par For an element $g \in GL(W)$ stabilizing a subspace $U \subset W$, write $g^U$ for the restriction of $g$ to $U$. Furthermore, if a basis $B$ of $U$ is specified, write $[g^U]_B$ for the matrix of $g^U$ with respect to $B$.
\par For positive integers $s,i$, define $J_{s,i} \in GL_{si}(q)$ as
\begin{equation} \label{unipotentelt}
J_{s,i} = 
\begin{pmatrix}
I_{s} \\
I_{s} & I_{s} \\
& I_{s} & I_{s} \\
&  & \ddots \\
& &  & I_{s} & I_{s}
\end{pmatrix}.
\end{equation}
\noindent It will be useful in the proofs of Lemmas \ref{fprx} and \ref{fpry} to note that from \cite[Theorem~7.1]{LiSe}, for $\sum_{1 \le i \le 4} il_i=m$, we have
\begin{equation} \label{unipotentcent}
|C_{GL_{m}(q)}\big( \bigoplus_{1 \le i \le 4} J_{l_i,i} \big)| \sim q^{\sum_{1 \le i \le 4} il_i^2+2\sum_{i<j}il_il_j}.
\end{equation}

\par Let $T$ be a finite group acting transitively on a set $\Omega$. For $t \in T$, define
\begin{equation*}
\operatorname{fix}(t,\Omega) = | \{ \omega \in \Omega:\omega t=\omega \} |.
\end{equation*}  
\noindent The fixed point ratio of $t \in T$ is defined as
\begin{equation*}
\operatorname{fpr}(t,\Omega)=\frac{\operatorname{fix}(t,\Omega) }{| \Omega |}.
\end{equation*}
\noindent For $\omega \in \Omega$, let $S=T_\omega$. An elementary argument counting pairs $\{ (\omega,s): \omega \in \Omega, s \in t^T, \omega s=\omega\}$ in two different ways shows we can also express the fixed point ratio as
\begin{equation} \label{fpreqn}
\operatorname{fpr}(t,\Omega) = \frac{| t^T \cap S |}{| t^T |}.
\end{equation}

\subsubsection{Fixed point ratios of almost-free involutions and elements of order 4}
\par Recall that $G$ is a finite simple classical group with natural module of dimension $n$ over $\mathbb{F}_{q^\delta}$, $A$ is a nontrivial 2-group embedded almost-freely into $G$ with $|A|=a, n=k_aa+s_a$ and $2 \le s < 2a+2$, and $x \in I_{2}(A)$.
\par We first bound the fixed point ratio of $x$.
\begin{lemma} \label{fprx}
Let $M=P_m$ be a maximal parabolic subgroup of $G$. Then
\begin{align*}
\operatorname{fpr}(x,M^G)  < q^{-f+cm}
\end{align*}
\noindent where $c=c(A)$ is a constant depending only on $A$, and 
\begin{align*}
 f=\left\{ 
  \begin{array}{l l}
\frac{1}{2}m(n-m) \hspace{0.3cm} &\text{if $G=PSL_n(q)$,} \\[0.2cm]
\frac{1}{4}m(2n-3m) &\text{if $G=PSp_n(q), P\Omega_n(q)$,} \\[0.2cm]
\frac{1}{2}m(2n-3m) &\text{if $G=PSU_n(q)$}.
\end{array}
\right. 
\end{align*}
\end{lemma}

\begin{proof}
\par If $V$ is the natural module of $G$, observe that $\operatorname{fix}(x,M^G)$ is less than or equal to the number of totally singular $m$-spaces of $V$ fixed by $x$. We count the number of possible totally singular $m$-subspaces $U \subset V$ invariant under $x$.
\par Throughout, write $k=k_a, s=s_a$ for ease of notation, so that $n=ka+s$.

\paragraph{Case $G=PSL_n(q).$} 
First suppose $q$ is odd. We have $V \downarrow x =E_{-1} \oplus E_1$, where $E_i$ is the $i$-eigenspace of $x$ for $i=\pm 1$, and $E_{-1}, E_1$ has dimension $\frac{ka}{2}, \frac{ka}{2}+s$ respectively. Hence for any $m$-subspace $U \subset V$ stabilized by $x$ we have $U=U_{-1} \oplus U_1$, where $U_{-1}$ is an $l$-subspace of  $E_{-1}$ and $U_1$ is an $(m-l)$-subspace of $E_1$, for some $0 \le l \le m$. Therefore, using (\ref{parabolicindex}) we have
\begin{align*}
\operatorname{fix}(x,M^G) &\le \sum_{0 \le l \le m} p_l(E_{-1}) p_{m-l}(E_1) \\
&< \sum_{0 \le l \le m} cq^{l(\frac{ka}{2}-l)+(m-l)(\frac{ka}{2}+s-(m-l))} \\
& <q^{\frac{1}{2}m(n-m)+c'm}
\end{align*}
\noindent for some constant $c'=c'(A)$ depending only on $A$. Hence
\begin{align*}
\operatorname{fpr}(x,M^G) < \frac{q^{\frac{1}{2}m(n-m)+c'm}}{c''q^{m(n-m)}} < q^{-\frac{1}{2}m(n-m)+c'''m}.
\end{align*}

\par Now suppose $q$ is even. With the almost-free embedding, there exists a basis $\{ e_i \}$ of $V$ with respect to which
\begin{align*}
x=\begin{pmatrix}
J_{\frac{ka}{2},2} \\
&  I_{s} 
\end{pmatrix}.
\end{align*}
\noindent Suppose $x$ stabilizes an $m$-subspace $U \subset V$. There exists a basis $\beta=\{ u_i \}_{i=1}^m$ of $U$ such that  
\begin{align} \label{xBeven}
[x^U]_\beta=\begin{pmatrix}
J_{l,2} \\ 
& I_{m-2l} 
\end{pmatrix}.
\end{align}
\noindent for some $0 \le l \le \frac{m}{2}$.
\par Let $u_i=\sum_{j=1}^n \alpha_{ij}e_j$ for $\alpha_{ij} \in \mathbb{F}_q$, and let $\alpha=(\alpha_{ij})$. Then 
\begin{equation} \label{xalphaeqn}
[x^U]_\beta^T \alpha = \alpha x^T.
\end{equation}
\noindent Therefore, we can write
\begin{align*}
\alpha = \begin{pmatrix}
A_1 & A_2 & B_1\\ & A_1 \\ & A_3 & B_2
\end{pmatrix}
\end{align*}
\noindent where
\begin{align*}
A_1, A_2 &\in M_{l,\frac{ka}{2}}(q), \\
 A_3 &\in M_{m-2l, \frac{ka}{2}}(q), \\
B_1 &\in M_{l, s}(q), \\
B_2 &\in M_{m-2l, s}(q).
\end{align*}
\par Let $\Delta(l)$ be the number of matrices of the form of $\alpha$ as above, so 
\begin{equation*}
\Delta(l)=q^{\frac{nm}{2}+(\frac{m}{2}-l)s}
\end{equation*}
\noindent Then $\Delta(l)$ is an upper bound for the number of bases $\beta$ of an $m$-subspace $U$ stabilized by $x$, with $x$ acting as (\ref{xBeven}).
\par Once $\alpha$ is fixed, $\beta$ (and hence $U$) is also fixed. Observe that $C_{GL_m(q)}([x^U]_\beta)$ acts regularly on the set of bases $\beta'$ of $U$ such that $[x^U]_\beta=[x^U]_{\beta'}$, and so the number of such bases is $|C_{GL_m(q)}([x^U]_\beta)|$. Hence, using (\ref{unipotentcent}),
\begin{align*}
\operatorname{fix}(x,M^G) <\sum_{0 \le l \le \frac{m}{2}} \frac{\Delta(l)}{|C_{GL_m(q)}([x^U]_\beta)|}<\sum_{0 \le l \le \frac{m}{2}} cq^{\frac{nm}{2}+(\frac{m}{2}-l)s-2l^2-(m-2l)^2-2l(m-2l)}<q^{\frac{nm}{2}-\frac{m^2}{2}+c'm}.
\end{align*}
 \noindent for a constant $c'=c'(A)$ depending only on $A$. Therefore
\begin{align*}
\operatorname{fpr}(x,M^G)<\frac{q^{\frac{nm}{2}-\frac{m^2}{2}+c'm}}{c''q^{m(n-m)}}<q^{-\frac{1}{2}m(n-m)+c'''m},
\end{align*}
\noindent completing the proof in the linear case.

\paragraph{Case $G=PSp_n(q).$}

First suppose $q$ is odd. We have $V \downarrow x=E_{-1} \perp E_1$, where $E_{-1}, E_1$ are as in the linear case. Hence for any $m$-subspace $U \subset V$ stabilized by $x$ we have $U=U_{-1} \perp U_1$, where $U_{-1}, U_1$ are totally singular $l, (m-l)$-subspaces of $E_{-1}, E_1$ respectively, for some $0 \le l \le m$. Therefore, by (\ref{parabolicindex}) we have
\begin{align*}
\operatorname{fix}(x,M^G) &\le \sum_{0 \le l \le m} p_l(E_{-1}) p_{m-l}(E_1) \\
&< \sum_{0 \le l \le m} cq^{l(\frac{ka}{2}-\frac{3l}{2}+\frac{1}{2})+(m-l)(\frac{ka}{2}+s-\frac{3(m-l)}{2}+\frac{1}{2})} \\
& <q^{\frac{1}{4}m(2n-3m)+c'm}
\end{align*}
\noindent for some constant $c'=c'(A)$ depending only on $A$. The result now follows as in the linear case using (\ref{parabolicindex}).

\par Now suppose $q$ is even. With the almost-free embedding, there exists a basis $\{ e_i \}$ of $V$ with Gram matrix 
\begin{equation*}
\begin{pmatrix} & I_{\frac{ka}{2}}  \\  I_{\frac{ka}{2}} \\ & & D \end{pmatrix}
\end{equation*} 
\noindent (where $D$ is a Gram matrix for an $s$-dimensional symplectic space) with respect to which
\begin{align} \label{xeven2}
x=\begin{pmatrix}
J_{\frac{ka}{4},2} \\
& J_{\frac{ka}{4},2}^T \\
& &  I_{s} 
\end{pmatrix}.
\end{align}
\noindent Suppose $x$ stabilizes a totally singular $m$-subspace $U \subset V$. There exists a basis $\beta=\{ u_i \}_{i=1}^m$ of $U$ such that $[x^U]_\beta$ is as in (\ref{xBeven}) for some $0 \le l \le \frac{m}{2}$.
\par Let $u_i=\sum_{j=1}^n \alpha_{ij}e_j$ for $\alpha_{ij} \in \mathbb{F}_q$, and let $\alpha=(\alpha_{ij})$. Then
\begin{equation*}
[x^U]_\beta^T \alpha = \alpha x^T.
\end{equation*}
\noindent Therefore, we can write
\begin{align*}
\alpha = \begin{pmatrix}
A_1 & A_2 & B_1 & B_2 & C_1\\ & A_1 & B_2 &  \\ & A_3 & B_3 & & C_2
\end{pmatrix}
\end{align*}
\noindent where 
\begin{align*}
A_1, A_2, B_1, B_2 &\in M_{l,\frac{ka}{4}}(q), \\
 A_3, B_3 &\in M_{m-2l, \frac{ka}{4}}(q), \\
 C_1 &\in M_{l, s}(q), \\ 
C_2 &\in M_{m-2l, s}(q).
\end{align*}
\par Let $\Delta(l)$ be the number of matrices of the form of $\alpha$ as above with linearly independent rows $u_i$ such that the subspace of $V$ spanned by $\{ u_i \}$ is totally singular. The number of choices for rows $u_i$ ($2l+1 \le i \le m$) of $\alpha$ is less than $q^{(m-2l)(\frac{ka}{2}+s)}$, as this gives
\begin{equation*}
\begin{pmatrix} 0 & A_3 & B_3 & 0 & C_2 \end{pmatrix}.
\end{equation*}
\noindent Given the above rows, we now count the number of choices for the remaining rows. Let $W_0$ be the subspace of $V$ spanned by the $u_i$ ($2l+1 \le i \le m$). Define a linear map $\theta:V \rightarrow V$ by
\begin{align} \label{thetaeqn}
\sum_{j=1}^{n} \gamma_j e_j &\mapsto \sum_{j=1}^{\frac{ka}{4}} \gamma_{j}e_{j+\frac{ka}{4}} +  \sum_{j=\frac{3ka}{4}+1}^{ka} \gamma_{j}e_{j-\frac{ka}{4}}.
\end{align}
\noindent Observe that for any $v_1,v_2 \in V$ we have $(v_1,  v_1\theta)=0$, $(v_1, v_2 \theta)=(v_1\theta, v_2)$ and $v_1\theta^2=0$. Inductively, for $1 \le i \le l$, choose $u_i \in W_{i-1}^\perp$ such that $\dim \left( \langle W_{i-1},u_i, u_i \theta \rangle \right)=\dim(W_{i-1})+2$; then define $W_i=\langle W_{i-1}, u_i, u_i\theta \rangle \subset V$ (note that $u_i \in W_{i-1}^\perp$ implies $u_i\theta \in W_{i-1}^\perp$ since $(u_i \theta,w)=(u_i,w \theta)=0$ for all $w \in W_{i-1}$). Then, for $1 \le i \le l$, $u_i$ are the rows of \begin{equation*}
\begin{pmatrix} A_1 & A_2 & B_1 &  B_2 & C_1 \end{pmatrix},
\end{equation*}
\noindent and $u_i \theta$ are the rows of 
\begin{equation*}
\begin{pmatrix} 0 & A_1 & B_2 & 0 & 0 \end{pmatrix}.
\end{equation*}
\noindent For $1 \le i \le l$, the number of choices for $u_i$ is less than $q^{n-(m-2l+2i-2)}$ since $\dim (W_{i-1})=m-2l+2i-2$. Therefore we have
\begin{align*}
\Delta(l)< q^{\left( \sum_{1 \le i \le l} n-(m-2l+2i-2) \right) + (m-2l)(\frac{ka}{2}+s)} = q^{\frac{nm}{2}+(\frac{m}{2}-l)s-l(m-l-1)}.
\end{align*}
\noindent Once $\alpha$ is fixed, $\beta$ (and hence $U$) is also fixed. As in the linear case, the number of bases $\beta'$ of $U$ such that $[x^U]_\beta=[x^U]_{\beta'}$ is $|C_{GL_m(q)}([x]_\beta)|$. Hence,
\begin{align} \label{xfixSpeven}
\operatorname{fix}(x,M^G)&<\sum_{0 \le l \le \frac{m}{2}} \frac{\Delta(l)}{|C_{GL_m(q)}([x]_\beta)|} \nonumber \\ &<\sum_{0 \le l \le \frac{m}{2}} cq^{\frac{nm}{2}+(\frac{m}{2}-l)s-l(m-l-1)-2l^2-(m-2l)^2-2l(m-2l)} \nonumber \\ &<q^{\frac{nm}{2}-\frac{3m^2}{4}+c'm}.
\end{align}
 \noindent for a constant $c'=c'(A)$ depending only on $A$. Therefore
\begin{align*}
\operatorname{fpr}(x,M^G)<\frac{q^{\frac{nm}{2}-\frac{3m^2}{4}+c'm}}{c''q^{m(n-\frac{3m}{2}+\frac{1}{2})}}<q^{-\frac{1}{4}m(2n-3m)+c'''m},
\end{align*}
\noindent for a constant $c'''=c'''(A)$ depending only on $A$, completing the proof in the symplectic case.

\paragraph{Case $G=P\Omega^\epsilon_n(q).$}
For $q$ odd, the proof is similar to the symplectic case.
\par For $q$ even, with the almost-free embedding there exists a basis $\{ e_i \}$ of $V$ such that $x$ is as in (\ref{xeven2}) and the quadratic form $Q$ on $V$ is defined as
\begin{equation*}
Q(\sum_{i=1}^n \lambda_i e_i) = \sum_{i=1}^\frac{ka}{2} \lambda_i \lambda_{i+\frac{ka}{2}} + Q'(\sum_{i=ka+1}^n \lambda_i e_i),
\end{equation*}
\noindent where $Q'$ is a quadratic form of type $\epsilon$ on the subspace $\langle e_{ka+1}, \dots, e_n \rangle \subset V$. This basis yields the Gram matrix 
\begin{equation*}
\begin{pmatrix} & I_{\frac{ka}{2}}  \\  I_{\frac{ka}{2}} \\ & & D \end{pmatrix}
\end{equation*}
\noindent (where $D$ is a Gram matrix for an $s$-dimensional orthogonal space of type $\epsilon$). 
\par If $x$ stabilizes a totally singular $m$-subspace of $U \subset V$, then $U$ is totally singular with respect to the symplectic form on $V$ determined by $Q$. Hence from (\ref{xfixSpeven}),
\begin{equation*}
\operatorname{fix}(x,M^G)<q^{\frac{nm}{2}-\frac{3m^2}{4}+cm}
\end{equation*}
 \noindent for a constant $c=c(A)$ depending only on $A$. Therefore
\begin{align*}
\operatorname{fpr}(x,M^G)<\frac{q^{\frac{nm}{2}-\frac{3m^2}{4}+c'm}}{c''q^{m(n-\frac{3m}{2}-\frac{1}{2})}}<q^{\frac{1}{4}m(3m-2n)+c'''m},
\end{align*}
\noindent for a constant $c'''=c'''(A)$ depending only on $A$, proving the result in the orthogonal case.

\paragraph{Case $G=PSU_n(q).$}

\par The proof in this case is similar to the symplectic case. \end{proof}

\par Recall that $G$ is a finite simple classical group with natural module of dimension $n$, $B$ is a nontrivial 2-group embedded almost-freely into $G$ with $|B|=b, n=k_bb+s_b$ and $2 \le s_b <2b+2$, and $y \in I_{4}(B)$. We now bound $\operatorname{fix}(y, M^G)$.

\begin{lemma} \label{fpry}
Let $M=P_m$ be a maximal parabolic subgroup of $G$. Then
\begin{align*}
\operatorname{fix}(y,M^G)  < q^{f+cm}
\end{align*}
\noindent where $c=c(B)$ is a constant depending only on $B$, and 
\begin{align*}
 f=\left\{ 
  \begin{array}{l l}
\frac{1}{4}m(n-m) \hspace{0.3cm} &\text{if $G=PSL_n(q)$,} \\[0.2cm]
\frac{1}{8}m(2n-3m) &\text{if $G=PSp_n(q), P\Omega_n(q)$,} \\[0.2cm]
\frac{1}{4}m(2n-3m) &\text{if $G=PSU_n(q)$}.
\end{array}
\right. 
\end{align*}
\end{lemma}

\begin{proof}
\par The approach is similar to that of Lemma \ref{fprx}. If $V$ is the natural module of $G$, observe that $\operatorname{fix}(y,M^G)$ is less than or equal to the number of totally singular $m$-spaces $U \subset V$ invariant under $y$. As in Lemma \ref{fprx}, we count the number of possible $U$.
\par For ease of notation, let $k=k_b, s=s_b$, so that $n=kb+s$.
\paragraph{Case $G=PSL_n(q).$} 
First suppose $q$ is odd. Then $V \downarrow y = E_{i,-i} \oplus E_{-1} \oplus E_1$, where $E_{\pm 1}$ is the $\pm 1$-eigenspace of $y$, with $E_{-1}, E_{1}$ of dimension $\frac{kb}{4}, \frac{kb}{4}+s$ respectively, and over $\mathbb{F}_{q^2}$, $E_{i,-i}=E_i \oplus E_{-i}$ where $E_{\pm i}$ is the $\pm i$-eigenspace of $y$, each of dimension $\frac{kb}{4}$. For any $m$-subspace $U \subset V$ stabilized by $y$ we have $U=U_{i,-i} \oplus U_{-1} \oplus U_1$, where $U_j \subset E_j$ is an $l_j$-subspace for $j= \pm 1$, and $U_{i,-i} \subset E_{i,-i}$; moreover, over $\mathbb{F}_{q^2}$, $U_{i,-i}=U_i \oplus U_{-i}$ with $U_{j} \subset E_{j}$ an $l_{j}$-subspace for $j=\pm i$ such that $\sum_{j=\pm i, \pm 1} l_j = m$.
\par If $q \equiv 1 \mod 4$, then $E_i$ and $E_{-i}$ are defined over $\mathbb{F}_q$, and so the number of subspaces $U_{i,-i}$ of dimension $l_i+l_{-i}$ is $p_{l_i}(E_i)p_{l_{-i}}(E_{-i}$); if $q \equiv 3 \mod 4$, then we have $l_i=l_{-i}$, and the number of $2l_i$-subspaces of $E_i \oplus E_{-i}$ defined over $\mathbb{F}_q$ is $p_{l_i}(E_i)$ (where $E_{\pm i}$ is defined over $\mathbb{F}_{q^2}$). Therefore, using (\ref{parabolicindex}) we have
\begin{align*}
\operatorname{fix}(y,M^G) &\le \left\{ 
  \begin{array}{l l}
\sum_{\sum l_j =m} p_{l_i}(E_i) p_{l_{-i}}(E_{-i}) p_{l_{-1}}(E_{-1})p_{l_1}(E_1) \hspace{0.3cm} &\text{if $q \equiv 1 \mod 4$,} \\[0.2cm]
\sum_{2l_i+l_{-1}+l_1=m} p_{l_i}(E_{i}) p_{l_{-1}}(E_{-1})p_{l_1}(E_1) &\text{if $q \equiv 3 \mod 4$,}  \\[0.2cm]
\end{array}
\right. \\
&< \left\{ 
  \begin{array}{l l}
\sum_{\sum l_j =m} cq^{l_i(\frac{kb}{4}-l_i)+l_{-i}(\frac{kb}{4}-l_{-i})+l_{-1}(\frac{kb}{4}-l_{-1})+l_1(\frac{kb}{4}+s-l_1)} \\[0.2cm]
 \sum_{2l_i+l_{-1}+l_1=m} cq^{2l_i(\frac{kb}{4}-l_i)+l_{-1}(\frac{kb}{4}-l_{-1})+l_1(\frac{kb}{4}+s-l_1)}   \\[0.2cm]
\end{array}
\right.  \\
&< q^{\frac{1}{4}m(n-m)+c'm}
\end{align*}
\noindent for some absolute constant $c'$, completing the case for $q$ odd.

\par Now suppose $q$ is even. With an almost-free embedding, there exists a basis $\{ e_i \}$ of $V$ such that
\begin{align*}
y=\begin{pmatrix}
J_{\frac{kb}{4},4} \\
&  I_{s} 
\end{pmatrix}.
\end{align*}
\noindent Suppose $y$ stabilizes an $m$-subspace $U \subset V$. There exists a basis $\beta=\{ u_i \}_{i=1}^m$ of $U$ such that  
\begin{align} \label{yBeven}
[y^U]_\beta=\begin{pmatrix}
J_{l_4,4} \\ 
& J_{l_3,3} \\
& & J_{l_2,2} \\
& & & I_{l_1} 
\end{pmatrix}
\end{align}
\noindent where $\sum il_i=m$.
\par Let $u_i=\sum_{j=1}^n \alpha_{ij}e_j$ for $\alpha_{ij} \in \mathbb{F}_q$, and let $\alpha=(\alpha_{ij})$. Then 
\begin{align*}
[y^U]_\beta^T \alpha = \alpha y^T.
\end{align*} 
\noindent Therefore, we can write
\begin{align*}
\alpha = \begin{pmatrix}
A_1 & A_2 & A_3 & A_4 & B_4\\ & A_1 & A_2 & A_3 \\ & & A_1 & A_2  \\ & & & A_1 \\ &  A_5 & A_6 & A_7 & B_3 \\ & & A_5 & A_6 \\ & & & A_5 \\ & & A_8 & A_9 & B_2 \\ & & & A_8 \\ & & & A_{10} & B_1
\end{pmatrix}
\end{align*}
\noindent where 
\begin{align*}
A_1, \dots, A_4 &\in M_{l_4,\frac{kb}{4}}(q), \\
 A_5, A_6, A_7 &\in M_{l_3, \frac{kb}{4}}(q), \\
A_8, A_9 &\in M_{l_2, \frac{kb}{4}}(q), \\
A_{10} &\in M_{l_1, \frac{kb}{4}}(q) , \\
B_j &\in M_{l_j, s}(q).
\end{align*}
\par Let $\Delta(l_1,l_2,l_3,l_4)$ be the number of matrices of the form of $\alpha$ as above, so 
\begin{equation*}
\Delta(l_1,l_2,l_3,l_4)=q^{\frac{nm}{4}+(\sum l_j-\frac{m}{4})s}
\end{equation*}
\noindent Once $\alpha$ is fixed, $\beta$ and $U$ are also fixed. As in the proof of Lemma \ref{fprx}, the number of bases $\beta'$ of $U$ such that $[y^U]_{\beta'}=[y^U]_\beta$ is $|C_{GL_m(q)}([y^U]_\beta)|$. Observe that $\sum il_i^2 -2\sum_{i<j}il_il_j \ge \frac{1}{4}m^2$. Hence, using (\ref{unipotentcent}),
\begin{align*}
\operatorname{fix}(y,M^G)<\sum_{\sum il_i=m} \frac{\Delta(l)}{|C_{GL_m(q)}([y^U]_\beta)|}<\sum_{\sum il_i=m} cq^{\frac{nm}{4}+(\sum l_i-\frac{m}{4})s-\sum il_i^2 -2\sum_{i<j}il_il_j}<q^{\frac{1}{4}m(n-m)+cm}
\end{align*}
\noindent for a constant $c=c(B)$ depending only on $B$, proving the result in the linear case.

\paragraph{Case $G=PSp_n(q).$}

First suppose $q$ is odd. As in the linear case, we have $V \downarrow y=E_{i,-i} \perp E_{-1} \perp E_1$, where $E_{i,-i}, E_{\pm 1}$ are non-degenerate. If $U \subset V$ is a totally singular $m$-space stabilized by $y$, then $U=U_{i,-i} \perp U_{-1} \perp U_1$, where $U_{\pm 1} \subset E_{\pm 1}$ are totally singular $l_{\pm 1}$-subspaces, and $U_{i,-i} \subset E_{i,-i}$ is totally singular.
\par First suppose $q \equiv 1 \mod 4$, so that $E_{i,-i}=E_i \oplus E_{-i}$ (where $E_{\pm i}$ is totally singular) and $U_{i,-i}=U_i \oplus U_{-i}$ where $U_{\pm i} \subset E_{\pm i}$ is an $l_{\pm i}$ subspace, with $\sum_{j=\pm 1, \pm i} l_j=m$. Observe that for an $l_i$-dimensional subspace $U_i \subset E_i$ we have $\dim(E_{-i} \cap U_i^\perp)=\frac{kb}{4}-l_i$. Also note that $l_i^2+l_{-i}^2+l_il_{-i}+\frac{3}{2}l_{-1}^2+\frac{3}{2}l_1^2 \ge \frac{3}{8}m^2$. Therefore, using (\ref{parabolicindex}),
\begin{align*}
\operatorname{fix}(y,M^G) &\le \sum_{\sum l_j=m} \sum_{\substack{U_i \subset E_i \\ \dim(U_i)=l_i}} p_{l_{-i}}(E_{-i} \cap U_{i}^\perp) p_{l_{-1}}(E_{-1}) p_{l_1}(E_1) \\
&< \sum_{\sum l_j=m} cq^{l_{-i}(\frac{kb}{4}-l_i-l_{-i})+l_{-1}(\frac{kb}{4}-\frac{3}{2}l_{-1}+\frac{1}{2})+l_{1}(\frac{kb}{4}+s-\frac{3}{2}l_{1}+\frac{1}{2})} p_{l_i}(E_i) \\
&< \sum_{\sum l_j=m}  c'q^{l_i(\frac{kb}{4}-l_i)+l_{-i}(\frac{kb}{4}-l_i-l_{-i})+l_{-1}(\frac{kb}{4}-\frac{3}{2}l_{-1}+\frac{1}{2})+l_{1}(\frac{kb}{4}+s-\frac{3}{2}l_{1}+\frac{1}{2})} \\
&< q^{\frac{1}{8}m(2n-3m)+c''m}
\end{align*}
\noindent for a constant $c''=c''(B)$ depending only on $B$.

\par Now suppose $q \equiv 3 \mod 4$. In this case $E_{i,-i}$ is isomorphic (as an $\mathbb{F}_q \langle y \rangle$-module) to a unitary space of dimension $\frac{kb}{4}$ over $\mathbb{F}_{q^2}$, with $y \in Z(\overline{GU_{\frac{kb}{4}}(q)})$. Totally singular $2l_{i,-i}$- subspaces of $E_{i,-i}$ stabilized by $y$ correspond to totally singular $l_{i,-i}$-subspaces of the unitary space. Using (\ref{parabolicindex}) and the Cauchy-Schwarz inequality, we have
\begin{align*}
\operatorname{fix}(y,M^G) &\le \sum_{2l_{i,-i}+l_{-1}+l_1=m}  p_{l_{i,-i}}(E_{i,-i})p_{l_{-1}}(E_{-1}) p_{l_1}(E_1) \\
&<\sum_{2l_{i,-i}+l_{-1}+l_1=m}  cq^{l_{i,-i}(\frac{kb}{2}-3l_{i,-i})+l_{-1}(\frac{kb}{4}-\frac{3}{2}l_{-1}+\frac{1}{2})+l_{1}(\frac{kb}{4}+s-\frac{3}{2}l_{1}+\frac{1}{2})} \\
&< q^{\frac{1}{8}m(2n-3m)+c'm}
\end{align*}
\noindent for a constant $c'=c'(B)$ depending only on $B$.

\par Now suppose $q$ is even. With the almost-free embedding, there exists a basis $\{ e_i \}$ of $V$ with Gram matrix 
\begin{equation*}
\begin{pmatrix} & I_{\frac{kb}{2}}  \\  I_{\frac{kb}{2}} \\ & & D \end{pmatrix}
\end{equation*}
\noindent (where $D$ is a Gram matrix of an $s$-dimensional symplectic space) such that
\begin{align} \label{yeven2}
y=\begin{pmatrix}
J_{\frac{kb}{8},4} \\
& J_{\frac{kb}{8},4}^{-T} \\
& &  I_{s} 
\end{pmatrix}.
\end{align}
\noindent Suppose $y$ stabilizes a totally singular $m$-subspace $U \subset V$. There exists a basis $\beta=\{ u_i \}_{i=1}^m$ of $U$ such that $[y^U]_\beta$ is as in (\ref{yBeven}).
\par Let $u_i=\sum_{j=1}^n \alpha_{ij}e_j$ for $\alpha_{ij} \in \mathbb{F}_q$, and let $\alpha=(\alpha_{ij})$. As above, $[y^U]_\beta^T \alpha=\alpha y^T$. Therefore, we can write
\begin{align}  \label{yalpha2}
\alpha = \begin{pmatrix}
A_1 & A_2 & A_3 & A_4 & B_1 & B_2 & B_3 & B_4 & C_4 \\ & A_1 & A_2 & A_3 &  B_2+B_3+B_4 & B_3+B_4 & B_4 & \\ & & A_1 & A_2 &  B_3 & B_4 &  \\ & & & A_1 & B_4 \\ &  A_5 & A_6 & A_7 & B_5 & B_6 & B_7 & & C_3 \\ & & A_5 & A_6 & B_6+B_7 & B_7 \\ & & & A_5 & B_7 \\ & & A_8 & A_9 & B_8 & B_9 & & & C_2 \\ & & & A_8 & B_9 \\ & & & A_{10} & B_{10} & & & & C_1
\end{pmatrix}
\end{align}
\noindent where 
\begin{align*} 
A_1, \dots, A_4, B_1, \dots, B_4 &\in M_{l_4,\frac{kb}{8}}(q), \\
 A_5, A_6, A_7, B_5, B_6, B_7 &\in M_{l_3, \frac{kb}{8}}(q), \\
 A_8, A_9, B_8, B_9 &\in M_{l_2, \frac{kb}{8}}(q), \\
A_{10}, B_{10} &\in M_{l_1, \frac{kb}{8}}(q), \\
 C_i &\in M_{l_i, s}(q).
\end{align*}
\par Let $\Delta(l_1,l_2,l_3,l_4)$ be the number of matrices of the form of $\alpha$ as above with linearly independent rows $u_i$ such that the subspace $U \subset V$ spanned by $\{ u_i \}$ is totally singular. Define the following subspaces of $V$:
\begin{align*}
V_0=\langle  e_{\frac{kb}{8}+1}, \dots, e_{\frac{7kb}{8}} \rangle, \ V_1=\langle e_{\frac{3kb}{8}+1}, \dots, e_{\frac{5kb}{8}} \rangle,
\end{align*}
\noindent so that $V_0/V_1$ is a $\frac{kb}{2}$-dimensional symplectic space. Fix the rows $u_i$  ($m-2l_2-l_1+1 \le i \le m$) of $\alpha$ corresponding to $A_8, A_9, A_{10}, B_8, B_9, B_{10},C_1,C_2$. The number of choices for these rows is less than
\begin{align} \label{sum1}
q^{(\frac{kb}{4}+s)(l_1+2l_2)}.
\end{align}
\noindent Given these rows, we compute the number of choices for the remaining rows. Let $U_0 \subset V_0$ be the span of the rows of 
\begin{align*}
\begin{pmatrix} 0& A_8 & A_9 & B_8 & B_9 & 0 \end{pmatrix},
\end{align*}
\noindent and define $W_0=(U_0+V_1)/V_1$. By the assumption that these rows are linearly independent, $\dim(W_0)=l_2$. Define a linear map $\theta:V \rightarrow V$ by
\begin{align*}  
\sum_{j=1}^{n} \gamma_j e_j &\mapsto \sum_{j=\frac{kb}{8}+1}^{\frac{kb}{2}} \gamma_{j-\frac{kb}{8}}e_{j} +  \sum_{j=\frac{kb}{2}+1}^{\frac{5kb}{8}} (\gamma_{j+\frac{kb}{8}}+\gamma_{j+\frac{kb}{4}}+\gamma_{j+\frac{3kb}{8}})e_{j} \\ &\hspace{2cm} +\sum_{j=\frac{5kb}{8}+1}^{\frac{3kb}{4}} (\gamma_{j+\frac{kb}{8}}+\gamma_{j+\frac{kb}{4}})e_{j}+\sum_{j=\frac{3kb}{4}+1}^{\frac{7kb}{8}} \gamma_{j+\frac{kb}{8}} e_{j}.
\end{align*}
\noindent Observe that $V_1 \theta =0, V_0 \theta \subset V_0$, and so $\theta: V_0/V_1 \rightarrow V_0/V_1$ is well-defined. Also observe that for $v_1, v_2 \in V_0$, we have $(v_1, v_1 \theta)=0$ and $(v_1 \theta, v_2)=(v_1,v_2 \theta)$. Inductively, for $4l_4+1 \le i \le 4l_4+l_3$, choose $w_i \in (W_{i-4l_4-1})^\perp \subset V_0/V_1$ such that
\begin{align*}
\dim(\langle W_{i-4l_4-1},w_i, w_i\theta \rangle)=\dim(W_{i-4l_4-1})+2,
\end{align*}
\noindent and define $W_{i-4l_4}=\langle W_{i-4l_4-1}, w_i, w_i\theta \rangle \subset V_0/V_1$ (note that $(w_i, w_i \theta)=0$, and that $w_i \in (W_{i-4l_4-1})^\perp$ implies $w_i \theta \in (W_{i-4l_4-1})^\perp$ since $(w_i \theta, w)=(w_i, w \theta)=0$ for all $w \in W_{i-4l_4-1}$). Then this yields a totally singular subspace $W_{l_3}$ of dimension $l_2+2l_3$, with 
\begin{align*}
W_{l_3}=\langle W_0, w_{4l_4+1}, \dots, w_{4l_3+l_3}, w_{4l_4+1} \theta, \dots, w_{4l_4+l_3} \theta \rangle.
\end{align*}
\noindent For each $w_i$ ($4l_4+1 \le i \le 4l_4+l_3$), choose a preimage $u_i'$ in $V_0$. This yields the rows $u_i'$ of 
\begin{align*}
\begin{pmatrix} A_5 & A_6 & A_7 &  B_5 & B_6 & B_7 \end{pmatrix},
\end{align*}
\noindent rows $u_i' \theta$ of 
\begin{align*}
\begin{pmatrix} 0 & A_5 & A_6&  B_6 & B_7 & 0 \end{pmatrix},
\end{align*} 
\noindent and rows $u_i' \theta^2$ of  
\begin{align*}
\begin{pmatrix} 0 & 0 & A_5 &  B_7 & 0 & 0\end{pmatrix}.
\end{align*}
\noindent Finally, choose $C_3$ with rows $c_i$ ($4l_4+1 \le i \le 4l_4+l_3$) such that the rows
\begin{align*}
u_i=\begin{pmatrix} 0 & u_i ' & 0 & c_i \end{pmatrix}
\end{align*}
\noindent of
\begin{align*}
\begin{pmatrix} 0 & A_5 & A_6 & A_7 &  B_5 & B_6 & B_7 & 0 & C_3
\end{pmatrix}
\end{align*}
\noindent form a totally singular subspace of $V$. The number of choices for rows $u_i$ ($4l_4+1 \le i \le 4l_4+l_3$) is less than
\begin{align} \label{sum2}
q^{\sum_{1 \le i \le l_3} (\dim(W_{i-1}^\perp)+\dim(V_1))+l_3s} =  q^{\sum_{1 \le i \le l_3} (\frac{kb}{2}-(l_2+2i-2)+\frac{kb}{4})+l_3s} = q^{l_3(\frac{3}{4}kb-l_2-l_3+1+s)}.
\end{align} 
\par We then repeat this argument for the remaining rows $u_i$ ($1 \le i \le 4l_4$). The number of choices for these rows is less than
\begin{align} \label{sum3}
q^{l_4(kb-l_1-2l_2-3l_3-2l_4+2+s)}.
\end{align}
\noindent Therefore, combining (\ref{sum1})-(\ref{sum3}),  
\begin{align*}
\Delta(l_1,l_2,l_3,l_4)<q^{\frac{nm}{4}+(\sum l_i-\frac{m}{4})s+l_3(1-l_2-l_3)+l_4(2+2l_4-m)}.
\end{align*}
\noindent Once $\alpha$ is fixed, $\beta$ and $U$ are also fixed. As above, the number of bases $\beta'$ of $U$ such that $[y^U]_\beta=[y^U]_{\beta'}$ is $|C_{GL_m(q)}([y^U]_\beta)|$. It is elementary to verify that
\begin{align*}
l_3(-l_2-l_3)+l_4(2l_4-m)-\sum_{1 \le i \le 4} il_i^2-2\sum_{1 \le i<j \le 4}il_il_j \le -\frac{3}{8}m^2.
\end{align*}
\noindent Hence, using (\ref{unipotentcent}),
\begin{align} \label{yfixSpeven}
\operatorname{fix}(y,M^G)<\sum_{\sum il_i=m} \frac{\Delta(l_1,l_2,l_3,l_4)}{|C_{\overline{GL_m(q)}}([y]_\beta)|}<q^{\frac{1}{8}m(2n-3m)+cm},
\end{align}
 \noindent completing the proof in the symplectic case.

\paragraph{Case $G=P\Omega^\epsilon_n(q).$}

\par If $q$ is odd, then the proof is similar to the symplectic case.

\par Suppose $q$ is even. There exists a basis $\{ e_i \}$ of $V$ with the quadratic form $Q$ on $V$ defined as
\begin{align*}
Q(\sum_{i=1}^n \lambda_ie_i) = \sum_{i=1}^{\frac{kb}{2}} \lambda_i\lambda_{i+\frac{kb}{2}}+Q'(\sum_{i=kb+1}^n \lambda_ie_i)
\end{align*}
\noindent where $Q'$ is a quadratic form of type $\epsilon$ on the subspace $\langle e_{kb+1}, \dots, e_n \rangle \subset V$. Moreover, with this basis, we have $y$ as in (\ref{yeven2}). Such a basis yields the Gram matrix 
\begin{equation*}
\begin{pmatrix} & I_{\frac{kb}{2}}  \\  I_{\frac{kb}{2}} \\  & & D \end{pmatrix}
\end{equation*}
\noindent (where $D$ is a Gram matrix for an $s$-dimensional orthogonal space of type $\epsilon$).
\par Suppose $y$ stabilizes a totally singular $m$-subspace $U \subset V$. Then $U$ is totally singular with respect to the symplectic form on $V$ determined by $Q$. Therefore $\operatorname{fix}(y,M^G)$ can be bounded as in (\ref{yfixSpeven}). The result now follows.

\paragraph{Case $G=PSU_n(q).$}
The proof in this case is similar to the symplectic case.
\end{proof}

\begin{proof}[Proof of Proposition \ref{i4Mprop2}]
\par By considering the action of $G$ on $M^G$ and (\ref{fpreqn}), we have
\begin{align}
\sum_{M^g \in M^G}  \frac{|x^G \cap M^g|}{|x^G|} \frac{|y^G \cap M^g|}{|y^G|} &= |G:M|   \frac{|x^G \cap M|}{|x^G|} \frac{|y^G \cap M|}{|y^G|} \nonumber \\
&=  \operatorname{fpr}(x,M^G) \operatorname{fix}(y,M^G). \label{5.12eqn}
\end{align}
\noindent For fixed $m$, the number of $G$-classes of parabolic subgroups $P_m$ is at most 2 by \cite[\S~4.1]{KlLi}. The result now follows from Lemmas \ref{fprx} and \ref{fpry}.
\end{proof}

\subsection{Proof of Theorem \ref{bigthm1}}

\par Before we begin the proof of Theorem \ref{bigthm1}, we need two results from \cite{LiSh1} and \cite{LiSh2}. Let $G$ be a finite simple classical group. For a real number $s$, define the zeta function
\begin{equation*}
\zeta_{G}(s)=\sum_{M \operatorname{max}G} |G:M|^{-s}.
\end{equation*}

\vspace{-0.3cm}

\noindent As we are dealing with finite groups, $\zeta_{G}(s)$ is convergent. The following theorem is \cite[Theorem~2.1]{LiSh1}.

\begin{theorem}[\cite{LiSh1}] \label{zetathm}
For any fixed $s>1$, 
\begin{equation*}
\zeta_G(s) \rightarrow 0 \hspace{0.3cm} \text{as} \hspace{0.3cm} |G| \rightarrow \infty.
\end{equation*}
\end{theorem}

\par The second result we need is proved in \cite[Lemma~2]{LiSh2}.
\begin{theorem}[\cite{LiSh2}] \label{i2ratiothm}
Let $G$ be as above. There exists an absolute constant $c$ such that for any $M \operatorname{max} G$ we have
\begin{align*}
\frac{i_2(M)}{i_2(G)} < |G:M|^{-\frac{1}{2}+\frac{c}{n}}.
\end{align*}
\end{theorem}
\noindent  There is a difference in notation between Theorem \ref{i2ratiothm} and \cite[Lemma~2]{LiSh2}: the reference proves the result for exponent $-\frac{1}{2}+o_n(1)$, where $o_n(1) \rightarrow 0$ as $n \rightarrow \infty$. However, inspection of the proof shows that we can replace $o_n(1)$ by $\frac{c}{n}$ for some absolute constant $c$.

\par We now prove Theorem \ref{bigthm1}, using Propositions \ref{i4Mprop1} and \ref{i4Mprop2}.

\begin{proof}[Proof of Theorem \ref{bigthm1}]
\par Let $G$ be a finite simple classical group with natural module of dimension $n$ over $\mathbb{F}_{q^\delta}$. Let $A$ and $B$ be nontrivial finite 2-groups, and assume $x \in I_2(A), y \in I_4(B)$, and $n \ge \operatorname{max}\{2|A|+2,2|B|+2\}$. Embed $A$ and $B$ almost-freely into $G$. 
\par We first consider the contribution to the summation in the statement of Theorem \ref{bigthm1} by non-parabolic subgroups. Let $\mathscr{M}_{\text{np}}$ be the set of non-parabolic maximal subgroups of $G$, and let $M \in \mathscr{M}_{\text{np}}$. By Propositions \ref{conjclassprop1}, \ref{irprop} and \ref{i4prop}, for an absolute constant $c$ we have $|x^G|>i_2(G)q^{cn}, |y^G|>i_4(G)q^{cn}$. By Proposition \ref{i4Mprop1}, $i_4(M)<|M|^{\frac{3}{4}}q^{c'n}$ for some absolute constant $c'$, and so $i_4(M)/i_4(G)<|G:M|^{-\frac{3}{4}+\frac{c''}{n}}$ by Proposition \ref{i4prop}. Therefore, by Theorem \ref{i2ratiothm} we have
\begin{align*}
 \sum_{M \in \mathscr{M}_{\text{np}}} \frac{|x^G \cap M|}{|x^G|} \frac{|y^G \cap M|}{|y^G|} &< \sum_{M \in\mathscr{M}_{\text{np}}} \frac{i_2(M)}{i_2(G)}\frac{i_4(M)}{i_4(G)}q^{c'''n} \\
&< \sum_{M \in\mathscr{M}_{\text{np}}}  |G:M|^{-\frac{5}{4}+\frac{c''''}{n}} \\
&< \zeta_G(\frac{5}{4}-\frac{c''''}{n}), 
\end{align*}
\noindent and by Theorem \ref{zetathm}, $\zeta_G(\frac{5}{4}-\frac{c''''}{n})<\frac{1}{2}$ for sufficiently large $n$.
\par We now consider the contribution by parabolic subgroups. Let $\mathscr{M}_{\text{p}}$ denote the set of maximal parabolic subgroups of $G$. By Proposition \ref{i4Mprop2} we have
\begin{align*}
 \sum_{M \in \mathscr{M}_{\text{p}}}  \frac{|x^G \cap M|}{|x^G|} \frac{|y^G \cap M|}{|y^G|} < q^{-\frac{n}{4}+c}
\end{align*}
\noindent for some absolute constant $c$. Clearly this contribution is less than $\frac{1}{2}$ for sufficiently large $n$. This completes the proof of Theorem \ref{bigthm1}.
\end{proof}

\section*{\center{7. Subgroups isomorphic to $C_2 \times C_2$: classical groups}}
\stepcounter{section}

\par Throughout this section, let $G$ be a finite almost simple classical group with natural module of dimension $n$ over $\mathbb{F}_{q^\delta}$ of characteristic $p$, and let $\operatorname{soc}(G)=S$. Let $H$ be the simple adjoint algebraic group over $l=\overline{\mathbb{F}_p}$ such that $(H^F)'=S$ for a Frobenius morphism $F$. 
\par Let $i_{2 \times 2}(G)$ denote the number of subgroups of $G$ isomorphic to $C_2 \times C_2$. In this section we bound $i_{2 \times 2}(G)$ for certain $G$, analogous to \S 5.

\begin{proposition} \label{i2x2prop}
There exist absolute constants $c_1, c_2$ such that if $S, G, H$ and $F$ are as above with $G \le H^F \langle \alpha \rangle$ for some $\alpha \in I_2(\operatorname{Aut}(S)) \cup \{ 1 \}$, then
\begin{equation*}
|G|^{\frac{3}{4}+\frac{c_1}{n}} < i_{2 \times 2}(G) < |G|^{\frac{3}{4}+\frac{c_2}{n}}.
\end{equation*}
\end{proposition}

\par Before we begin the proof, we require a technical result to help us count the number of elements in certain centralizers.

\begin{lemma} \label{psi}
Let $\psi(b,c,d)$ denote the number of pairs of matrices $(A, B)$ where $A \in M_{b,c}(q), B \in M_{c,d}(q)$ such that $AB=0$. For $i \le j$, let $Q_i(j)=\prod_{k=0}^{i-1} (q^j-q^k)$. Then
\begin{equation*}
\psi(b,c,d)=\sum_{r=0}^{\operatorname{min}(b,c)} \frac{Q_r(b)Q_r(c)}{Q_r(r)}q^{d(c-r)}.
\end{equation*}
\noindent In particular, there exists an absolute constant $c_0$ such that $\psi(b,c,d) < c_0q^{\frac{1}{4}(b+c-d)^2+cd}$.
\end{lemma}

\begin{proof}
Let $\phi_r(b,c)$ denote the number of matrices in $M_{b,c}(q)$ of rank $r$. By \cite{Landsberg},
\begin{equation*}
\phi_r(b,c)=\frac{Q_r(b)Q_r(c)}{Q_r(r)}.
\end{equation*}

\par For a given matrix $A \in M_{b,c}(q)$ with rank $r$ as above, the columns of a matrix $B$ with $AB=0$ must lie in the nullspace of $A$. Therefore, by the rank-nullity theorem, the number of matrices $B \in M_{c,d}(q)$ such that $AB=0$ is $q^{d(c-r)}$. Summing over $0 \le r \le \operatorname{min}(b,c)$ then yields the result. 
\end{proof}

\par Observe that to prove the upper bound in Proposition \ref{i2x2prop} it suffices to prove the result for $G=H^F\langle \alpha \rangle$. The lower bound holds by Proposition \ref{conjclassprop2}, and so throughout the proof we assume $G=H^F \langle \alpha \rangle$.
\par We require one further piece of notation before beginning the proof. For an even integer $m$, define
\begin{align*}
E_m = \begin{pmatrix}
J \\
& J \\
 & & \ddots \\
& & & J
\end{pmatrix} \in M_{m,m}(q)
\end{align*}
\noindent where 
\begin{align*}
J=\begin{pmatrix}
 & 1 \\ 1
\end{pmatrix}.
\end{align*}
\begin{lemma} \label{i2x2prop1}
Proposition \ref{i2x2prop} holds in the case $\alpha=1$.
\end{lemma}

\begin{proof}
\par First suppose $p \ne 2$. By \cite[Theorem 4.3.2]{GoLySo}, involutions $y_1 \in H^F$ have $C_H(y_1)^\circ$ isomorphic to one of
\begin{align*}
\overline{\left( GL_m \times GL_{n-m} \right) \cap SL_n}  \hspace{0.5cm} &\text{if} \ H=PSL_n, \\
\overline{Sp_m \times Sp_{n-m}} \hspace{0.5cm} \text{or} \hspace{0.5cm} \overline{GL_\frac{n}{2}}\hspace{0.5cm} &\text{if} \ H=PSp_n, \\
\overline{SO_m \times SO_{n-m}}  \hspace{0.5cm} \text{or} \hspace{0.5cm} \overline{GL_\frac{n}{2}} \hspace{0.5cm} &\text{if}\  H=PSO_n,
\end{align*}
\noindent with at most 4 $H^F$-conjugacy classes of involutions for each isomorphism type. 
\par Let $C=C_H(y_1)$, and let $y_2 \in I_2(C^F)$ with $y_1 \ne y_2$. First suppose $y_2 \in C^\circ$. Observe that $y_2 \notin Z(C^\circ)$, and so we can consider $y_2$ as an involution in $C^\circ/Z(C^\circ)$. Then $C_{C^\circ/Z(C^\circ)}(y_2)^\circ \cong D_1 \times D_2$ where the possibilities for $D_1$ and $D_2$ are listed in Table \ref{d1timesd2}. Each possibility for $C_{C^\circ}(y_2)^\circ$ represents at most 16 $C^F$-classes of involutions $y_2 \in (C^\circ)^F$.

\begin{table}[h]
\begin{center}
\def\arraystretch{1.4}
\caption {Possibities for $C_{C^\circ/Z(C^\circ)}(y_2)^\circ$}
\label{d1timesd2}
\hspace*{-1.77cm}
\begin{tabular}{ c c c c }
$H$  &  $C^\circ/Z(C^\circ)$ & $D_1$ & $D_2$ \\ \hline
$PSL_n$ & $PSL_m \times PSL_{n-m}$ & $\overline{\left( GL_{l_1} \times GL_{m-l_1} \right) \cap SL_{m}}$ &  $\overline{\left( GL_{l_2} \times GL_{n-m-l_2} \right) \cap SL_{n-m}}$  \\ \hline 
$PSp_n$ & $PSp_m \times PSp_{n-m}$ & $\overline{Sp_{l_1} \times Sp_{m-l_1}}, \overline{GL_{\frac{m}{2}}}$ & $ \overline{Sp_{l_2} \times Sp_{n-m-l_2}}, \overline{GL_{\frac{n-m}{2}}}$ \\
& $PGL_\frac{n}{2}$ & $\overline{GL_m \times GL_{\frac{n}{2}-m}}$ & 1 \\ \hline
$PSO_n$ & $PSO_m \times PSO_{n-m}$ & $\overline{SO_{l_1} \times SO_{m-l_1}}, \overline{GL_{\frac{m}{2}}}$ & $ \overline{SO_{l_2} \times SO_{n-m-l_2}}, \overline{GL_{\frac{n-m}{2}}}$ \\ 
 & $PGL_\frac{n}{2}$ & $\overline{GL_m \times GL_{\frac{n}{2}-m}}$ & 1 \\ \hline
\end{tabular}
\end{center}
\end{table}
\setlength\LTleft{0cm}

\par We have $\dim(C_{C^\circ}(y_2)) \ge \dim(C_{C^\circ/Z(C^\circ)}(y_2))$, and by the Cauchy-Schwarz inequality
\begin{align} \label{7.1eqn1}
 \dim(C_{C^\circ/Z(C^\circ)}(y_2)) > \frac{1}{4} \dim(H)+cn
\end{align}
\noindent for some absolute constant $c$.
   
\par Now suppose $y_2 \in C \backslash C^\circ$. Then $C^\circ$ and the structure of $C/C^\circ$ can be found in \cite[Theorem~4.3.1]{GoLySo}, and we describe them in Table \ref{outertable}. In all but one case, $C/C^\circ \cong C_2$, and an element in $C \backslash C^\circ$ either switches the factors of $C^\circ$ or acts as a graph automorphism on one or two factors (if two factors are present), and this element is described by the symbol $\leftrightarrow, \gamma \  1$ or $\gamma \ \gamma$ respectively (if there is an automorphism on one of two factors, it acts on the left factor). In the case $H=PSO_n$ and $C^\circ = \overline{SO_\frac{n}{2}^2}$ we have $C/C^\circ \cong C_2 \times C_2$, generated by an element switching the factors and an element inducing a graph automorphism on both factors; we denote the three nontrivial elements of $C/C^\circ$ by the symbols $\leftrightarrow, \gamma \ \gamma$, and $\leftrightarrow \times \gamma \ \gamma$. For $y_2 \in C \backslash C^\circ$ with its image in the outer automorphism group of $C^\circ$ as listed, the centralizer $C_{C^\circ}(y_2)^\circ$ is given in the fourth column up to isomorphism.

\begin{table}[h]
\footnotesize
\begin{center}
\def\arraystretch{1.5}
\caption {Involutions in $C \backslash C^\circ$}
\label{outertable}
\hspace*{-1.4cm}
\begin{tabular}{c c c c c}
$H$  & $C^\circ $ & $C/C^\circ$ & $C_{C^\circ}(y_2)^\circ$ & Conditions \\ \hline
$PSL_n$ & $\overline{GL_\frac{n}{2}^2 \cap SL_n}$ & $\leftrightarrow$ & $\overline{GL_\frac{n}{2}}$ & $n$ even   \\ \hline 
$PSp_n$ & $\overline{Sp_\frac{n}{2}^2}$ & $\leftrightarrow$ &  $\overline{Sp_\frac{n}{2}}$ & $\frac{n}{2}$ even \\
& $\overline{GL_\frac{n}{2}}$ & $\gamma$ & $\overline{SO_\frac{n}{2}} \ \text{or}\ \overline{Sp_\frac{n}{2}}$ & \\ \hline
$PSO_n$ & $\overline{SO_m \times SO_{n-m}}$ & $\gamma \ 1$  & $\overline{SO_{l_0} \times SO_{m-l_0} \times SO_{n-m}}$ & $n, l_0$ odd, $m$ even \\ 
& $\overline{SO_m \times SO_{n-m}}$  & $\gamma \ \gamma$  & $\overline{SO_{l_1} \times SO_{m-l_1} \times SO_{l_2} \times SO_{n-m-l_2}}$ & $n, m$ even, $l_1, l_2$ odd, $m \ne \frac{n}{2}$ \\ 
& $\overline{SO_\frac{n}{2}^2}$ & $\leftrightarrow$ & $\overline{SO_\frac{n}{2}}$ & $\frac{n}{2}$ even  \\
& & $\gamma \ \gamma$ & $\overline{SO_{l_1} \times SO_{\frac{n}{2}-l_1} \times SO_{l_2} \times SO_{\frac{n}{2}-l_2}}$ & $\frac{n}{2}$ even, $l_1, l_2$ odd \\
& & $\leftrightarrow \times \gamma \ \gamma$ & $\overline{SO_\frac{n}{2}}$ & $\frac{n}{2}$ even  \\
& $\overline{GL_\frac{n}{2}}$ & $\gamma$ & $\overline{SO_\frac{n}{2}} \ \text{or}\ \overline{Sp_\frac{n}{2}}$ & \\ \hline
\end{tabular}
\end{center}
\end{table}

\par Each $C^\circ$ listed in Table \ref{outertable} represents at fewer than $cq$ $C^F$-classes of involutions $y_2 \in C^F \backslash (C^\circ)^F$, for some absolute constant $c$. Moreover, in each case we have
\begin{align} \label{7.1eqn2}
\dim(C_{C^\circ}(y_2)^\circ) > \frac{1}{4} \dim(H)+cn
\end{align}
\noindent for some absolute constant $c$.
\par Therefore there are fewer than $c'qn$ $H^F$ classes of subgroups $\langle y_1,y_2 \rangle \cong C_2 \times C_2$ in $H^F$, and for each we have
\begin{align*}
\dim(C_H(y_1) \cap C_H(y_2)) > \frac{1}{4}\dim(H) + c''n
\end{align*}
\noindent by (\ref{7.1eqn1}) and (\ref{7.1eqn2}). Let $D=C_H(y_1) \cap C_H(y_2)$, and observe that $\operatorname{max}\{ \operatorname{rank}(Z(D^\circ)), |D:D^\circ| \} \le 4$ by \cite[Theorem~4.3.1]{GoLySo}. Hence, by Lemma \ref{centorder}, for any commuting involutions $y_1, y_2 \in H^F$,
\begin{align*}
|\langle y_1,y_2 \rangle^{H^F}| < q^{\frac{3}{4}\dim(H)+c''n}
\end{align*}
\noindent for some absolute constant $c''$, and the result follows.

\par Now suppose $p=2$. First consider the linear case. Involutions in $H^F$ are conjugate to
\begin{align*}
j_{k,n} = \begin{pmatrix}
I_k \\
& I_{n-2k} \\
I_k & & I_k
\end{pmatrix}
\end{align*}
\noindent for some $1 \le k \le \frac{n}{2}$, and so
\begin{align} \label{7.2eqn2}
i_{2 \times 2}(H^F) &< \sum_{1 \le k \le \frac{n}{2}} |j_{k,n}^{H^F}|  i_2(C_{H^F}(j_{k,n})) .
\end{align}

\par The centralizer $C_{H^F}(j_{k,n})$ contains matrices of the form
\begin{align*}
\begin{pmatrix}
X \\
P & Y \\
Q & R & X
\end{pmatrix}
\end{align*}
\noindent with $X \in GL_k(q), Y \in GL_{n-2k}(q)$ and arbitrary $P, Q, R$, and so 
\begin{align} \label{7.2eqn3}
|j_{k,n}^{H^F}|<cq^{2k(n-k)}.
\end{align} 
\par Observe that $C_{H^F}(j_{k,n})=UL$, where $L$ is the subgroup of $C_{H^F}(j_{k,n})$ with $P=Q=R=0$, and $U$ is the subgroup with $X=Y=1$. Involutions in $C_{H^F}(j_{k,n})$ are $L$-conjugate to
\begin{align*}
\begin{pmatrix}
j_{l_1, k} \\
P & j_{l_2,n-2k} \\
Q & R & j_{l_1,k}
\end{pmatrix}
\end{align*}
\noindent for some $0 \le l_1 \le \frac{k}{2}, 0 \le l_2 \le \frac{n-2k}{2}$ and some $P,Q,R$ such that 
\begin{align} 
Pj_{l_1,k}=j_{l_2,n-2k}P, \nonumber \\
j_{l_1,k}R=Rj_{l_2,n-2k}, \label{conditions1} \\
j_{l_1,k}Q+Qj_{l_1,k}=RP. \nonumber
\end{align}
\noindent We can therefore write
\begin{align} \label{conditions2}
P=\begin{pmatrix}
P_1\\
P_3 & P_2 \\
P_4 & P_5 & P_1
\end{pmatrix}, \ 
R=\begin{pmatrix}
R_1\\
R_3 & R_2 \\
R_4 & R_5 & R_1
\end{pmatrix},
\end{align}
\noindent where 
\begin{align}
P_1, R_1^T &\in M_{l_2,l_1}(q), \nonumber \\
P_2, R_2^T &\in M_{n-2k-2l_2,k-2l_1}(q),  \label{conditions3} \\
R_2P_2 &=0. \nonumber
\end{align} 
\par Let $w_{l_1,l_2,k,n} \in L$ denote the involution
\begin{align*}
\begin{pmatrix}
j_{l_1, k} \\
 & j_{l_2,n-2k} \\
 &  & j_{l_1,k}
\end{pmatrix},
\end{align*}
\noindent and let $I_{l_1, l_2,k,n}$ denote the number of involutions $uw_{l_1,l_2,k_,n} \in UL$. Then we have
\begin{align} \label{7.2eqn4}
i_2(C_{H^F}(j_{k,n})) < \sum_{\substack{0 \le l_1 \le \frac{k}{2} \\ 0 \le l_2 \le \frac{n-2k}{2}}} |w_{l_1,l_2,k,n}^L| I_{l_1, l_2,k,n}.
\end{align}

\noindent Using Lemma \ref{psi}, the number of pairs ($R_2, P_2$) such that (\ref{conditions3}) holds is
\begin{align*}
\psi(k-2l_1,n-2k-2l_2,k-2l_1) < c q^{\frac{1}{4}(n-2k-2l_2)^2+(k-2l_1)(n-2k-2l_2)}
\end{align*}
\noindent for some abolsute constant $c$. Using (\ref{conditions2}), we count the number of possible $P_i, R_i$ ($i=1,3,4,5$); we then calculate the number of possible $Q$ given $R$ and $P$ using (\ref{conditions1}), and we find
\begin{align*}
I_{l_1,l_2,k,n} &< cq^{\frac{1}{4}(n-2k-2l_2)^2+(k-2l_1)(n-2k-2l_2)+2l_1(n-2k-2l_2)+2l_2(k-2l_1)+4l_1l_2+2l_1^2+2l_1(k-2l_1)+(k-2l_1)^2} \\
&= cq^{\frac{1}{4}n^2-l_2(n-2k-l_2)-2l_1(k-l_1)}.
\end{align*}

\noindent Therefore, by (\ref{7.2eqn3}) and (\ref{7.2eqn4}), we have
\begin{align*}
i_2(C_{H^F}(j_{k,n})) < \sum_{\substack{0 \le l_1 \le \frac{k}{2} \\ 0 \le l_2 \le \frac{n-2k}{2}}}   c'q^{\frac{1}{4}n^2+l_2(n-2k-l_2)}.
\end{align*}

\noindent Hence by (\ref{7.2eqn2}),
\begin{align*}
i_{2 \times 2}(H^F) &< \sum_{1 \le k \le \frac{n}{2}} \sum_{\substack{0 \le l_1 \le \frac{k}{2} \\ 0 \le l_2 \le \frac{n-2k}{2}}}   c''q^{\frac{1}{4}n^2+2k(n-k)+l_2(n-2k-l_2)}.
\end{align*}
\noindent It can easily be shown that $\frac{1}{4}n^2+2k(n-k)+l_2(n-2k-l_2) \le \frac{3}{4}n^2$, and the result follows for the linear case.

\par Next consider the symplectic case. By \cite{AsSe}, involutions in $H^F$ are conjugate to $a_{k,n} \ (2 \le k \le \frac{n}{2}, k \ \text{even}), b_{k,n} \ (1 \le k \le \frac{n}{2}, k \ \text{odd}),$ or $c_{k,n} \ (2 \le k \le \frac{n}{2}, k \ \text{even})$, where
\begin{align} \label{7.2eqn9}
|a_{k,n}^{H^F}| \sim q^{k(n-k)}, |b_{k,n}^{H^F}|, |c_{k,n}^{H^F}| \sim q^{k(n-k+1)},
\end{align}
\noindent and each has Jordan normal form $j_{k,n}$ as in the linear case. We have
\begin{align} 
i_{2 \times 2}(H^F) &< \sum_{a_{k,n}} |a_{k,n}^{H^F}| i_{2 \times 2}(C_{H^F}(a_{k,n})) + \sum_{b_{k,n}} |b_{k,n}^{H^F}| i_{2 \times 2}(C_{H^F}(b_{k,n})) + \sum_{c_{k,n}} |c_{k,n}^{H^F}| i_{2 \times 2}(C_{H^F}(c_{k,n})) \nonumber \\
&= \Sigma_a+\Sigma_b+\Sigma_c. \label{7.2eqn10}
\end{align}
\par We first consider $\Sigma_a$. Consider the involution $a_{k,n}$, which acts as $j_{k,n}$ with respect to a Gram matrix 
\begin{equation*}
J=\begin{pmatrix} & & F \\ & E \\ F \end{pmatrix},
\end{equation*}
\noindent where $F \in M_{k,k}(q), E \in M_{n-2k, n-2k}(q)$. The centralizer $C_{H^F}(a_{k,n})$ contains matrices of the form
\begin{align*}
\begin{pmatrix}
X \\ P & Y \\ Q & R & X
\end{pmatrix}
\end{align*}
\noindent where $X \in Sp_k(q), Y \in Sp_{n-2k}(q)$, and $P, Q, R$ satisfy
\begin{align*}
P^TEY+X^TFR&=0, \\
Q^TFX + P^TEP + X^TFQ&=0.
\end{align*}
\noindent Let $L$ be the subgroup of $C_{H^F}(a_{k,n})$ with $P=Q=R=0$, and let $U$ be the subgroup with $X=Y=1$. As in the linear case, $C_{H^F}(a_{k,n})=UL$. Involutions in $L$ are conjugate to
\begin{align*}
w_{l_1,l_2,k,n}(d,e)=
\begin{pmatrix}
d_{l_1,k} \\
& e_{l_2,n-2k} \\
& & d_{l_1,k}
\end{pmatrix},
\end{align*}
\noindent where $d,e \in \{ a,b,c \}$. Recall the definition of $E_m \in GL_m(q)$ preceding Lemma \ref{i2x2prop1}. We have $d_{l_1,k}, e_{l_2,n-2k}$ acting as $j_{l_1,k}, j_{l_2,n-2k}$ with respect to a basis such that $E, F$ are of the form
\begin{align*}
E&=\begin{pmatrix}
& & E_e \\ & E_{n-2k-2l_2} \\ E_e
\end{pmatrix}, \\
F &=\begin{pmatrix}
& & F_d \\ & E_{k-2l_1} \\ F_d
\end{pmatrix},
\end{align*}
\noindent where
\begin{align*}
  E_e &= \left\{ 
  \begin{array}{l l}
E_{l_2} \hspace{0.3cm} &\text{if $e=a$,} \\[0.2cm]
\begin{pmatrix} 1 \\ & E_{l_2-1}  \end{pmatrix} &\text{if $e=b$,} \\[0.5cm]
\begin{pmatrix} & & 1 \\ & E_{l_2-2} \\ 1 & & 1 \end{pmatrix} &\text{if $e=c$}, \\[0.2cm]
\end{array}
\right. 
\end{align*}
\noindent and $F_d \in M_{l_1,l_1}(q)$ is defined similarly.
\par Let $I_{l_1,l_2,k,n}(d,e)$ be the number of involutions equal to
\begin{align*}
u w_{l_1,l_2,k,n}(d,e) = \begin{pmatrix}
d_{l_1,k} \\
P & e_{l_2,n-2k} \\
Q & R & d_{l_1,k}
\end{pmatrix}  \in UL.
\end{align*}
\noindent We now count the number of such involutions. The matrices $P,Q,R$ are subject to the conditions
\begin{align}
Pd_{l_1,k}=e_{l_2,n-2k}P, \label{7.2eqn11} \\
d_{l_1,k}Q+Qd_{l_1,k}=RP, \label{7.2eqn12} \\
P^TEe_{l_2,n-2k}=d_{l_1,k}^TFR, \label{7.2eqn13} \\
Q^TFd_{l_1,k}+d_{l_1,k}^TFQ=P^TEP. \label{7.2eqn14}
\end{align}
\noindent Condition (\ref{7.2eqn11}) implies $P$ is of the form
\begin{align} \label{Pform}
\begin{pmatrix}
P_1 \\
P_3 & P_2 \\
P_4 & P_5 & P_1
\end{pmatrix}
\end{align}
\noindent where $P_1 \in M_{l_2,l_1}(q)$, $P_2 \in M_{n-2k-2l_2, k-2l_1}(q)$. Condition (\ref{7.2eqn13}) implies that $R$ is determined by $P$, with 
\begin{align*}
R= \begin{pmatrix} R_1 \\ R_3 & R_2 \\ R_4 & R_5 & R_1 \end{pmatrix} =\begin{pmatrix}
F_d P_1^T E_e \\
E_{k-2l_1} P_5^T E_e & E_{k-2l_1} P_2^T E_{n-2k-2l_2} \\
F_d P_4^T E_e & F_d P_3^T E_{n-2k-2l_2} & F_d P_1^T E_e 
\end{pmatrix}.
\end{align*}
\noindent By condition (\ref{7.2eqn12}), if
\begin{align*}
\begin{pmatrix}
Q_1 & Q_2 & Q_3 \\
Q_4 & Q_5 & Q_6 \\
Q_7 & Q_8 & Q_9
\end{pmatrix},
\end{align*} 
\noindent then $Q_1+Q_9,Q_2,Q_3,Q_6$ are determined by $RP$, and we also have 
\begin{align}
R_2P_2 =E_{k-2l_1} P_2^T E_{n-2k-2l_2} P_2=0. \label{r2p2condition}
\end{align}
\noindent If we write $P_2$ as the block matrix
\begin{align*}
\begin{pmatrix}
S_1 & S_2 \\ S_3 & S_4
\end{pmatrix}
\end{align*}
\noindent where 
\begin{align*}
S_1, S_2 &\in M_{2\lfloor \frac{1}{4}(n-2k-2l_2) \rfloor, \frac{1}{2}(k-2l_1)}(q), \\
S_3, S_4 &\in M_{2\lfloor \frac{1}{4}(n-2k-2l_2+2) \rfloor, \frac{1}{2}(k-2l_1)}(q),
\end{align*}
\noindent then (\ref{r2p2condition}) implies
\begin{align*}
\begin{pmatrix}
S_1^T E_{2\lfloor \frac{1}{4}(n-2k-2l_2) \rfloor} & S_3^T E_{2\lfloor \frac{1}{4}(n-2k-2l_2+2) \rfloor}
\end{pmatrix}
\begin{pmatrix}
S_2 \\ S_4 
\end{pmatrix}=0.
\end{align*}
\noindent Therefore, by Lemma \ref{psi}, the number of choices for $P_2$ is less than 
\begin{align*}
\psi(\frac{1}{2}(k-2l_1),n-2k-2l_2,\frac{1}{2}(k-2l_1))<cq^{\frac{1}{4}(n-2k-2l_2)(2(k-2l_1)+(n-2k-2l_2))}
\end{align*}
\noindent for some absolute constant $c$. By (\ref{Pform}), the number of choices for $P$ is therefore less than
\begin{align*}
cq^{2l_1l_2+l_1(n-2k-2l_2)+l_2(k-2l_1)+\frac{1}{4}(n-2k-2l_2)(2(k-2l_1)+(n-2k-2l_2))}.
\end{align*}
\noindent Given $P$, by (\ref{7.2eqn12}) we have that $Q_1+Q_9,Q_2, Q_3$ and $Q_6$ are determined. Moreover, by (\ref{7.2eqn14}) and matrix calculations, the number of choices for the remaining blocks of $Q$ is less than
\begin{align*}
q^{l_1^2+l_1(k-2l_1)+\frac{1}{2}(k-2l_1)^2+c'n}
\end{align*}
\noindent for an absolute constant $c'$. Hence for any $d,e \in \{ a,b,c \}$ we have
\begin{align*}
I_{l_1,l_2,k,n}(d,e) <q^{\frac{1}{2}k(n-2k)+\frac{1}{4}(n-2k-2l_2)^2+l_1^2+\frac{1}{2}(k-2l_1)^2+l_1(k-2l_1)+c''n}.
\end{align*}
\noindent Therefore, using (\ref{7.2eqn3}) and (\ref{7.2eqn9}),
\begin{align*}
\hspace{-0.2cm} i_{2 \times 2}(C_{H^F}(a_{k,n})) &< \sum_{\substack{0 \le l_1 \le \frac{k}{2} \\ 0 \le l_2 \le \frac{n-2k}{2} \\ d,e \in \{ a,b,c \} }} |w_{l_1,l_2,k,n}(d,e)^L| I_{l_1,l_2,k,n}(d,e) \\ 
&< \sum_{\substack{0 \le l_1 \le \frac{k}{2} \\ 0 \le l_2 \le \frac{n-2k}{2} \\ d,e \in \{ a,b,c \} }} q^{\frac{1}{2}k(n-2k)+\frac{1}{4}(n-2k-2l_2)^2+l_1^2+\frac{1}{2}(k-2l_1)^2+l_1(k-2l_1)+l_1(k-l_1)+l_2(n-2k-l_2)+c'''n} \\
&=\sum_{\substack{0 \le l_1 \le \frac{k}{2} \\ 0 \le l_2 \le \frac{n-2k}{2} \\ d,e \in \{ a,b,c \} }} q^{\frac{1}{4}(n-k)^2+\frac{1}{4}k^2+c'''n} \\
&< q^{\frac{1}{4}(n-k)^2+\frac{1}{4}k^2+c''''n}.
\end{align*}
\noindent Therefore
\begin{align*}
\Sigma_a &= \sum_{\substack{2 \le k \le \frac{n}{2} \\ k \ \text{even}}} |a_{k,n}^{H^F}| i_{2 \times 2}(C_{H^F}(a_{k,n})) < \sum_{2 \le k \le \frac{n}{2}} q^{\frac{1}{8}n^2+\frac{1}{8}(n-2k)^2+k(n-k)+d'n}.
\end{align*}
\noindent for some absolute constant $d'$. Since $\frac{1}{8}n^2+\frac{1}{8}(n-2k)^2+k(n-k) \le \frac{3n^2}{8}$, we have $\Sigma_a<q^{\frac{3n^2}{8}+d''n}$ for some absolute constant $d''$. 
\par Similar calculations show $\operatorname{max}(\Sigma_b, \Sigma_c)<q^{\frac{3n^2}{8}+d'''n}$ for an absolute constant $d'''$, and this completes the proof in the symplectic case by (\ref{7.2eqn10}).

\par Next consider the case where $H$ is orthogonal. In this case the result follows immediately from the symplectic result, since $H^F \le PGSp_n(q)$.

\par Finally, in the case where $H^F=PGU_n(q)$, the proof is similar to the symplectic case.
\end{proof}

\begin{lemma}
Proposition \ref{i2x2prop}  holds if $\alpha \ne 1$.
\end{lemma}

\begin{proof}
\par By Lemma \ref{i2x2prop1}, it suffices to count Klein four-groups $\langle y_1,y_2 \rangle \le H^F\langle \alpha \rangle$ with $y_1 \in H^F\alpha$. In most cases, we bound the number of conjugacy classes of $y_1 \in I_2(H^F \alpha)$, and for each class representative we show
\begin{align} \label{6.4eqn1}
|y_1^{H^F}| i_2(C_{H^F\langle \alpha \rangle}(y_1)) < q^{\frac{3}{4}\dim(H)+cn}
\end{align}
\noindent for some absolute constant $c$. The result follows from this.
\par First suppose $\alpha$ is a field or graph-field automorphism. By \cite[Proposition~4.1]{GoLySo}, $y_1 \in I_2(H^F \alpha)$ is $H^F$-conjugate to $\alpha$, and either
\begin{align*}
C_{H^F}(\alpha) \cong PGL_n(q^\frac{1}{2}), PGSp_n(q^\frac{1}{2}), PGO^+_n(q^\frac{1}{2})
\end{align*}
\noindent if $\alpha$ is a field automorphism and $H$ is linear, symplectic, orthogonal respectively, or
\begin{align*}
C_{H^F}(\alpha) \cong PGU_n(q^\frac{1}{2}), PGO^-_n(q^\frac{1}{2})
\end{align*}
\noindent if $\alpha$ is a graph-field automorphism and $H$ is linear, orthogonal respectively. Therefore, by Proposition \ref{irprop}, 
\begin{align*}
i_2(C_{H^F \langle \alpha \rangle}(\alpha)) < q^{\frac{1}{4}\dim(H)+cn}
\end{align*}
\noindent for some absolute constant $c$. We have
\begin{align*}
|\alpha^{H^F}| < q^{\frac{1}{2}\dim(H)+c'n}
\end{align*}
\noindent for an absolute constant $c'$, and so (\ref{6.4eqn1}) holds.
\par Now suppose $\alpha$ is a graph automorphism, so $H$ is linear or orthogonal. In the orthogonal case, we have $H^F \langle \alpha \rangle \cong PGO^\epsilon_n(q)$, and so if $q$ is even the result follows by Lemma \ref{i2x2prop1} since $H^F \langle \alpha \rangle \le Sp_n(q)$. If $q$ is odd, then by \cite[Theorem~4.5.1]{GoLySo} there are fewer than $\frac{n}{4}+2$ $H^F$-classes of involutory graph automorphisms $y_1 \in H^F \alpha$ (assuming $n>8$), with centralizer in $H^F$ isomorphic to one of
\begin{align*}
C_{H^F}(y_1) \cong
SO_{n-2i+1}(q) \times SO_{2i-1}(q) \ (1 \le i \le \frac{n}{4}), \\
SO_{\frac{n}{2}}(q)^2.2 \ \text{($\frac{n}{2}$ odd)}, 
SO_\frac{n}{2}(q^2).2 \ \text{($\frac{n}{2}$ odd)}, 
\end{align*}
\noindent where in the final two cases the outer involution acts by switching the factors and as a field automorphism respectively. If $C_{H^F}(y_1) \cong SO_{n-2i+1}(q) \times SO_{2i-1}(q)$, then by Proposition \ref{irprop} we have
\begin{align*}
i_2(C_{H^F}(y_1)) &< (i_2(SO_{n-2i+1}(q)+1)(i_2(SO_{2i-1}(q))+1) \\ &< |SO_{n-2i+1}(q)|^{\frac{1}{2}+\frac{c}{n}}|SO_{2i-1}(q)|^{\frac{1}{2}+\frac{c}{n}},
\end{align*}
\noindent and so 
\begin{align*}
|y_1^{H^F\langle \alpha \rangle}| i_2(C_{H^F}(y_1)) < \frac{|H^F \langle \alpha \rangle|}{|C_{H^F\langle \alpha \rangle}(y_1)|^{\frac{1}{2}+\frac{c'}{n}}},
\end{align*}
\noindent and the result follows in this case as (\ref{6.4eqn1}) holds. If $C_{H^F}(y_1) \cong SO_{\frac{n}{2}}(q)^2.2$, then
\begin{align*}
i_2(C_{H^F}(y_1)) < i_2(SO_\frac{n}{2}(q)^2)+|SO_\frac{n}{2}(q)|,
\end{align*}
\noindent and the result now follows as above. If $C_{H^F}(y_1) \cong SO_\frac{n}{2}(q^2).2$, then let
\begin{align*}
C_{H^F}(y_1) \cong SO_{\frac{n}{2}}(q^2)\langle \phi \rangle,
\end{align*}
\noindent where $\phi$ is an involutory field automorphism of $SO_\frac{n}{2}(q^2)$. By \cite[Proposition~4.1]{GoLySo}, all involutions in $SO_{\frac{n}{2}}(q^2) \phi$ are $SO_\frac{n}{2}(q^2)$-conjugate to $\phi$, and so by Proposition \ref{irprop} we have
\begin{align*}
i_2(C_{H^F}(y_1)) &= i_2(SO_\frac{n}{2}(q^2))+|\phi^{SO_\frac{n}{2}(q^2)}| \\ &= i_2(SO_\frac{n}{2}(q^2)) +\frac{|SO_{\frac{n}{2}}(q^2)|}{|SO_{\frac{n}{2}}(q)|} \\
&< q^{\frac{n^2}{8}+c'n},
\end{align*}
\noindent and the result now follows as (\ref{6.4eqn1}) holds.
\par In the linear case, if $q$ is even then by \cite[\S~19.9]{AsSe} there are at most 2 classes of involutions $y_1 \in H^F\alpha$, and we have $C_{H^F}(y_1)$ isomorphic to a subgroup of $PGSp_n(q)$. If $q$ is odd, then by \cite[Theorem~4.5.1]{GoLySo} there are at most 3 classes of involutions $y_1 \in H^F \alpha$, and $C_{H^F}(y_1)$ is isomorphic to $PGSp_n(q)$ or $PGO_n(q)$. The result now follows from Proposition \ref{irprop} and (\ref{6.4eqn1}).
\end{proof}

\section*{\center{8. Subgroups isomorphic to $C_2 \times C_2$: maximal subgroups}}
\stepcounter{section}
\par Let $G=S$ be a finite simple classical group with natural module of dimension $n$ over $\mathbb{F}_{q^\delta}$ as in \S~2.1, and let $A$ and $B$ be nontrivial finite elementary abelian 2-groups with $|A|=a, |B|=b$. Assume $|B|>2$ and $n \ge \operatorname{max}\{2a+2, 2b+2\}$, and embed $A$ and $B$ into $G$ almost-freely as in \S~2.4, with $n=k_a a+s_a=k_b b+s_b$ such that $k_i$ is even and $2 \le s_i < 2i+2$ for $i=a,b$. Let $x \in I_{2}(A)$ and $K=\langle y_1, y_2 \rangle \in I_{2 \times 2}(B)$.

\par In this section we prove Theorem \ref{bigthm2} over three subsections, analogous to \S~6. In \S~8.1, we prove the following result.

\begin{proposition} \label{i2x2Mprop1}
There exists an absolute constant $c$ such that for any non-parabolic maximal subgroup $M$ of $G$,
\begin{align*}
i_{2 \times 2}(M) < |M|^\frac{3}{4}q^{cn}.
\end{align*}
\end{proposition}
\noindent In \S~8.2 we prove
\begin{proposition} \label{i2x2Mprop2}
For a maximal parabolic subgroup $M=P_m$ of $G$ ($1 \le m \le \frac{n}{2}$) 
\begin{align*}
\sum_{M^h \in M^G} \frac{|x^G \cap M^h|}{|x^G|} \frac{| \{ K^g: g \in G, K^g \le M^h \} |}{|K^G|}<q^{-f+cm}
\end{align*}
\noindent where $c=c(A,B)$ is a constant depending only on $A$ and $B$, and
\begin{align*}
 f=\left\{ 
  \begin{array}{l l}
\frac{1}{4}m(n-m) \hspace{0.3cm} &\text{if $G=PSL_n(q)$,} \\[0.2cm]
\frac{1}{44}m(11n-21m) &\text{if $G=PSp_n(q), P\Omega_n(q)$,} \\[0.2cm]
\frac{1}{22}m(11n-21m) &\text{if $G=PSU_n(q)$}.
\end{array}
\right. 
\end{align*}
\noindent In particular, $-f+cm < -\frac{\delta n}{4}+c'$ for some constant $c'=c'(A,B)$.
\end{proposition}

\par Using Propositions \ref{i2x2Mprop1} and \ref{i2x2Mprop2} we prove Theorem \ref{bigthm2} in \S~8.3.

\subsection{Non-parabolic maximal subgroups}

\par Recall the descriptions of Aschbacher classes $\mathscr{C}_i, 1 \le i \le 8$ and $\mathscr{S}$ from \S~2.3. Here, we prove that Proposition \ref{i2x2Mprop1} holds for each Aschbacher class. We first require a result counting Klein four-groups in symmetric groups.

\begin{lemma} \label{i2x2sn}
There exists an absolute constant $c$ such that for all $n$, 
\begin{align*}
i_{2 \times 2}(S_n)<|S_n|^{\frac{3}{4}}e^{cn}.
\end{align*}
\end{lemma}

\begin{proof}
\par Let $x_m \in S_n$ be an involution with $m$ 2-cycles ($1 \le m \le \frac{n}{2}$). Then $C_{S_n}(x_m)\cong S_{n-2m} \times S_2 \wr S_m$. By \cite{MW},
\begin{align*}
i_2(S_{n-2m}) &\sim \operatorname{exp}(\frac{n-2m}{2}(\operatorname{log}(n-2m)-1)+\sqrt{n-2m}) \hspace{0.5cm} \text{if $n \ne 2m$}, \\
2^mi_2(S_m) &\sim 2^m\operatorname{exp}(\frac{m}{2}(\operatorname{log}(m)-1)+\sqrt{m}).
\end{align*}
\noindent Also, by Stirling's formula, if $n \ne 2m$ then
\begin{align*}
|x_m^{S_n}| \sim 2^{-m}\operatorname{exp}\bigg( n(\log n -1)-(n-2m)&(\log (n-2m) -1) \\ &-m(\log m -1)+\frac{1}{2}\log \frac{n}{2\pi m(n-2m)} \bigg). 
\end{align*}
\noindent In the case where $n=2m$ we have
\begin{align*}
|x_{\frac{n}{2}}^{S_n}| \sim e^{-\frac{n}{2}}n^{\frac{n}{2}}.
\end{align*}
 \noindent Therefore,
\begin{align*}
\hspace{0cm} i_{2 \times 2}({S_n}) &< \sum_{m=1}^{\lfloor \frac{n}{2} \rfloor} |x_m^{S_n} | \  i_2(C_{S_n}(x_m)) \\
&< \sum_{m=1}^{\lfloor \frac{n-1}{2} \rfloor} c\operatorname{exp} \bigg( n \left( \log n-1 \right)-\frac{n-2m}{2} \left( \operatorname{log}(n-2m)-1 \right) \\ &\hspace{2cm} -\frac{m}{2}\left( \operatorname{log}(m)-1 \right) +\sqrt{n-2m}+\sqrt{m}+\frac{1}{2}\log \frac{n}{2\pi m(n-2m)} \bigg) \\
& \hspace{5cm} +c2^{\frac{n}{2}+\frac{1}{2}}\operatorname{exp}\left( \frac{n}{4}\left( \log \frac{n}{2}-1 \right)+\frac{n}{2} \left( \log n-1 \right)+\sqrt{\frac{n}{2}} \ \right).
\end{align*}
\noindent It is elementary to show that
\begin{align*}
n(\log n-1)-\frac{n-2m}{2}(\operatorname{log}(n-2m)-1)-\frac{m}{2}(\operatorname{log}(m)-1)<\frac{3}{4}n\log n +c'n
\end{align*}
\noindent for $n \ne 2m$ and for some absolute constant $c'$, and similarly
\begin{align*}
\frac{n}{4}(\log \frac{n}{2}-1)+\frac{n}{2}(\log n-1)<\frac{3}{4}n\log n +c'n.
\end{align*}
\noindent The result follows by further use of Stirling's formula.
\end{proof}

\begin{lemma}
Proposition \ref{i2x2Mprop1}  holds for $M \in \mathscr{C}_1$ with $M$ non-parabolic.
\end{lemma}

\begin{proof}
\par Let $M \in \mathscr{C}_1$ with $M$ non-parabolic as in Table \ref{maxsubgroups}. Suppose first that $G$ is unitary and $M$ lies in the image modulo scalars of $GU_m(q) \times GU_{n-m}(q)$. Then by Propositions \ref{irprop} and \ref{i2x2prop},
\begin{align*}
i_{2 \times 2}(M) &< c(i_{2 \times 2}(PGU_m(q))+2i_2(PGU_{m}(q))+1)(i_{2 \times 2}(PGU_{n-m}(q))+2i_2(PGU_{n-m}(q))+1)\\
&< |PGU_m(q)|^{\frac{3}{4}+\frac{c'}{m}}|PGU_{n-m}(q)|^{\frac{3}{4}+\frac{c'}{n-m}} \\
&< |M|^{\frac{3}{4}}q^{c''n}
\end{align*}
\noindent for some absolute constant $c''$.
\par The remaining cases are proved similarly.
\end{proof}

\begin{lemma} \label{i2x2class2}
Proposition \ref{i2x2Mprop1}  holds for $M \in \mathscr{C}_2$.
\end{lemma}

\begin{proof}
\par Suppose $M \in \mathscr{C}_2$ as in Table \ref{maxsubgroups}, and first suppose that $M$ is of the form $Cl_{m}(q) \wr S_t$. For a scalar $\lambda$, let $I_{2,\lambda}(Cl_m(q))$ denote the set of elements $z \in Cl_m(q)$ such that $z^2= \lambda$. Suppose we have commuting involutions in $M$, and let $x=((a_1, \dots, a_t), \sigma)$ and $y=((b_1, \dots, b_t), \tau)$ be preimages. Then for fixed scalars $\lambda, \mu$ and $\gamma$ we have $\sigma \tau = \tau \sigma$, $x^2=\lambda, y^2 =\mu$, $a_ib_{\sigma(i)}=\gamma b_ia_{\tau(i)}$ for all $1 \le i \le t$, and the following hold:
\begin{enumerate}[label=(\roman*)]
\item if $\sigma(i)=\tau(i)=i,$ then $a_i  \in I_{2, \lambda}(Cl_m(q)), b_i \in I_{2, \mu}(Cl_m(q))$ and $a_ib_i=\gamma b_ia_i$;
\item if $\sigma(i)=j \ne i, \tau(i)=i$ (so $\tau(j)=j$), then $a_i \in I_{2, \lambda}(Cl_m(q)), b_ib_j=\mu$ and $b_j=\gamma b_i^{a_i}$;
\item if $\sigma(i)=i, \tau(i)=j \ne i$ (so $\sigma(j)=j$), then $b_i \in I_{2, \mu}(Cl_m(q)), a_ia_j=\lambda$ and $a_j=\gamma a_i^{b_i}$;
\item if $\sigma(i)=\tau(i)=j \ne i$, then $a_i=\lambda a_j^{-1}, b_i=\mu b_j^{-1},$ and $a_ib_j \in I_{2,\lambda \mu \gamma}(Cl_m(q))$;
\item if $\sigma(i)=j, \tau(i)=l$, then $\sigma(l)=\tau(j)=s$, and $a_i=\lambda a_l^{-1},a_j=\lambda a_s^{-1},b_i=\mu b_j^{-1},b_l=\mu b_s^{-1}, a_ib_j=\gamma b_ia_l$.
\end{enumerate}
\noindent Let 
\vspace{-0.2cm}
 \begin{align*}
r_1&=r_1(\sigma, \tau)=| \{1 \le i \le t: \sigma(i)=\tau(i)=i \} |, \\
r_2 &=r_2(\sigma, \tau) = | \{ 1 \le i \le t: \sigma(i)=j, \tau(i)=i \ \text{for some $j \ne i$}\} |, \\
r_3 &= r_3(\sigma, \tau) = | \{ 1 \le i \le t:\sigma(i)=i, \tau(i)=j \ \text{for some $j \ne i$}\} |, \\
r_4 &= r_4(\sigma, \tau) = | \{1 \le i \le t: \sigma(i)=\tau(i)=j \ \text{for some $j \ne i$} \} |, \\
r_5 &= r_5(\sigma, \tau) = | \{ 1 \le i \le t: \sigma(i)=j, \tau(i)=l \ \text{for some $j,l \ne i, j \ne l$}\} |,
\end{align*}
\noindent so $t=r_1+2r_2+2r_3+2r_4+4r_5$. Also, let $I(s_1, \dots, s_5)$ be the number of unordered pairs of commuting involutions $\sigma, \tau \in S_t$ with $r_i(\sigma, \tau)=s_i$ for $1 \le i \le 5$. Then by Propositions \ref{irprop}, \ref{i2x2prop} and Lemma \ref{i2x2sn}, 
\begin{align*}
i_{2 \times 2}(M) &< \sum_{r_1+2r_2+2r_3+2r_4+4r_5=t} I(r_1,\dots, r_5) i_{2 \times 2}(PCl_m(q))^{r_1} \\ &\hspace{5cm}  \times i_2(PCl_m(q))^{r_2+r_3+r_4} |Cl_m(q)|^{r_2+r_3+r_4+3r_5}q^{ct} \\[0.1em]
&< |S_n|^\frac{3}{4} |Cl_m(q)|^{\frac{3}{4}(r_1+2r_2+2r_3+2r_4+4r_5)}q^{c'n} \\
&< |M|^\frac{3}{4}q^{c''n}
\end{align*}
\noindent for some absolute constant $c''$.
\par Now suppose $M$ is of type $GL_\frac{n}{2}(q^\delta).2$. We have $i_{2 \times 2}(M)<|M|^{\frac{3}{4}+\frac{c}{n}}$ by Proposition \ref{i2x2prop}, and the result follows.
\end{proof}

\begin{lemma}
Proposition \ref{i2x2Mprop1}  holds for $M \in \mathscr{C}_i, 3 \le i \le 8$, and for $M \in \mathscr{S}$.
\end{lemma}

\begin{proof}
\par If $M \in \mathscr{C}_3, \mathscr{C}_5$ or $\mathscr{C}_8$, then the result follows from Proposition \ref{i2x2prop}.
\par Suppose $M \in \mathscr{C}_4$ with $M \le Cl^1_d(q) \times Cl^2_e(q)$. Then
\begin{align*}
i_{2 \times 2}(M) <(i_{2 \times 2}(Cl^1_d(q))+2i_2(Cl^1_d(q))+1)(i_{2 \times 2}(Cl^2_e(q))+2i_2(Cl^2_e(q))+1),
\end{align*} 
\noindent and using Propositions \ref{irprop} and \ref{i2x2prop} gives
\begin{align*}
i_{2 \times 2}(M) &< |Cl^1_d(q)|^{\frac{3}{4}+\frac{c}{d}} |Cl^2_e(q)|^{\frac{3}{4}+\frac{c'}{e}} \\
&< |M|^{\frac{3}{4}}q^{c''n}
\end{align*}
\noindent for some absolute constant $c''$.
\par For $M \in \mathscr{C}_6$, it can easily be shown that $|M|<q^{cn}$ for some absolute constant $c$, and the result follows.
\par If $M \in \mathscr{C}_7$, then arguing as in Lemma \ref{i2x2class2} we can show $i_{2 \times 2}(M)<|M|^\frac{3}{4}q^{cn}$ for some absolute constant $c$. 
\par Finally suppose $M \in \mathscr{S}$. Then $M$ is an almost simple group acting absolutely irreducibly on $V$, the natural module of $G$. By \cite{Li}, if $\operatorname{soc}(M) \ne A_{n+1}, A_{n+2}$, then $|M|<q^{3n}$. If $\operatorname{soc}(M)=A_{n+1}$ or $A_{n+2}$, then $i_{2 \times 2}(M) \le |M|^\frac{3}{4}q^{cn}$ by Lemma \ref{i2x2sn}. This completes the proof.   \end{proof}

\subsection{Parabolic maximal subgroups}

\par In this subsection we prove Proposition \ref{i2x2Mprop2}. Recall the definitions of $G$, $x \in A$ and $K \le B$. Let $M=P_m$ be a maximal parabolic subgroup of $G$. As in \S~6.2, instead of bounding $i_{2 \times 2}(M)$, we consider the fixed points of $K$ acting on $M^G$ in Lemma \ref{fprk}.
\par Let $T$ be a group acting transitively on a set $\Omega$, and recall the definition of $\operatorname{fix}(t,\Omega)$ for $t \in T$ from \S~6.2. We extend this definition: for a subgroup $S \le T$, define 
\begin{align*}
\operatorname{fix}(S, \Omega) = |\{ \omega \in \Omega: \omega t=\omega \ \forall t \in S \}|,
\end{align*}
\noindent so that $\operatorname{fix}(t,\Omega)=\operatorname{fix}(\langle t \rangle, \Omega)$. For $\omega \in \Omega$, let $R=T_\omega$. An elementary counting argument (similar to the argument used to show (\ref{fpreqn})) shows
\begin{align} \label{fpreqn2}
\frac{| \{ S^t: t \in T, S^t \le R \} |}{|S^T|} =  \frac{\operatorname{fix}(S, \Omega) }{| \Omega |}.
\end{align}
\par Before we begin the proof of Proposition \ref{i2x2Mprop2} , we require a technical result. To state the result, let $L_i$ be the Jordan block of size $i$ with diagonal entries equal to 0.

\begin{lemma} \label{nilpotent}
Let $\Lambda(l_1,l_2,l,m)$ denote the number of nilpotent block matrices of the form
\begin{align*}
\lambda=\begin{pmatrix}
\lambda_{11} & \lambda_{12} & \lambda_{13} \\
 & \lambda_{22} & \lambda_{23} \\
 &  & \lambda_{11} \\
\end{pmatrix} \in M_{m,m}(q)
\end{align*}
\noindent where $\lambda^2=0, \lambda_{11} \in M_{l,l}(q), \lambda_{22} \in M_{m-2l, m-2l}(q)$, and $\lambda_{11}, \lambda_{22}$ have Jordan normal forms $L_2^{l_1} \oplus L_1^{l-2l_1}, L_2^{l_2} \oplus L_1^{m-2l-2l_2}$ respectively. Then 
\begin{align*}
\Lambda(l_1,l_2,l,m) < cq^{4l_1(l-l_1)+2l_2(m-2l-l_2)+2(l-l_1)(m-2l-l_2)+2l_1l_2}
\end{align*} 
\noindent for some absolute constant $c$. In particular, for some absolute constant $c'$,
\begin{align*}
\sum_{\substack{0 \le l_1 \le \frac{l}{2}, \\ 0 \le l_2 \le \frac{m-2l}{2}}} \Lambda(l_1,l_2,l,m) < c'q^{\frac{1}{2}(m^2-2lm+2l^2)}.
\end{align*}
\end{lemma}

\begin{proof}
We count the number of possible entries of $\lambda$. Since $\lambda^2=0$, we have the following conditions:
\begin{enumerate}[label=(\roman*)]
\item $\lambda_{ii}^2=0, i=1,2;$ 
\item $\lambda_{11}\lambda_{12}+\lambda_{12}\lambda_{22}=0; $ 
\item $\lambda_{22}\lambda_{23}+\lambda_{23}\lambda_{11}=0;$ 
\item $\lambda_{11}\lambda_{13}+\lambda_{12}\lambda_{23}+\lambda_{13}\lambda_{11}=0.$
\end{enumerate}
\par We consider condition (i). Suppose $\lambda_{11}$ has Jordan normal form $L_{2}^{l_1} \oplus L_1^{l-2l_1}$. Then $\lambda_{11}+1=u$ for a unipotent element $u \in GL_l(q)$ with Jordan normal form $J_{2}^{l_1} \oplus J_{1}^{l-2l_1}$, using the notation of (\ref{unipotentelt}). Observe that $C_{GL_l(q)}(\lambda_{11})=C_{GL_l(q)}(u)$. By \cite[Theorem~7.1]{LiSe}, 
\begin{align*}
|C_{GL_l(q)}(u)| > cq^{l^2-2l_1(l-l_1)}
\end{align*}
\noindent for some absolute constant $c$, and so
\begin{align*} 
|\lambda_{11}^{GL_l(q)}| < c'q^{2l_1(l-l_1)}.
\end{align*}
\noindent Similarly, if $\lambda_{22}$ has Jordan normal form $L_2^{l_2}\oplus L_1^{m-2l-2l_2}$, then
\begin{align*} 
|\lambda_{22}^{GL_{m-2l}(q)}| < c''q^{2l_2(m-2l-l_2)}.
\end{align*}
\par Next consider condition (ii). For $g \in GL_l(q), h \in GL_{m-2l}(q)$, let $\phi(g,h)$ be the map
\begin{align*}
\phi(g,h):M_{l,m-2l}(q) &\longrightarrow M_{l,m-2l}(q) \\
\mu \mapsto g^{-1} \mu h.
\end{align*}
\noindent We can conjugate by $\lambda_{11}, \lambda_{22}$ to their respective Jordan normal forms, since condition (ii) holds for $(\lambda_{11}, \lambda_{12}, \lambda_{22})$ if and only if it holds for $(\lambda_{11}^g, \lambda_{12}\phi(g,h), \lambda_{22}^h)$ for $g \in GL_{l}(q), h \in GL_{m-2l}(q)$. We calculate that for given $\lambda_{11}, \lambda_{22}$ as above, the number of $\lambda_{12}$ satisfying condition (ii) is $q^{(l-l_1)(k-2l-l_2)+l_1l_2}$.

\par We now consider the number of $\lambda_{23}$ satisfying condition (iii) given $\lambda_{11}, \lambda_{22}$ with Jordan normal forms as above. Similar reasoning to the above paragraph shows the number of $\lambda_{23}$ is also $q^{(l-l_1)(k-2l-l_2)+l_1l_2}$.

\par We now consider the number of $\lambda_{13}$ satisfying condition (iv) given $\lambda_{11}, \lambda_{12}, \lambda_{23}$. For given $\lambda_{12}, \lambda_{23}$ satisfying conditions (ii) and (iii), we calculate the number of $\lambda_{13}$ satisfying condition (iv) is $q^{2l_1(l-l_1)}$.
\par Hence
\begin{align*}
\Lambda(l_1,l_2,l,m) <  c'''q^{4l_1(l-l_1)+2l_2(m-2l-l_2)+2(l-l_1)(k-2l-l_2)+2l_1l_2}.
\end{align*}
\noindent It is elementary to show $4l_1(l-l_1)+2l_2(m-2l-l_2)+2(l-l_1)(k-2l-l_2)+2l_1l_2 \le \frac{1}{2}(m^2-2lm+2l^2)$, and the result follows.
\end{proof}

\par Recall that $G$ is a finite simple classical group with natural module of dimension $n$ over $\mathbb{F}_{q^\delta}$. Let $A$ and $B$ be nontrivial finite 2-groups embedded almost-freely into $G$, with $|A|=a, |B|=b$ and $n=k_ii+s_i$ with $2 \le s_i \le 2i+1$ and $k_i$ even for $i=a,b$. Assume $x \in I_2(A)$ and $K=\langle y_1,y_2 \rangle \in I_{2 \times 2}(B)$. We now move towards proving Proposition \ref{i2x2Mprop2} by proving the following result.

\begin{lemma} \label{fprk}
Let $M=P_m$ be a maximal parabolic subgroup of $G$. Then
\begin{align*}
\operatorname{fix}(K,M^G)  < q^{f+cm}
\end{align*}
\noindent where $c=c(B)$ is a constant depending only on $B$, and 
\begin{align*}
 f=\left\{ 
  \begin{array}{l l}
\frac{1}{4}m(n-m) \hspace{0.3cm} &\text{if $G=PSL_n(q)$,} \\[0.2cm]
\frac{1}{44}m(11n-12m) &\text{if $G=PSp_n(q), P\Omega_n(q)$,} \\[0.2cm]
\frac{1}{22}m(11n-12m) &\text{if $G=PSU_n(q)$}.
\end{array}
\right.
\end{align*}
\end{lemma}

\begin{proof}
\par Observe that $\operatorname{fix}(K,M^G)$ is less than or equal to the number of totally singular $m$-spaces $U \subset V$ invariant under $K$. We count the number of these subspaces.
\par For ease of notation, let $k_b=k, s_b=s,$ so $n=kb+s$.

\paragraph{Case $G=PSL_n(q).$} 
First suppose $q$ is odd. Let $E_j^s$ be the $j$-eigenspace of $y_t$ for $j=\pm 1, t=1,2$. Then $V\downarrow K=E_{-1,-1} \oplus E_{-1,1} \oplus E_{1,-1} \oplus E_{1,1}$, where $E_{i,j}=E_i^1 \cap E_j^2$ for $i, j = \pm 1$, and 
\begin{align*}
\dim(E_{i,j}) = \left\{ 
  \begin{array}{l l}
\frac{kb}{4} \hspace{0.3cm} &\text{if $(i,j) \ne (1,1)$,} \\[0.2cm]
\frac{kb}{4}+s &\text{if $i=j=1$.} \\[0.2cm]
\end{array}
\right. 
\end{align*}
\noindent If $U \subset V$ is an $m$-space stabilized by $K$, then $U=U_{-1,-1}\oplus U_{-1,1} \oplus U_{1,-1} \oplus U_{1,1}$, where $U_{i,j} \subset E_{i,j}$ is an $l_{i,j}$-subspace for $i, j = \pm 1$ such that $\sum l_{i,j}=m$. Therefore, by (\ref{parabolicindex}) and use of the Cauchy-Schwarz inequality, we have
\begin{align*}
\operatorname{fix}(K,M^G) &\le \sum_{\sum l_{i,j}=m} p_{l_{-1,-1}}(E_{-1,-1})p_{l_{-1,1}}(E_{-1,1})p_{l_{1,-1}}(E_{1,-1})p_{l_{1,1}}(E_{1,1}) \\
&< \sum_{\sum l_{i,j}=m} cq^{l_{-1,-1}\left( \frac{kb}{4}-l_{-1,-1} \right) + l_{-1,1}\left( \frac{kb}{4}-l_{-1,1} \right) + l_{1,-1}\left( \frac{kb}{4}-l_{1,-1} \right) + l_{1,1}\left( \frac{kb}{4}+s-l_{1,1} \right)} \\
&<q^{\frac{1}{4}m(n-m)+c'm}
\end{align*}
\noindent for some absolute constant $c'$, completing the proof of the linear case for $q$ odd.

\par Now suppose $q$ is even. With respect to some basis $\{ e_i \}$ of $V$ we have (using the notation of (\ref{unipotentelt}))
\begin{align*}
y_1=\begin{pmatrix}
J_{\frac{kb}{2},2} \\
&  I_{s} 
\end{pmatrix}, \hspace{0.5cm} 
y_2=\begin{pmatrix}
J_{\frac{kb}{4},2} \\
& J_{\frac{kb}{4},2} \\
& &  I_{s} 
\end{pmatrix}.
\end{align*}
\noindent Suppose $K$ stabilizes an $m$-subspace $U \subset V$. There exists a basis $\beta=\{ u_i \}_{i=1}^m$ of $U$ with respect to which
\begin{align} \label{xKeven}
[y_1^U]_\beta=\begin{pmatrix}
J_{l,2} \\ 
& I_{m-2l} 
\end{pmatrix}
\end{align}
\noindent for some $0 \le l \le \frac{m}{2}$.
\par Let $u_i=\sum_{j=1}^n \alpha_{ij}e_j$ for $\alpha_{ij} \in \mathbb{F}_q$, and let $\alpha=(\alpha_{ij})$. Then
\begin{align} \label{Keveneqn}
[y_1^U]_\beta^T \alpha =  \alpha y_1^T.
\end{align}
 \noindent Therefore, we can write
\begin{align} \label{Kalphaodd}
\alpha = \begin{pmatrix}
A_1 & A_2 & A_3& A_4 & B_1\\ & &  A_1 & A_2\\ & & A_5 & A_6 & B_2
\end{pmatrix}
\end{align}
\noindent where 
\begin{align*}
A_1, \dots, A_4 &\in M_{l,\frac{kb}{4}}(q), \\
A_5, A_6 &\in M_{m-2l, \frac{kb}{4}}(q), \\
B_1 &\in M_{l, s}(q), \\ 
B_2 &\in M_{m-2l, s}(q).
\end{align*}
\par Observe that $y_2$ stabilizes $U$ if and only if $y_2+1$ stabilizes $U$, and this occurs if and only if there exists $\lambda \in M_{m,m}(q)$ such that
\begin{align} \label{Keveneqn2}
\lambda \alpha = \alpha (y_2+1)^T,
\end{align}
\noindent so $[(y_2+1)^U]_\beta=\lambda^T$ and $\lambda^2=0$. Write $\lambda$ as
\begin{align}
\begin{pmatrix}
\lambda_{11} & \lambda_{12} & \lambda_{13} \\
\lambda_{21} & \lambda_{22} & \lambda_{23} \\
\lambda_{31} & \lambda_{32} & \lambda_{33} \\
\end{pmatrix}
\end{align}
\noindent where 
\begin{align*}
\lambda_{11}, \lambda_{22} &\in M_{l,l}(q), \\
 \lambda_{3,3} &\in M_{m-2l, m-2l}(q).
\end{align*}
\noindent Then we have
\begin{align}
\lambda_{21} \begin{pmatrix} A_1 & A_2 \end{pmatrix} &= 0 \label{lambdaeqn1},\\
\lambda_{31}  \begin{pmatrix} A_1 & A_2 \end{pmatrix} &= 0 \label{lambdaeqn2},
\end{align}
\noindent and so $\lambda_{21}, \lambda_{31}=0$ since $\alpha$ is right-invertible (the rows of $\alpha$ are linearly independent by assumption, as they form a basis of $U$). Also, we find
\begin{align} \label{lambdaeqn3}
\begin{pmatrix} \lambda_{11}+\lambda_{22} & \lambda_{23} \end{pmatrix} \begin{pmatrix} A_1 & A_2 \\ A_5 & A_6 \end{pmatrix} =0,
\end{align}
\noindent which implies
\begin{align} \label{lambdaeqn4}
\lambda=\begin{pmatrix}
\lambda_{11} & \lambda_{12} & \lambda_{13} \\
 & \lambda_{11} &  \\
 & \lambda_{32} & \lambda_{33} \\
\end{pmatrix}.
\end{align}
\par  From (\ref{Keveneqn}) and (\ref{Keveneqn2}) it is clear that $\alpha$ (and hence $U$) is determined by $\lambda, A_2, A_4, A_6, B_1, B_2$. Let $\Delta(l)$ be the number of tuples $(\lambda, A_2, A_4, A_6, B_1, B_2)$ with $\lambda$ as in (\ref{lambdaeqn4}), and let $\Delta(l,l_1,l_2)$ be the number of tuples such that $\lambda_{11}, \lambda_{33}$ have Jordan normal forms $L_2^{l_1} \oplus L_1^{l-2l_1}, L_2^{l_2} \oplus L_1^{m-2l-2l_2}$ respectively ($0 \le l_1 \le \frac{l}{2}, 0 \le l_2 \le \frac{m-2l}{2}$), so that
\begin{align*}
\Delta(l) =  \sum_{\substack{0 \le l_1 \le \frac{l}{2} \\ 0 \le l_2 \le \frac{m-2l}{2}}} \Delta(l,l_1,l_2).
\end{align*} 
\noindent Then $\Delta(l)$ is the total number of bases $\beta$ of an $m$-subspace $U$ such that (\ref{Keveneqn}) holds and (\ref{Keveneqn2}) holds for some $\lambda$. Observe that $\lambda$ is conjugate by
\begin{align*}
\begin{pmatrix}
I_l \\ & & I_l \\
& I_{m-2l}
\end{pmatrix}
\end{align*}
\noindent to a matrix considered in Lemma \ref{nilpotent}, and so 
\begin{align*}
\Delta(l) =  \sum_{\substack{0 \le l_1 \le \frac{l}{2} \\ 0 \le l_2 \le \frac{m-2l}{2}}}  \Delta(l,l_1,l_2)< \sum_{\substack{0 \le l_1 \le \frac{l}{2} \\ 0 \le l_2 \le \frac{m-2l}{2}}} \Lambda(l_1,l_2,l,m)q^{\frac{mn}{4}+(\frac{3m}{4}-l)s}<cq^{\frac{mn}{4}+\frac{1}{2}(m^2-2lm+2l^2)+(\frac{3m}{4}-l)s}.
\end{align*}
\par Observe that $C_{GL_m(q)}([y_1^U]_\beta)$ acts regularly on the set of bases $\beta'$ of $U$ such that $[y_1^U]_\beta=[y_1^U]_{\beta'}$, and so the number of such bases is $|C_{GL_m(q)}([y_1^U]_\beta)|$. By (\ref{unipotentcent}),
\begin{equation*}
|C_{GL_m(q)}([y_1^U]_\beta)| > c'q^{l^2+(m-l)^2}
\end{equation*}
\noindent for some absolute constant $c'$. Hence
\begin{align*}
\operatorname{fix}(K,M^G) &< \sum_{0 \le l \le \frac{m}{2}}  \frac{\Delta(l)}{|C_{GL_m(q)}([y_1^U]_\beta)|} \\ &<\sum_{0 \le l \le \frac{m}{2}} c''q^{\frac{mn}{4}+\frac{1}{2}(m^2-2lm+2l^2)+(\frac{3m}{4}-l)s-l^2-(m-l)^2} \\ &<q^{\frac{1}{4}m(n-m)+c'''m},
\end{align*}
\noindent completing the proof in the linear case.

\paragraph{Case $G=PSp_n(q).$} 
First suppose $q$ is odd. Let $E_j^t$ and $E_{i,j}$ be as in the linear case for $i, j=\pm 1, t=1,2$. Then $V\downarrow K=E_{-1,-1} \perp E_{-1,1} \perp E_{1,-1} \perp E_{1,1}$. If $U \subset V$ is an $m$-space stabilized by $K$, then $U=U_{-1,-1} \perp U_{-1,1} \perp U_{1,-1} \perp U_{1,1}$, where $U_{i,j} \subset E_{i,j}$ is a totally singular $l_{i,j}$-subspace for $i, j = \pm 1$ such that $\sum l_{i,j}=m$. Therefore, by (\ref{parabolicindex}) and using the Cauchy-Schwarz inequality, we have
\begin{align*}
 \hspace{-0.3cm}\operatorname{fix}(K,M^G) &\le \sum_{\sum l_{i,j}=m} p_{l_{-1,-1}}(E_{-1,-1})p_{l_{-1,1}}(E_{-1,1})p_{l_{1,-1}}(E_{1,-1})p_{l_{1,1}}(E_{1,1}) \\
&< \sum_{\sum l_{i,j}=m} cq^{l_{-1,-1}\left( \frac{kb}{4}-\frac{3}{2}l_{-1,-1} +\frac{1}{2}\right) + l_{-1,1}\left( \frac{kb}{4}-\frac{3}{2}l_{-1,1} +\frac{1}{2} \right) + l_{1,-1}\left( \frac{kb}{4}-\frac{3}{2}l_{1,-1} +\frac{1}{2} \right) + l_{1,1}\left( \frac{kb}{4}+s-\frac{3}{2}l_{1,1} +\frac{1}{2} \right)} \\
&<q^{\frac{1}{8}m(2n-3m)+c'm}
\end{align*}
\noindent for some absolute constant $c'$.

\par Now suppose $q$ is even. With the almost-free embedding, there exists a basis $\{ e_i \}$ of $V$ with Gram matrix 
\begin{equation*}
\begin{pmatrix} & I_\frac{kb}{2} \\ I_\frac{kb}{2} \\ & & D \end{pmatrix}
\end{equation*}
\noindent (where $D$ is a Gram matrix of an $s$-dimensional symplectic space) with respect to which
\begin{align} \label{xyKeven2}
y_1=\begin{pmatrix}
J_{\frac{kb}{4},2} \\
& J_{\frac{kb}{4},2}^T \\
& &  I_{s} 
\end{pmatrix}, \hspace{0.5cm} 
y_2=\begin{pmatrix}
J_{\frac{kb}{8},2} \\
& J_{\frac{kb}{8},2} \\
& & J_{\frac{kb}{8},2}^T \\
& & & J_{\frac{kb}{8},2}^T \\
& &  & & I_{s} 
\end{pmatrix}.
\end{align}
\noindent Suppose $K$ stabilizes a totally singular $m$-subspace $U \subset V$. There exists a basis $\beta=\{ u_i \}_{i=1}^m$ of $U$ such that $[y_1^U]_\beta$ is as in (\ref{xKeven}).
\par Let $u_i=\sum_{j=1}^n \alpha_{ij}e_j$ for $\alpha_{ij} \in \mathbb{F}_q$, and let $\alpha=(\alpha_{ij})$. Then
\begin{align} \label{Keveneqn3}
[y_1^U]_\beta^T \alpha = \alpha y_1^T,
\end{align} 
\noindent and we can write
\begin{align} \label{Kevenalpha2}
\alpha = \begin{pmatrix}
A_1 & A_2 & A_3& A_4 & B_1 & B_2 & B_3 & B_4 & C_1\\ & &  A_1 & A_2 & B_3 & B_4 \\ & & A_5 & A_6 & B_3 & B_4 & & & C_2
\end{pmatrix}
\end{align}
\noindent where 
\begin{align*}
A_1, \dots, A_4, B_1, \dots, B_4 &\in M_{l,\frac{kb}{8}}(q), \\
A_5, A_6, B_5, B_6 &\in M_{m-2l, \frac{kb}{8}}(q), \\
C_1 &\in M_{l, s}(q), \\
C_2 &\in M_{m-2l, s}(q).
\end{align*}
\par Observe that if $y_2$ stabilizes $U$, then $y_2+1$ stabilizes $U$, and so there exists $\lambda \in M_{m,m}(q)$ such that 
\begin{align} \label{Keveneqn4}
\lambda \alpha = \alpha (y_2+1)^T,
\end{align}
\noindent so $[(y_2+1)^U]_\beta=\lambda^T$ and $\lambda^2=0$. As in the linear case, we can write $\lambda$ as in (\ref{lambdaeqn4}).

\par For rows $u_i, u_j$ of $\alpha$, consider $(u_i, u_j)$, where $( \cdot  \hspace{0.05cm} , \cdot)$ denotes the symplectic form on $V$. Since the rows of $\alpha$ generate a totally singular subspace, then by considering $(u_i, u_j)$ for $1 \le i \le l<j \le 2l$, we see that 
\begin{align*}
A_1B_3^T+A_2B_4^T
\end{align*} 
\noindent is symmetric. Since $A_1=\lambda_{11}A_2$ and $B_4=\lambda_{11}B_3$ from (\ref{Keveneqn4}), this implies that
\begin{align} \label{symeqn2}
\lambda_{11} \begin{pmatrix} A_2B_3^T+B_3A_2^T \end{pmatrix}
\end{align}
\noindent is symmetric. For $g \in GL_l(q)$, observe that the tuple $(\lambda_{11},A_2,B_3)$ satisfies (\ref{symeqn2}) if and only if $(\lambda_{11}^g,g^{-1}A_2,g^{-1}B_3)$ satisfies (\ref{symeqn2}). Suppose $\lambda_{11}$ is equal to its Jordan normal form $L_2^{l_1} \oplus L_1^{l-2l_1}$ ($0 \le l_1 \le \frac{l}{2}$). Writing
\begin{align} \label{A2B3eqn}
A_2=\begin{pmatrix}
E \\ F
\end{pmatrix}, \hspace{0.5cm} 
B_3= \begin{pmatrix}
P \\ Q
\end{pmatrix}
\end{align}
\noindent where 
\begin{align*}
E, P &\in M_{2l_1, \frac{kb}{8}}(q), \\ 
F,Q &\in M_{l-2l_1,\frac{kb}{8}}(q),
\end{align*}
\noindent if $A_2, B_3$ satisfy (\ref{symeqn2}) with $\lambda_{11}=L_2^{l_1} \oplus L_1^{l-2l_1}$ then 
\begin{align} \label{symeqn3}
\begin{pmatrix} EP^T+PE^T \end{pmatrix}_{2i, 2j}=0
\end{align}
\noindent for $1 \le i, j \le l_1$. Writing $E_i$ for the $i$th row of $E$ and similarly for $P$, (\ref{symeqn3}) is equivalent to 
\begin{align*}
E_{2i} \cdot P_{2j} + P_{2i} \cdot E_{2j}=0,
\end{align*}
\noindent and this in turn is equivalent to the even-indexed rows of $\begin{pmatrix} E & P \end{pmatrix}$ generating a totally singular subspace of a $\frac{kb}{4}$-dimensional symplectic space. Therefore, given $\lambda$ with $\lambda_{11}$ having Jordan normal form $L_2^{l_1} \oplus L_1^{l-2l_1}$, by (\ref{parabolicindex}) the number of possible even-indexed rows of $E$ and $P$ is less than $cq^{l_1(\frac{kb}{4}-\frac{l_1}{2}+\frac{1}{2})}$, and so the number of possible $A_2, B_3$ is less than
\begin{align} \label{A2B3eqn}
c'q^{\frac{lkb}{4}-\frac{l_1^2}{2}+\frac{l_1}{2}}.
\end{align} 
\par  From (\ref{Keveneqn3}) and (\ref{Keveneqn4}), $\alpha$ is determined by $\lambda, A_2, A_4, A_6, B_1, B_3, B_5, C_1, C_2$. Let $\Delta(l)$ be the number of tuples $(\lambda,A_2, A_4, A_6, B_1, B_3, B_5, C_1, C_2)$ with $\lambda$ as in (\ref{lambdaeqn4}) such that $\alpha$ obtained from (\ref{Keveneqn3}) and (\ref{Keveneqn4}) has linearly independent rows and generates a totally singular subspace of $V$. Also, let $\Delta(l,l_1,l_2)$ be the number of such tuples with $\lambda_{11}, \lambda_{33}$ having Jordan normal form $L_2^{l_1} \oplus L_1^{l-2l_1},  L_2^{l_2} \oplus L_1^{m-2l-2l_2}$ respectively ($0 \le l_1 \le \frac{l}{2}, 0 \le l_2 \le \frac{m-2l}{2}$), so that
\begin{align*}
\Delta(l) = \sum_{\substack{0 \le l_1 \le \frac{l}{2} \nonumber \\ 0 \le l_2 \le \frac{m-2l}{2}}} \Delta(l,l_1,l_2).
\end{align*}
\noindent Then $\Delta(l)$ is the total number of bases $\beta$ of a totally singular $m$-subspace $U$ such that (\ref{Keveneqn3}) holds and (\ref{Keveneqn4}) holds for some $\lambda$. Observe that, as in the linear case, $\lambda$ is conjugate to a matix considered in Lemma \ref{nilpotent}, and so by this result and (\ref{A2B3eqn}) we have
\begin{align*}
\Delta(l) &=  \sum_{\substack{0 \le l_1 \le \frac{l}{2} \nonumber \\ 0 \le l_2 \le \frac{m-2l}{2}}} \Delta(l,l_1,l_2)  \\ &<\sum_{\substack{0 \le l_1 \le \frac{l}{2} \nonumber \\ 0 \le l_2 \le \frac{m-2l}{2}}}  \Lambda(l_1,l_2,l,m)q^{\frac{mn}{4}+(\frac{3m}{4}-l)s-\frac{l_1^2}{2}+\frac{l_1}{2}} \\
&<\sum_{\substack{0 \le l_1 \le \frac{l}{2} \nonumber \\ 0 \le l_2 \le \frac{m-2l}{2}}}  cq^{\frac{mn}{4}+\frac{1}{2}l_1(8l-9l_1+1)+2l_2(m-2l-l_2)+2(l-l_1)(m-2l-l_2)+2l_1l_2+(\frac{3m}{4}-l)s}
\end{align*}
\noindent for an absolute constant $c$. It is elementary to show that $\frac{1}{2}l_1(8l-9l_1+1)+2l_2(m-2l-l_2)+2(l-l_1)(m-2l-l_2)+2l_1l_2 \le \frac{m^2}{2}-lm+\frac{9l^2}{10}+c'm$ for a constant $c'=c'(B)$ depending only on $B$, and so
\begin{align*}
\Delta(l)< q^{\frac{nm}{4}+\frac{m^2}{2}-lm+\frac{9l^2}{10}+c'm}.
\end{align*}
\noindent As in the linear case, the number of bases $\beta'$ of $U$ such that $[y_1^U]_\beta=[y_1^U]_{\beta'}$ is $|C_{GL_m(q)}([y_1^U]_\beta)|$. Hence, by (\ref{unipotentcent}),
\begin{align} 
\operatorname{fix}(K,M^G) &<\sum_{0 \le l \le \frac{m}{2}} \frac{\Delta(l)}{|C_{GL_m(q)}([y_1^U]_\beta)|} \nonumber \\ &<\sum_{0 \le l \le \frac{m}{2}} q^{\frac{9l^2}{10}+\frac{m^2}{2}-lm+\frac{mn}{4}+c''m-l^2-(m-l)^2} \nonumber \\ &<q^{\frac{1}{44}m(11n-12m)+c''''m}, \label{symplecticKeqn}
\end{align}
\noindent completing the proof in the symplectic case.

\paragraph{Case $G=P\Omega^\epsilon_n(q).$} 
In the case where $q$ is odd, the proof is similar to the symplectic case.

\par Now suppose $q$ is even. There exists a basis $\{ e_i \}$ of $V$ with quadratic form $Q$ on $V$ defined as
\begin{align*}
Q \left( \sum_{i=1}^n \lambda_i e_i \right) = \sum_{i=1}^\frac{kb}{2} \lambda_i \lambda_{\frac{kb}{2}+i}+Q' \left( \sum_{i=kb+1}^n \lambda_i e_i \right)
\end{align*}
\noindent where $Q'$ is a quadratic form of type $\epsilon$ on the $s$-space $\langle e_{kb+1}, \dots, e_n \rangle \subset V$. With respect to this basis, we have $y_1$ and $y_2$ as in (\ref{xyKeven2}), and Gram matrix
\begin{equation*}
\begin{pmatrix} & I_\frac{kb}{2} \\ I_\frac{kb}{2} \\ & & D \end{pmatrix}
\end{equation*}
\noindent (where $D$ is the Gram matrix of the $s$-dimensional space with quadratic form $Q'$).
\par Any totally singular $m$-subspace of $V$ with respect to $Q$ is totally singular with respect to the symplectic form determined by $Q$. Therefore $\operatorname{fix}(K, M^G)$ can be bounded as in (\ref{symplecticKeqn}), completing the proof in the orthogonal case.

\paragraph{Case $G=PSU_n(q).$} 
The proof in this case is similar to the symplectic case.
\end{proof}

\begin{proof}[Proof of Proposition \ref{i2x2Mprop2} ]
\par By considering the action of $G$ on $M^G=P_m^G$ ($1 \le m \le \frac{n}{2}$) and (\ref{fpreqn2}), we have
\begin{align*}
\sum_{M^h \in M^G}  \frac{|x^G \cap M^h|}{|x^G|} \frac{|\{ K^g: g \in G, K^g \le M^h \} |}{|K^G|} &= |G:M|  \frac{|x^G \cap M|}{|x^G|} \frac{|\{ K^g: g \in G, K^g \le M \} |}{|K^G|} \\
&= \operatorname{fpr}(x,M^G) \operatorname{fix}(K,M^G).
\end{align*}
\noindent The number of $G$-classes of parabolic subgroups $P_m$ is at most 2 by \cite[\S~4.1]{KlLi}, and the result now follows by Lemmas \ref{fprx} and \ref{fprk}. 
\end{proof}

 \vspace{-0.4cm}
\subsection{Proof of Theorem \ref{bigthm2}}

We now prove Theorem \ref{bigthm2}, using Propositions \ref{i2x2Mprop1} and \ref{i2x2Mprop2}. This completes the proof of the main theorem by \S~3.

\begin{proof}[Proof of Theorem \ref{bigthm2}]
\par Let $G$ be a finite simple classical group with natural module of dimension $n$ over $\mathbb{F}_{q^\delta}$. Let $A$ and $B$ be nontrivial finite 2-groups, and assume $x \in I_2(A), K \in I_{2 \times 2}(B)$, and $n \ge \operatorname{max}\{2|A|+2,2|B|+2\}$. Embed $A$ and $B$ almost-freely into $G$. 
\par We first consider the contribution to the summation in Theorem \ref{bigthm2} by non-parabolic subgroups $M$. Let $\mathscr{M}_{\text{np}}$ be the set of non-parabolic maximal subgroups of $G$, and let $M \in \mathscr{M}_{\text{np}}$. By Proposition \ref{i2x2Mprop1}, $i_{2 \times 2}(M)<|M|^{\frac{3}{4}}q^{cn}$ for some absolute constant $c$. By Propositions \ref{conjclassprop1}, \ref{irprop} and \ref{conjclassprop2}, for an absolute constant $c'$ we have $|x^G|>i_2(G)q^{c'n}, |K^G|>|G|^{\frac{3}{4}+\frac{c'}{n}}$. Therefore, using Theorem \ref{i2ratiothm} we have
\begin{align*}
 \sum_{M \in \mathscr{M}_{\text{np}}} \frac{|x^G \cap M|}{|x^G|} \frac{|\{ K^g: g \in G, K^g \le M \} |}{|K^G|}  &< \sum_{M \in \mathscr{M}_{\text{np}}} \frac{i_2(M)}{i_2(G)}\frac{|M|^\frac{3}{4}}{|G|^\frac{3}{4}}q^{c''n} \\
&< \sum_{M \in \mathscr{M}_{\text{np}}}  |G:M|^{-\frac{5}{4}+\frac{c'''}{n}} \\
&< \zeta_G(\frac{5}{4}-\frac{c'''}{n}),
\end{align*}
\noindent and for $n$ sufficiently large we have $\zeta_G(\frac{5}{4}-\frac{c'''}{n})<\frac{1}{2}$ by Theorem \ref{zetathm}.
\par We now consider the contribution to the summation by parabolic subgroups. Denote by $\mathscr{M}_{\text{p}}$ the set of maximal parabolic subgroups of $G$. By Proposition \ref{i2x2Mprop2} we have
\begin{align*}
 \sum_{M \in \mathscr{M}_{\text{p}}} \frac{|x^G \cap M|}{|x^G|} \frac{|\{ K^g: g \in G, K^g \le M \} |}{|K^G|} < q^{-\frac{\delta n}{4}+c''''}
\end{align*}
\noindent for some absolute constant $c''''$. Clearly, for sufficiently large $n$, this contribution is less than $\frac{1}{2}$. This completes the proof.
\end{proof}

\end{document}